\title{\bf One-dimensional wave equation with set-valued boundary damping: well-posedness, asymptotic stability, and decay rates}

\author{Yacine Chitour$^{1}$, Swann Marx$^{2}$, Guilherme Mazanti$^{1, 3}$}

\documentclass[12pt]{article}

\usepackage{amsfonts,latexsym,amsmath,amssymb,amsthm}
\usepackage[mathscr]{eucal}
\usepackage{graphicx,color}
\usepackage{comment}
\usepackage[hidelinks]{hyperref}
\usepackage{epstopdf}
\usepackage[margin = 1in]{geometry}
\usepackage[inline]{enumitem}
\usepackage{mathtools}
\usepackage{cite}
\usepackage{xcolor}
\usepackage{cleveref}
\usepackage{tikz}
\usepackage{array}
\usepackage{calc}
\usepackage[normalem]{ulem}
\usetikzlibrary{arrows, arrows.meta}

\theoremstyle{plain}

\newtheorem{theorem}{Theorem}
\newtheorem{lemma}[theorem]{Lemma}
\newtheorem{proposition}[theorem]{Proposition}
\newtheorem{corollary}[theorem]{Corollary}
\theoremstyle{definition}
\newtheorem{definition}[theorem]{Definition}

\newtheorem{hypotheses}[theorem]{Hypotheses}
\newtheorem{remarkx}[theorem]{Remark}

\newtheorem{examplex}[theorem]{Example}
\newenvironment{remark}
  {\pushQED{\qed}\remarkx}
  {\popQED\endremarkx}
\newenvironment{example}
  {\pushQED{\qed}\examplex}
  {\popQED\endexamplex}

\DeclarePairedDelimiter{\norm}{\lVert}{\rVert}
\DeclarePairedDelimiter{\abs}{\lvert}{\rvert}
\DeclarePairedDelimiter{\floor}{\lfloor}{\rfloor}
\DeclarePairedDelimiter{\ceil}{\lceil}{\rceil}

\DeclareMathOperator{\loc}{loc}
\DeclareMathOperator{\sign}{sgn}
\DeclareMathOperator{\diff}{d\!}
\DeclareMathOperator{\id}{Id}
\DeclareMathOperator{\artanh}{artanh}
\DeclareMathOperator*{\esssup}{ess\,sup}
\DeclareMathOperator*{\essinf}{ess\,inf}
\DeclareMathOperator{\Seq}{Seq}
\newcommand{\suchthat}{\ifnum\currentgrouptype=16 \mathrel{}\middle|\mathrel{}\else\mid\fi}

\def\R{{\mathbb R}}

\makeatletter
\newcommand{\customlabel}[2]{%
   \protected@write \@auxout {}{\string \newlabel {#1}{{#2}{\thepage}{#2}{#1}{}} }%
   \hypertarget{#1}{#2}
}
\makeatother

\begin{document}

\maketitle

\newcommand\blfootnote[1]{%
  \begingroup
  \renewcommand\thefootnote{}\footnote{#1}%
  \addtocounter{footnote}{-1}%
  \endgroup
}


\footnotetext[1]{Universit\'e Paris-Saclay, CNRS, CentraleSup\'elec, Laboratoire des signaux et syst\`emes, 91190, Gif-sur-Yvette, France.}
\footnotetext[2]{LS2N, \'Ecole Centrale de Nantes \& CNRS UMR 6004, F-44000 Nantes, France.}
\footnotetext[3]{Inria, DISCO Team, Saclay--\^Ile-de-France Research Center, France.}

\setlist[enumerate]{label={(\alph*)}}
\setlist[enumerate, 2]{label={(\alph{enumi}-\roman*)}}
\newlist{listhypo}{enumerate}{1}
\setlist[listhypo]{label={\textup{(H\arabic*)}}, ref={\textup{(H\arabic*)}}}

\abstract{This paper is concerned with the analysis of a one dimensional wave equation $z_{tt}-z_{xx}=0$ on $[0,1]$ with a Dirichlet condition at $x=0$ and a damping acting at $x=1$ which takes the form $(z_t(t,1),-z_x(t,1))\in\Sigma$ for every $t\geq 0$, where $\Sigma$ is a given subset of $\mathbb R^2$. The study is performed within an $L^p$ functional framework, $p\in [1, +\infty]$. We aim at determining conditions on $\Sigma$ ensuring existence and uniqueness of solutions of that wave equation as well as strong stability and uniform global asymptotic stability of its solutions. In the latter case, we also study the decay rates of the solutions and their optimality. We first establish a one-to-one correspondence between the solutions of that wave equation and the iterated sequences of a discrete-time dynamical system in terms of which we investigate the above mentioned issues. This enables us to provide a simple necessary and sufficient condition on $\Sigma$ ensuring existence and uniqueness of solutions of the wave equation as well as an efficient strategy for determining optimal decay rates when $\Sigma$ verifies a generalized sector condition. As an application, we solve two conjectures stated in the literature, the first one seeking a specific optimal decay rate and the second one associated with a saturation type of damping. In case the boundary damping is subject to perturbations, we derive sharp results regarding asymptotic perturbation rejection and input-to-state issues.}

\tableofcontents

\section{Introduction}

In this paper, we focus on the following wave equation
\begin{equation}\label{eq:wave}
\left\{
\begin{aligned}
& z_{tt}(t, x) = z_{xx}(t, x), & \qquad & (t, x) \in \mathbb R_+ \times [0, 1], \\
& z(t, 0) = 0, & & t \in \mathbb R_+, \\
& (z_t(t, 1), -z_x(t, 1)) \in \Sigma, & & t \in \mathbb R_+, \\
& z(0, x) = z_0(x), & & x \in [0, 1], \\
& z_t(0, x) = z_1(x), & & x \in [0, 1],
\end{aligned}
\right.
\end{equation}
where $\Sigma \subset \mathbb R^2$. The immense majority of works on that subject (see for instance \cite{Alabau2012Recent,V-Martinez2000} for an overview of the subject) assumes that $\Sigma$
is the graph of a function $\sigma:\mathbb R\to \mathbb R$ and hence the corresponding condition on $z_x(t,1)$ and  $z_t(t,1)$ reduces to 
\begin{equation}\label{eq:feedback}
z_x(t,1) = -\sigma(z_t(t,1)),\qquad \forall t\geq 0,
\end{equation}
which can be interpreted as a feedback law prescribing $z_x(t,1)$
in terms on $z_t(t,1)$ at the boundary $x=1$ for every $t\geq 0$. Note that considering $\Sigma$ more general that the mere graph of a function is a possible alternative to model the fact that the function $\sigma$ is subject to uncertainties or discontinuities, as in the case, for instance, where $\Sigma$ is equal to the graph of the sign set-valued map $\sign:\mathbb R \rightrightarrows \mathbb R$ defined by $\sign(s)=\{s/\abs{s}\}$ for nonzero $s$ and $\sign(0)=[-1,1]$, cf.\ \cite{Xu2019Saturated}.

The use of a set $\Sigma$ in \eqref{eq:wave} can also model switching boundary conditions, such as those considered in \cite{Gugat2010Stars, Chitour2016Stability, Amin2012Exponential, Lamare2015Switching}. In this setting, the boundary condition is usually written as
\begin{equation}
\label{eq:boundary-switch}
z_x(t, 1) = -\sigma_{a(t)}(z_t(t, 1)),
\end{equation}
where $a(\cdot)$ is piecewise constant and takes values in a given (possibly infinite) index set $\mathcal I$ and $\sigma_i: \mathbb R \to \mathbb R$ for $i \in \mathcal I$. Defining $\Sigma$ as the set of pairs $(x, y) \in \mathbb R^2$ such that $y \in \sigma_i(x)$ for some $i \in \mathcal I$, one obtains that any solution of the wave equation with the boundary condition \eqref{eq:boundary-switch} is a solution of \eqref{eq:wave}. This construction is the analogue of the classical transformation of finite-dimensional switched systems into differential inclusions (see, e.g., \cite{Filippov1988Differential, Aubin1984Differential, Liberzon2003Switching}).

\subsection{Existing results}

The standard issues addressed for solutions of \eqref{eq:wave} 
(either with \eqref{eq:feedback} or $\Sigma$ equal to the graph of $\sign$) may be divided into three main questions: find conditions on $\sigma$ so that $(Q1)$ for every initial condition, there exists a global and possibly unique solution of \eqref{eq:wave}, $(Q2)$ in case solutions of \eqref{eq:wave} tend to zero as the time $t$ tends to infinity, one can characterize their decay rates, and $(Q3)$ one can try to establish optimality of these decay rates, where optimality is defined more precisely below as in \cite{V-Martinez2000}.

Question $(Q1)$ is usually addressed within a Hilbertian framework, i.e., (weak) solutions of \eqref{eq:wave} belong to $\mathsf X_2=W^{1,2}_\ast(0,1)\times L^2(0,1)$ where $W^{1, 2}_\ast(0, 1)$ is the Hilbert space made of the absolutely continuous functions $u:[0,1]\to\mathbb{R}$ so that $u^\prime \in L^2(0, 1)$ and $u(0) = 0$. Functional analysis arguments are then used, typically by considering appropriate unbounded operators and their associated $\mathcal C^0$ semigroups. Most of the time, the function $\sigma$ is assumed to be locally Lipschitz, nondecreasing and subject to the classical \emph{damping} condition, i.e., $s\sigma(s)\geq 0$ for every $s\in\mathbb R$, which allows one to get a maximal monotone operator, and hence bringing a positive answer to $(Q1)$.

Let us emphasize on the damping condition, since not only it helps to address $(Q1)$ but it is also a first step to handle $(Q2)$. Indeed, such a condition on $\sigma$ makes the natural energy $E(t)$ of \eqref{eq:wave} defined by $E(t)=\int_0^1\left(\abs{z_x(t,x)}^2+ \abs{z_t(t,x)}^2\right)\diff x$ nonincreasing along trajectories of \eqref{eq:wave}. This is why one usually refers to such a function $\sigma$ as a \emph{damping} function. Note that \cite{pierre2000strong} (as other few earlier works mentioned in that reference) assumes $\sigma$ to be a damping function not necessarily monotone. Addressing $(Q1)$ in that case relies instead on the d'Alembert decomposition of solutions of \eqref{eq:wave}. 

Other functional frameworks have been considered recently \cite{haraux1D, chitour2019p, amadori2019decay}, where the functional spaces are of $L^p$-type, $p\in [1,+\infty]$, but these works consider $1$D wave equations with localized distributed damping, i.e., $z_{tt}-z_{xx}=-a(x)\sigma(z_t)$. Recall that the semigroup generated by the D'Alembertian $\square z:=z_{tt}-\Delta z$ with  Dirichlet boundary conditions on an open bounded subset in $\mathbb{R}^n$, $n\geq 2$, is not defined for any suitable extension of the Hilbertian framework to $L^p$-type spaces for $p\neq 2$, as explained in \cite{peral1980lp}. This is why the study of issues of asymptotic behavior involving $L^p$ spaces with $p\neq 2$ makes sense only in the one-dimensional case.

Regarding more precisely results obtained for $(Q2)$, there exist two main concepts of convergence of solutions of \eqref{eq:wave} to zero: the basic one referred to as \emph{strong stability}, which says that the $\mathsf X_2$-norm of every solution of \eqref{eq:wave} tends to zero as $t$ tends to infinity and a stronger notion, that of \emph{uniform globally asymptotic stability} (UGAS for short), which says that there exists a $\mathcal{KL}$-function $\beta:\mathbb R_+\times  \mathbb R_+\to \mathbb R$ such that $\norm{z(t)}_{\mathsf X_2}\leq\beta(\norm{z(0)}_{\mathsf X_2},t)$ for every $t\geq 0$ and solution $z(\cdot)$ of \eqref{eq:wave}. Recall that a $\mathcal{KL}$-function (or a function of class $\mathcal{KL}$) $\beta$ is continuous with  $\beta(0,\cdot)\equiv 0$, increasing with respect to its first argument and, for every $s\geq 0$,  $t\mapsto \beta(s,t)$ is decreasing and tends to zero as $t$ tends to infinity. One may interpret the function $\beta$ as a generalized rate of convergence to zero of solutions of \eqref{eq:wave} and one may even ask what could be the ``best'' $\mathcal{KL}$ function $\beta$ for which UGAS holds true. After \cite{V-Martinez2000} a reasonable definition that we will adopt in the paper goes as follows: we say that a $\mathcal{KL}$ function $\beta$ is \emph{optimal} for \eqref{eq:wave} if the latter is UGAS with rate $\beta$ and there exists an initial condition in $\mathsf X_2$ yielding a nontrivial solution $z$ of \eqref{eq:wave} so that $\norm{z(t)}_{\mathsf X_2}\geq \varepsilon \beta(\norm{z(0)}_{\mathsf X_2},t)$ for some positive constant $\varepsilon$ and every $t\geq 0$.

With the exception of \cite{pierre2000strong}, results on strong stability rely on a LaSalle argument and assume that $\sigma$ is nondecreasing and locally Lipschitz: strong stability is first established for a dense and compactly embedded subset of $\mathsf X_2$ made of regular solutions of \eqref{eq:wave} and then it is extended to the full $\mathsf X_2$ by a density argument, cf.\ \cite{Alabau2012Recent}. As regards UGAS, it can be shown that, under the damping assumption and a linear cone condition (i.e., there exist positive $a,b$ such that $a s^2\leq s\sigma(s)\leq b s^2$ for every $s\in\mathbb R$), the stability is exponential, i.e., one can choose $\beta(s,t)=C s e^{-\mu t}$ for some positive constants $C,\mu$, cf.\ \cite{V-Martinez2000} for instance. If the linear cone condition only holds in a neighborhood of zero then exponential stability cannot hold in general, as shown in \cite{V-Martinez2000} where $\sigma$ is chosen as a saturation function, i.e., such as $\sigma(s)=\arctan(s)$. Besides the linear cone condition, several results establishing UGAS have been obtained, cf.\ \cite{Lasiecka-Tataru, Martinez1999New, Liu-Zuazua,Xu2019Saturated} where $\sigma$ verifies a linear cone condition for large $s$ and is either of polynomial type or weaker than any polynomial in a neighborhood of the origin, see also \cite{V-Martinez2000} for a extensive list of references. It has to be noticed that many of these studies deal with wave equations in dimension not necessarily equal to one and, for all of them, the estimates are obtained by refined arguments based on the multiplier method or highly nontrivial Lyapunov functionals. 

Finally, for results handling $(Q3)$, most of the existing results are gathered in \cite{V-Martinez2000} and \cite{Alabau2012Recent} (cf.\ Theorems~$1.7.12$, $1.7.15$ and $1.7.16$ in the last reference)
where several of the above mentioned upper estimates (of UGAS type) are shown to be optimal in the sense defined previously and more particularly in the case where the damping function $\sigma$ in \eqref{eq:DTDS}
is of class $\mathcal C^1$ and verifies $\sigma'(0)=0$. In particular, a list of examples for which optimality is shown is provided in \cite[Theorem~1.7.12]{Alabau2012Recent} while an example is also given (Example~$5$ in that list) for which only an upper estimate is given and it is stated in \cite{Alabau2012Recent} that the general case (even under the condition $\sigma'(0)=0$) is still open. 
In these references, the upper estimates are derived by delicate manipulations of Lyapunov functions relying on the multiplier method and lower estimates results are obtained with appropriate solutions of \eqref{eq:wave} with piecewise constant Riemann invariants, where the computations are actually similar in spirit to those of \cite{pierre2000strong}. 

\subsection{Discrete-time dynamical system}

The approach proposed in this paper is inspired by \cite{pierre2000strong}, since, as in that reference, we focus on the Riemann invariants of a solution of \eqref{eq:wave}. To describe our basic finding on the matter of interest, let us consider the 
discrete-time dynamical system $\mathcal{S}$ defined on the Hilbert space $\mathsf Y_2=L^2(-1,1)$ which associates with every $h\in \mathsf Y_2$ the subset $\mathcal{S}(h)$ of $\mathsf Y_2$ made of the functions $j$ so that 
\begin{equation}\label{eq:DTDS}
\mathcal{S}: \quad \quad (h(s),j(s))\in R\Sigma,\quad \hbox{for a.e.\ } s\in [-1,1],
\end{equation}
where $R$ is the planar rotation of angle $-\pi/4$. One should notice that a related discrete-time dynamical system has been first characterized in \cite{pierre2000strong} (see also Remark~\ref{RemkImplicitDescriptionOfS} below).

We show that there exists an isometry of Hilbert spaces $\mathfrak I:\mathsf X_2\to \mathsf Y_2$ such that, for every $(z_0,z_1)\in \mathsf X_2$, \eqref{eq:wave} admits a (global in time) weak solution $z$ in $\mathsf X_2$ starting at $(z_0,z_1)$ if and only if there exists a sequence $(g_n)_{n\in\mathbb N}$ of elements in $\mathsf Y_2$ with $g_0=\mathfrak I(z_0,z_1)$ such that $g_{n+1}\in \mathcal{S}(g_n)$ for $n\in\mathbb N$. The concatenation of the $g_n$'s, which yields an element of $L^2_{\loc}(-1,+\infty)$,
is exactly the Riemann invariant $\frac{z_t - z_x}{\sqrt{2}}$ associated with $z$. It is immediate that the previously described correspondence between \eqref{eq:wave} and \eqref{eq:DTDS} can be adapted for any $p\in[1,+\infty]$ after replacing $\mathsf Y_2$ by $\mathsf Y_p=L^p(-1,1)$ and $\mathsf X_2$ by $\mathsf X_p$ defined as
\begin{equation}
\mathsf X_p:=W^{1,p}_\ast(0,1)\times L^p(0,1),
\end{equation}
where $W^{1, p}_\ast(0, 1) = \{u \in L^p(0, 1) \mid u^\prime \in L^p(0, 1) \text{ and } u(0) = 0\}$, and equipped with the norms
\begin{equation}\label{eq:norms}
\begin{aligned}
\norm*{(u,v)}^p_{\mathsf X_p} & :=\frac{1}{2^{\frac{p}2}}\int_0^1 \left(\abs*{u^\prime+v}^p+\abs*{u^\prime-v}^p\right) \diff x, & \quad p\in [1,+\infty),\\
\norm*{(u,v)}_{\mathsf X_\infty} & := \frac{1}{\sqrt{2}}\max\left(\norm{u^\prime+v}_{L^\infty(0,1)},\norm{u^\prime-v}_{L^\infty(0,1)}\right).
\end{aligned}
\end{equation}
The norm $\norm{\cdot}_{\mathsf X_p}$, previously used, for instance, in \cite{haraux1D}, is equivalent to the standard norm in $\mathsf X_p$, but it has the advantage of being well-adapted to the analysis of wave equations, since it is expressed in terms of Riemann invariants and it is nonincreasing as soon as $\Sigma$ satisfies a damping condition (see Proposition~\ref{PropDampingNonStrict}). In addition, with this choice of norm, the mapping (still denoted by) $\mathfrak I:\mathsf X_p\to \mathsf Y_p$ becomes an isometry between Banach spaces.

One can reformulate appropriately the definition of $\mathcal S$ by considering the set-valued map $S:\mathbb R\rightrightarrows \mathbb R$ whose graph is $R\Sigma\subset \mathbb R^2$. For instance \eqref{eq:DTDS} reads $j(s)\in S(h(s))$ for a.e.\ $s\in [-1,1]$. A first interesting application of this formalism concerns the case where $\Sigma$ is equal to the graph of the sign set-valued map. Then the set-valued map $S$ becomes single-valued and equal to the odd function defined on $\mathbb R_+$ by $S(s)=s$ for $s\in [0,1/\sqrt{2}]$ and $S(s)=\sqrt{2}-s$ for $s\geq 1/\sqrt{2}$. We also show how one can easily adapt the previously described approach to the case where the Dirichlet boundary condition at $x=0$ in \eqref{eq:wave} is replaced by the more general boundary condition $(z_t(t, 0), z_x(t, 0)) \in \Sigma'$, $t\geq 0$, where $\Sigma'\subset \mathbb R^2$.

By using the one-to-one correspondence between (global in time) solutions of \eqref{eq:wave} and iterated sequences of \eqref{eq:DTDS}, we can translate the previously described issues of $(Q1)$--$(Q3)$ associated with the original $1$D wave equation \eqref{eq:wave} into the study of equivalent questions in terms of $\mathcal{S}$. Question $(Q1)$ of existence and uniqueness of solutions of \eqref{eq:wave} in $\mathsf X_p$ for every initial condition is equivalently restated  as follows: determine conditions on $\Sigma$ so that, for every $h \in \mathsf Y_p$, there exists some (possibly unique) $j \in \mathsf Y_p$ such that $j \in \mathcal{S}(h)$. Similarly, questions regarding the decay rates, i.e.\ $(Q2)$ and $(Q3)$, can be equivalently simply expressed as questions regarding the asymptotic behavior of iterated sequences $(g_n)_{n\in \mathbb N}$ associated with $\mathcal{S}$. 

\subsection{Main results}
Concerning the question of existence of (global in time) solutions of \eqref{eq:wave} for every initial condition in $\mathsf X_p$, we provide a sufficient condition in terms of $\Sigma$ which turns out to be also necessary in case $S$ is single-valued. The condition reads differently whether $p$ is finite or not but, in both cases, $S$ must contain the graph of a \emph{universally measurable} function, cf.\ \cite{Bertsekas1978Stochastic, Cohn2013Measure, Nishiura2008Absolute} and Definition~\ref{DefiUniversally} in Appendix~\ref{AppUniversally} for a definition of the latter concept.

As regards the asymptotic behavior of solutions of \eqref{eq:wave} we always work under the assumption that $\Sigma$ is a damping set, i.e., for every $(x,y)\in \Sigma$, it holds $xy\geq 0$. We also refer to a \emph{strict} damping set in case the previous inequality is strict for $(x,y)\in \Sigma\setminus\{(0,0)\}$. These conditions generalize the case where $\Sigma$ is the graph of a damping function $\sigma: \mathbb R\to \mathbb R$. Note that the damping assumption on $\Sigma$ translates to $\abs{y}\leq \abs{x}$ for $(x,y)\in R\Sigma$ and to the corresponding strict inequality for $(x,y) \in R\Sigma \setminus \{(0,0)\}$.

Our first main result says that the asymptotic behavior of solutions of \eqref{eq:wave} in $\mathsf X_p$ (and hence of iterated sequences $(g_n)_{n\in\mathbb N}$ for $\mathcal{S}$ in $\mathsf Y_p$) is governed by the asymptotic behavior of real iterated sequences $(x_n)_{n\in\mathbb N}$ for $S$, i.e., real sequences such that $x_{n+1}\in S(x_n)$, for every $n\in\mathbb N$. In particular, strong stability in $\mathsf{X}_p$ for $p$ finite is equivalent to the fact that every real iterated sequence $(x_n)_{n\in\mathbb N}$ for $S$ converges to zero, while for $p=+\infty$, strong stability and UGAS are equivalent, themselves holding true if and only if real iterated sequences $(x_n)_{n\in\mathbb N}$ for $S$ converge to zero, uniformly with respect to compact sets of initial conditions $x_0$. Moreover, if the set-valued map $S\circ S$ has a closed graph, then strong stability holds true in $\mathsf{X}_p$, $p\in [1,+\infty]$, if and only if $S\circ S$ is a strict damping. This greatly generalizes a result of \cite{pierre2000strong} where strong stability in $\mathsf{X}_2$ has been established in the case where $S$ is a (single-valued) continuous function on $\mathbb R$. For UGAS with $p\in [1,+\infty)$, we characterize two conditions on $\Sigma$ ensuring, for the first one (see \ref{HypoSigma-NoDampingAtInfty} below), that UGAS does not hold true (cf.\ Proposition~\ref{PropPFiniteNotUGAS}) while, on the opposite, the second one (see \ref{HypoSigma-SectorInfty} below) combined with UGAS in $\mathsf X_{\infty}$ is sufficient for UGAS in $\mathsf X_p$ to hold true (see Theorem~\ref{thm-ugas}). 

As far as decay rates of solutions of \eqref{eq:wave} are concerned, we consider damping sets $\Sigma$ subject to two generalized sector conditions describing the behavior of the set $R \Sigma$ in some neighborhood of the origin as follows: for points $(x, y)$ in that neighborhood, the first condition (Hypothesis \ref{HypoSigma-gUpperTilde} below) assumes that $\abs{y} \leq Q(\abs{x})$, whereas the second one (see Hypothesis \ref{HypoSigma-gLowerTilde} below) assumes that $\abs{y} \geq Q(\abs{x})$, where $Q \in \mathcal C^1(\mathbb R_+, \mathbb R_+)$ with $Q(0) = 0$, $0 < Q(x) < x$, and $Q^\prime(x) > 0$ for every $x>0$. These two conditions are inspired from \cite{V-Martinez2000} where they are only expressed in terms of $\Sigma$ as the graph of a continuous function (see also Hypotheses \ref{HypoSigma-gUpper} and \ref{HypoSigma-gLower} below). Thanks to our approach, the decay rate issue amounts to study real iterated sequences $(x_n)_{n\in \mathbb N}$ verifying either $\abs{x_{n+1}}\leq Q(\abs{x_n})$ when \ref{HypoSigma-gUpperTilde} holds or 
$\abs{x_{n+1}}\geq Q(\abs{x_n})$ when Hypothesis \ref{HypoSigma-gLowerTilde}
holds.

As for the optimality issue, it is now reduced to the determination of equivalents in terms of a function of $n\in \mathbb N$, as it tends to infinity, for the real iterated sequences $(x_n)_{n\in \mathbb N}$ verifying $\abs{x_{n+1}}=Q(\abs{x_n})$ for $n\geq 0$. We provide precise asymptotic results on the decay of solutions of \eqref{eq:wave} depending on the value of $Q'(0)\in [0,1]$. First of all, we are able to recover all the cases listed in \cite[Theorem~1.7.12]{Alabau2012Recent}, even for Example~$5$, which was left open, all of them corresponding  to situations where $Q'(0)=1$. Note that we obtain the previously known cases with simpler arguments and we also characterize the largest possible sets of initial conditions admitting optimal decay rates. If $Q'(0)\in (0,1)$, one has (local) exponential stability, a decay rate which is optimal. The handling of this case is rather elementary and is known in the literature, as least in the Hilbertian case $p=2$. In the particular case where $Q(s)=\mu s$ for every $s\geq 0$ for some $\mu\in (0,1)$, we provide alternative arguments for exponential stability and also give a necessary and sufficient and condition in terms of $S$ only for exponential stability to hold true. Our results concerning the case $Q'(0)=0$ seem to be new and exhibit convergence rates faster than any exponential ones.

We close the set of our findings on decay rates with the solution of a conjecture formulated in \cite{V-Martinez2000} asking whether arbitrary slow convergence is possible in case $\Sigma$ is a damping set of saturation type, i.e., the values of $\abs{y}$ for $(x,y)\in\Sigma$ and $\abs{x}$ large remain bounded by a given positive constant. We bring a positive answer for the possible occurrence of such an arbitrarily slow convergence in any space $\mathsf{X}_p$, $p\in [1,+\infty)$. 

We also have sharp results when the wave equation \eqref{eq:wave} is subject to boundary perturbations, i.e., the condition $(z_t(t, 1), -z_x(t, 1))\in\Sigma$ for $t \in \mathbb R_+$ is replaced by $(z_t(t, 1), -z_x(t, 1))\in\Sigma+d(t)$ for $t \in \mathbb R_+$, where
$d: \mathbb R_+ \to \mathbb R^2$ is a measurable function representing the perturbation. We provide two sets of results, the first one dealing with asymptotic perturbation rejection (i.e., conditions on $d$ and $\Sigma$ so that solutions of \eqref{eq:wave} converge to zero despite the presence of $d$) and another set of results proposing sufficient conditions on $\Sigma$ ensuring input-to-state stability for the perturbed wave equation (cf.\ \cite{Mironchenko2020Input} for a definition of input-to state-stability).

Finally, we revisit the case where $\Sigma$ is equal to the graph of the sign set-valued map $\sign$ and extend all the results obtained in \cite{Xu2019Saturated} regarding this question. In particular, we provide optimal results for existence and uniqueness of solutions of \eqref{eq:wave} in any $\mathsf{X}_p$, $p\in [1,+\infty]$ without relying on semigroup theory and we characterize the $\omega$-limit set of every solution of \eqref{eq:wave} in an explicit manner in terms of the initial condition. 

\subsection{Structure of the paper}

The paper is organized in six sections and four appendices. After the present introduction describing the contents of our paper and gathering the main notations, Section~\ref{sec:description} is devoted to the description of the precise correspondence between the wave equation described by \eqref{eq:wave} and the discrete-time dynamical system given in \eqref{eq:DTDS}, as well as the list of meaningful hypotheses one can assume on $\Sigma$ and auxiliary results on the set-valued map $S$. Section~\ref{SecExistUnique} provides results on the existence and uniqueness of solutions of \eqref{eq:wave}. Section~\ref{SecAsymptotic} gathers results dealing with stability concepts, asymptotic behavior, decay rates, and their optimality for solutions of both iterated sequences of $\mathcal{S}$ and solutions of \eqref{eq:wave}. Section~\ref{SecISS} treats the case of boundary perturbations. Section~\ref{SecSign} addresses the situation where $\Sigma$ is equal to the graph of the sign set-valued map $\sign$. The four appendices collect lemmas and technical arguments used in the core of the text.

\subsection{Notations}

The sets of integers, nonnegative integers, real numbers, nonnegative real numbers, and nonpositive real numbers are denoted in this paper respectively by $\mathbb Z$, $\mathbb N$, $\mathbb R$, $\mathbb R_+$, and $\mathbb R_-$. For $\mathbb A \in \{\mathbb Z, \mathbb N, \mathbb R, \mathbb R_+, \mathbb R_-\}$, we use $\mathbb A^\ast$ to denote the set $\mathbb A \setminus \{0\}$.

For $x \in \mathbb{R}$, we use $\floor{x}$ and $\ceil{x}$ to denote, respectively, the greatest integer less than or equal to $x$ and the smallest integer greater than or equal to $x$. The set $\mathbb R^d$, $d \in \mathbb N^\ast$, is assumed to be endowed with its usual Euclidean norm, denoted by $\abs{\cdot}$, and, given $M > 0$, $B(0,M)$ denotes the ball of $\mathbb R^d$ centered at zero and of radius $M$. If $A \subset \mathbb R$, we define $\norm{A} = \sup_{a \in A} \abs{a}$. All along the paper, we use the letter $R$ to denote the matrix corresponding to the plane rotation of angle $-\frac{\pi}4$, i.e.,
\[
R = \begin{pmatrix}
\frac{1}{\sqrt{2}} & \frac{1}{\sqrt{2}} \\
-\frac{1}{\sqrt{2}} & \frac{1}{\sqrt{2}} \\
\end{pmatrix}.
\]
The identity function of $\mathbb R$ is denoted by $\id$ and, for $\Sigma \subset \mathbb R^2$ and $d \in \mathbb R^2$, we define the sum $\Sigma + d$ as the set $\{x + d \suchthat x \in \Sigma\}$.

A \emph{set-valued} map $F: \mathbb R \rightrightarrows \mathbb R$ is a function that, with each $x \in \mathbb R$, associates some (possibly empty) $F(x) \subset \mathbb R$, and its graph is the set $\{(x, y) \in \mathbb R^2 \suchthat y \in F(x)\}$. A set-valued map $F$ is said to be \emph{multi-valued} if $F(x) \neq \emptyset$ for every $x \in \mathbb R$, i.e., the graph of $F$ contains the graph of a function $\varphi: \mathbb R \to \mathbb R$. It is said to be \emph{single-valued} when $F(x)$ is a singleton for every $x \in \mathbb R$, i.e., the graph of $F$ is the graph of a function $\varphi: \mathbb R \to \mathbb R$. In that case, we usually make the slight abuse of notation of considering $F = \varphi$.

The composition of two set-valued maps $S$ and $T$ is the set-valued map $S \circ T: \mathbb R \rightrightarrows \mathbb R$ which, to each $x \in \mathbb R$, associates the set of points $z\in \mathbb R$ such that there exists $y\in \mathbb R$ for which $z\in S(y)$ and $y\in T(x)$.

Consider a function $f:\mathbb{R}\to\mathbb{R}$. Then, for $n \in \mathbb N$, the $n$-th iterate of $f$, i.e., the composition of $f$ with itself $n$ times, is denoted by $f^{[n]}$, with the convention that $f^{[0]}=\id$. This notation is extended in a straightforward manner to set-valued functions $F: \mathbb R \rightrightarrows \mathbb R$: $y \in F^{[n]}(x)$ if and only if there exists $x_0, \dotsc, x_n \in \mathbb R$ such that $x_0 = x$, $x_n = y$, and $x_{i+1} \in F(x_i)$ for every $i \in \{0, \dotsc, n-1\}$.

Several notions of measurability are used in some parts of the paper (see Appendix~\ref{AppUniversally}). Unless otherwise specified, the word ``measurable'' means ``Lebesgue measurable''. For an interval $I \subset \mathbb R$, $d \in \mathbb N^\ast$, and $p \in [1, +\infty]$, the space $L^p(I, \mathbb R^d)$ is endowed with the norm defined by
\begin{align*}
\norm{u}_{L^p(I, \mathbb R^d)}^p & = \int_I \abs{u(t)}^p \diff t, & \qquad & \text{if } p < +\infty,\\
\norm{u}_{L^\infty(I, \mathbb R^d)} & = \esssup_{t \in I} \abs{u(t)}.
\end{align*}
The space $\mathbb R^d$ is omitted from the notation when $d = 1$. We use $\mathsf Y_p$ to denote the space $L^p(-1, 1)$ and we write its norm simply by $\norm{\cdot}_{p}$.

A function $\gamma: \mathbb R_+ \to \mathbb R_+$ is said to be of \emph{class $\mathcal K$} if $\gamma$ is continuous, increasing, and $\gamma(0) = 0$. If moreover $\lim_{x \to +\infty} \gamma(x) = +\infty$, we say that $\gamma$ is of \emph{class $\mathcal K_\infty$}.

A function $\beta: \mathbb R_+ \times \mathbb R_+ \to \mathbb R_+$ is said to be of \emph{class $\mathcal{KL}$} if it is continuous, $\beta(\cdot, t)$ is of class $\mathcal K$ for every $t \in \mathbb R_+$, and, for every $x \in \mathbb R_+$, $\beta(x, \cdot)$ is decreasing and $\lim_{t \to +\infty} \beta(x, t) = 0$.

\section{Description of the model}
\label{sec:description}

\subsection{Equivalent discrete-time dynamical system}

In order to introduce a notion of weak solution of \eqref{eq:wave} adapted both to $L^p$ spaces and to the one-dimensional case, let us recall the following classical result on regular solutions to the one-dimensional wave equation, which corresponds to its d'Alembert decomposition into traveling waves (see, e.g., \cite[Section~2.4.1.a]{Evans2010Partial}).

\begin{proposition}
\label{PropDAlembert}
Let $z \in \mathcal C^2(\mathbb R_+ \times [0, 1])$. Then $z$ satisfies $z_{xx} = z_{tt}$ in $\mathbb R_+ \times [0, 1]$ if and only if there exist functions $f \in \mathcal C^1([0, +\infty))$ and $g \in \mathcal C^1([-1, +\infty))$ such that
\begin{equation}
\label{EqWeakWave}
z(t, x) = z(0, 0) + \frac{1}{\sqrt{2}} \int_0^{t+x} f(s) \diff s + \frac{1}{\sqrt{2}} \int_0^{t-x} g(s) \diff s.
\end{equation}
\end{proposition}

\begin{proof}
Assume that $z$ satisfies $z_{xx} = z_{tt}$ in $\mathbb R_+ \times [0, 1]$ and let $u, v: \mathbb R_+ \times [0, 1] \to \mathbb R$ satisfy
\begin{equation}
\label{dAlembert}
\begin{pmatrix}
v(t, x) \\
-u(t, x) \\
\end{pmatrix} = R \begin{pmatrix}
z_t(t, x) \\
-z_x(t, x) \\
\end{pmatrix}
\end{equation}
Then $u, v \in \mathcal C^1(\mathbb R_+ \times [0, 1])$ and $u_t = u_x$, $v_t = -v_x$ in $\mathbb R_+ \times [0, 1]$. One immediately verifies that, for every $(t, x) \in \mathbb R_+ \times [0, 1]$, the functions $h \mapsto u(t+h, x-h)$ and $h \mapsto v(t+h, x+h)$ are constant in their domains. Letting $f: [0, +\infty) \to \mathbb R$ and $g: [-1, +\infty) \to \mathbb R$ being defined by
\begin{equation}\label{eq:IR}
f(s) = u(s, 0), \qquad g(s) = \begin{dcases*}
v(s, 0) & if $s \geq 0$, \\
v(0, -s) & if $-1 \leq s < 0$,
\end{dcases*}
\end{equation}
one can easily check that $f \in \mathcal C^1([0, +\infty))$, $g \in \mathcal C^1([-1, +\infty))$, and $u(t, x) = f(t+x)$ and $v(t, x) = g(t-x)$ for every $(t, x) \in \mathbb R_+ \times [0, 1]$. In particular, it follows from \eqref{dAlembert} that
\begin{equation}
\label{eq:der-sol}
\begin{pmatrix}
z_t(t, x) \\
-z_x(t, x) \\
\end{pmatrix} = R^{-1} \begin{pmatrix}
g(t - x) \\
-f(t + x) \\
\end{pmatrix}.
\end{equation}
Hence
\[
\begin{split}
z(t, x) & = z(0, 0) + \int_0^x z_x(0, s) \diff s + \int_0^t z_t(s, x) \diff s \\
& = z(0, 0) + \frac{1}{\sqrt{2}} \int_0^x f(s) \diff s - \frac{1}{\sqrt{2}} \int_0^x g(-s) \diff s \\
& \hphantom{{} = z(0, 0) } + \frac{1}{\sqrt{2}} \int_0^t f(s+x) \diff s + \frac{1}{\sqrt{2}} \int_0^t g(s-x) \diff s \\
& = z(0, 0) + \frac{1}{\sqrt{2}} \int_0^{t+x} f(s) \diff s + \frac{1}{\sqrt{2}} \int_0^{t-x} g(s) \diff s,
\end{split}
\]
as required. Conversely, if $z$ is given by \eqref{EqWeakWave}, it is easy to see that $z_{tt} = z_{xx}$.
\end{proof}

The functions $f$ and $g$ from \eqref{eq:IR} are called \emph{Riemann invariants} in the classical literature of hyperbolic PDEs (see, for instance, \cite{bastin2016stability}). Proposition~\ref{PropDAlembert} motivates the following definition of weak solution to \eqref{eq:wave}.

\begin{definition}
\label{DefWeakSolution}
Let $(z_0, z_1) \in \mathsf X_p$. We say that $z: \mathbb R_+ \times [0, 1] \to \mathbb R$ is a \emph{weak global (in time) solution} of \eqref{eq:wave} in $\mathsf X_p$ with initial condition $(z_0, z_1)$ if there exist $f \in L^p_{\loc}(0, +\infty)$ and $g \in L^p_{\loc}(-1, +\infty)$ such that
\begin{equation}
\label{eq:defWeakSol}
\left\{
\begin{aligned}
& z(t, x) = \frac{1}{\sqrt{2}} \int_{0}^{t+x} f(s) \diff s + \frac{1}{\sqrt{2}} \int_0^{t-x} g(s) \diff s & \quad & \text{for all } (t, x) \in \mathbb R_+ \times [0, 1], \\
& z(t,0) = 0,\: (z_t(t, 1), -z_x(t, 1)) \in \Sigma, & & \text{for a.e.\ } t \in \mathbb{R}_+,\\
&z(0,x) = z_{0}(x),\: z_t(0,x) = z_1(x), & & \text{for a.e.\ } x \in [0,1].
\end{aligned}
\right.
\end{equation}
In that case, we use $e_p(z)(t)$ to denote the $\mathsf X_p$ norm of the weak global solution $z$ of \eqref{eq:wave} at time $t \in \mathbb R_+$, defined by
\[
e_p(z)(t) = \norm{(z(t, \cdot), z_t(t, \cdot))}_{\mathsf X_p}.
\]
\end{definition}

Note that, if $z$ is a weak global solution of \eqref{eq:wave} in $\mathsf X_p$, then $(z(t, \cdot), z_t(t, \cdot)) \in \mathsf X_p$ for every $t \in \mathbb R_+$ and $z_{tt} = z_{xx}$ is satisfied in $\mathbb R_+^\ast \times (0, 1)$ in the sense of distributions. In the sequel, we refer to weak global solutions of \eqref{eq:wave} simply as \emph{solutions} or \emph{trajectories} of \eqref{eq:wave} and, by a slight abuse of expression, we refer to $e_p(z)(t)$ as the \emph{energy} of $z$ at time $t$.

By rewriting the boundary and initial conditions of \eqref{eq:wave} in terms of 
the functions $f$ and $g$ from Definition~\ref{DefWeakSolution}, one obtains at once the following characterization of solutions of 
\eqref{eq:wave}, when they exist.

\begin{proposition}\label{prop:charac1-sol}
Let $z: \mathbb R_+ \times [0, 1] \to \mathbb R$ be a solution of \eqref{eq:wave} with initial condition $(z_0, z_1) \in \mathsf X_p$ and $f \in L^p_{\loc}(0, +\infty)$ and $g \in L^p_{\loc}(-1, +\infty)$ be the corresponding functions from Definition~\ref{DefWeakSolution}. Then $f$ and $g$ satisfy
\begin{subequations}
\label{BoundaryInitialFG}
\begin{align}
& \begin{pmatrix}
g(-s) \\
-f(s) \\
\end{pmatrix} = R \begin{pmatrix}
z_1(s) \\
-z_0^\prime(s) \\
\end{pmatrix}, & & \text{for a.e.\ } s \in [0, 1], \label{first-interval} \\
& f(s) = - g(s), & & \text{for a.e.\ } s \geq 0, \label{relation-f-g} \\
& (g(s - 2), g(s)) \in R \Sigma, & & \text{for a.e.\ } s \geq 1. \label{recurrence-g}
\end{align}
\end{subequations}

Conversely, consider $g \in L^p_{\loc}(-1, +\infty)$ verifying \eqref{recurrence-g} and let $f \in L^p_{\loc}(0, +\infty)$ be given by \eqref{relation-f-g}. Then the function $z: \mathbb R_+ \times [0, 1] \to \mathbb R$ defined by the integral formula from \eqref{eq:defWeakSol} is a solution of \eqref{eq:wave} whose initial condition $(z_0, z_1) \in \mathsf X_p$ is the unique couple of functions satisfying \eqref{first-interval}.
\end{proposition}

The main technique underlying all the results of our paper consists in establishing links between trajectories of \eqref{eq:wave} and trajectories of discrete-time dynamical systems which are defined next. 

\begin{definition}\label{def:iterated}
Let $S: \mathbb R \rightrightarrows \mathbb R$ be a set-valued map. We refer to the inclusion
\[
x_{n+1} \in S(x_n), \qquad n \in \mathbb N,\; x_n \in \mathbb R,
\]
as the \emph{discrete-time dynamical system associated with $S$ on $\mathbb R$} and its corresponding trajectories $(x_n)_{n \in \mathbb N}$ are called \emph{real iterated sequences for $S$}.

Similarly, we refer to the inclusion
\begin{equation}
\label{MainEquationOnGn}
g_{n+1}(s) \in S(g_n(s)), \qquad n \in \mathbb N,\; \text{a.e.\ }s \in [-1, 1],
\end{equation}
as the \emph{discrete-time dynamical system associated with $S$ on the space of real-valued measurable functions defined on $[-1, 1]$} and its corresponding trajectories $(g_n)_{n \in \mathbb N}$ are called \emph{iterated sequences} for $S$. For $p \in [1, +\infty]$, the \emph{discrete-time dynamical system associated with $S$ on $\mathsf Y_p$} is defined as the restriction of the above dynamical system to sequences $(g_n)_{n \in \mathbb N}$ in $\mathsf Y_p$.
\end{definition}

In the sequel of the paper, we will connect solutions of \eqref{eq:wave} and solutions of \eqref{MainEquationOnGn} with $S$ being the set-valued map whose graph is $R \Sigma$, in which case \eqref{MainEquationOnGn} is another way of writing \eqref{recurrence-g}. For that purpose, we need to introduce some additional notations.
In the following definition, we use $\mathsf Y_p^{\mathbb{N}}$ to denote the set of sequences taking values in $\mathsf Y_p$.

\begin{definition}
\label{DefiSeqI}
Let $p \in [1, +\infty]$.

\begin{enumerate}
\item We use $\Seq: L^p_{\loc}(-1, +\infty) \to \mathsf Y_p^{\mathbb N}$ to denote the bijection which associates, with each $g \in L^p_{\loc}(-1,\allowbreak +\infty)$, the sequence $(g_n)_{n \in \mathbb N}$ in $\mathsf Y_p$ defined by $g_n(s) = g(s + 2n)$ for $n \in \mathbb N$ and a.e.\ $s \in [-1, 1]$.

\item We use $\mathfrak I: \mathsf X_p \to \mathsf Y_p$ to denote the isometry which associates, with each $(z_0, z_1) \in \mathsf X_p$, the element $g_0 \in \mathsf Y_p$ defined by
\[
\begin{pmatrix}
g_0(-s) \\
g_0(s) \\
\end{pmatrix} = R \begin{pmatrix}
z_1(s) \\
-z_0^\prime(s) \\
\end{pmatrix}, \qquad \text{for a.e.\ } s \in [0, 1].
\]
\end{enumerate}
\end{definition}

As a consequence of Proposition~\ref{prop:charac1-sol} and the above definitions, one immediately obtains the following one-to-one correspondence between solutions of \eqref{eq:wave} in $\mathsf X_p$ and trajectories of the dynamical system \eqref{MainEquationOnGn} in $\mathsf Y_p$.

\begin{proposition}
\label{PropGn}
Let $p \in [1, +\infty]$, $\Sigma \subset \mathbb R^2$, and $S: \mathbb R \rightrightarrows \mathbb R$ be the set-valued map whose graph is $R \Sigma$.
\begin{enumerate}
\item\label{ItemEquivSols1} Let $z$ be a solution of \eqref{eq:wave} with initial condition $(z_0, z_1) \in \mathsf X_p$, $g \in L_{\loc}^p(-1, +\infty)$ be the corresponding function from Definition~\ref{DefWeakSolution}, and $(g_n)_{n \in \mathbb N} = \Seq(g)$. Then $(g_n)_{n \in \mathbb N}$ is an iterated sequence for $S$ on $\mathsf Y_p$ starting at $g_0 = \mathfrak I(z_0, z_1)$.

\item\label{ItemEquivSols2} Conversely, let $(g_n)_{n \in \mathbb N}$ be an iterated sequence for $S$ on $\mathsf Y_p$ starting at some $g_0 \in \mathsf Y_p$. Let $g = \Seq^{-1}((g_n)_{n \in \mathbb N})$, $f \in L^p_{\loc}(0, +\infty)$ be given by $f(s) = -g(s)$ for a.e.\ $s \geq 0$, and $z$ be defined from $f$ and $g$ as in the first equation of \eqref{eq:defWeakSol}. Then $z$ is a solution of \eqref{eq:wave} in $\mathsf X_p$ with initial condition $(z_0, z_1) = \mathfrak I^{-1}(g_0)$.

\item Let $z$, $g$, and $(g_n)_{n \in \mathbb N}$ be as in \ref{ItemEquivSols1} or \ref{ItemEquivSols2}. Then
\begin{equation}
\label{eq:norm}
e_p(z)(t) = \norm{g(t + \cdot)}_p, \qquad \text{ for all } t \in \mathbb R_+
\end{equation}
and, in particular,
\begin{equation}
\label{eq:normGn}
e_p(z)(2n) = \norm{g_n}_p, \qquad \text{ for all } n \in \mathbb N.
\end{equation}
\end{enumerate}
\end{proposition}

\begin{proof} Items \ref{ItemEquivSols1} and \ref{ItemEquivSols2} are reformulations of
Proposition~\ref{prop:charac1-sol}. As for \eqref{eq:norm}, it follows from the definition of $e_p$, \eqref{eq:der-sol}, and \eqref{relation-f-g}.
\end{proof}

Saying that $(g_n)_{n \in \mathbb N}$ is an iterated sequence for the set-valued map $S$ given in the statement of Proposition~\ref{PropGn} is equivalent to
\begin{equation}
\label{eq:dynSystGn}
(g_n(s), g_{n+1}(s)) \in R \Sigma, \qquad n \in \mathbb N,\; \text{a.e.\ } s \in [-1, 1],
\end{equation}
which is nothing but \eqref{recurrence-g} rewritten in terms of the sequence $(g_n)_{n \in \mathbb N}$. 

It is now clear, at the light of what precedes, that addressing standard issues for solutions of \eqref{eq:wave} such as existence, uniqueness, and decay rates and their optimality is completely equivalent to addressing the same issues for sequences $(g_n)_{n \in \mathbb N}$ in $\mathsf Y_p$ verifying \eqref{eq:dynSystGn}. This is the point of view that we will adopt all along the paper.

\begin{remark}
\label{RemkExistUnique}
Thanks to the iterative nature of discrete-time dynamical systems in $\mathsf Y_p$, Proposition~\ref{PropGn} reduces the issue of existence (resp.\ existence and uniqueness) of solutions of \eqref{eq:wave} in $\mathsf X_p$ for every initial condition in $\mathsf X_p$ to the following equivalent statement in terms of $S$: for every $g \in \mathsf Y_p$, there exists (resp.\ there exists a unique) $h \in \mathsf Y_p$ such that $h(s) \in S(g(s))$ for a.e.\ $s \in [-1, 1]$.
\end{remark}

\begin{remark}
\label{remk:energy-derivative}
For $p < +\infty$, one deduces from \eqref{eq:norm} that, for every $t \geq 0$,
\[
e_p^p(z)(t) = \int_{t-1}^{t+1} \abs{g(s)}^p \diff s,
\]
and hence $t \mapsto e_p^p(z)(t)$ is absolutely continuous and
\begin{equation}
\label{eq:deriv-energy}
\frac{\diff}{\diff t} e_p^p(z)(t) = \abs{g(t+1)}^p - \abs{g(t-1)}^p,\hbox{ for a.e.\ }t\geq 0.
\end{equation}
\end{remark}

\begin{remark}
\label{RemkImplicitDescriptionOfS}
When $\Sigma$ is the graph of a function $\sigma: \mathbb R \to \mathbb R$, the set-valued map $S: \mathbb R \rightrightarrows \mathbb R$ whose graph is $R \Sigma$ can be described as follows: for $x \in \mathbb R$, $S(x)$ is the set of solutions $y \in \mathbb R$ of the equation
\[\sigma\left(\frac{x - y}{\sqrt{2}}\right) = \frac{x + y}{\sqrt{2}}.\]
A similar equation has been given in \cite{pierre2000strong} but, instead of working with a set-valued map $S$, the authors consider instead solutions $y$ of the above equation of minimal absolute value.
\end{remark}

\subsection{Hypotheses on \texorpdfstring{$\Sigma$}{Sigma}}

To prepare for the sequel of the paper, we provide a list of assumptions on $\Sigma$ that will be useful to characterize existence, uniqueness, or asymptotic behavior of \eqref{eq:wave}. The results of this paper will require subsets of these assumptions, which are explicitly stated in each result. We stress the fact that we do not assume all these assumptions on $\Sigma$ at the same time, since most of our results do not require all of them. Note that we will use the two notions of \emph{universally measurable function} (whose definition is recalled in Appendix~\ref{AppUniversally}) and \emph{function with linear growth}, i.e., functions $\varphi: \mathbb R \to \mathbb R$ for which there exist $a, b \in \mathbb R_+$ such that $$\abs{\varphi(x)} \leq a \abs{x} + b,\qquad \forall x\in \mathbb{R}.$$

\begin{hypotheses}
\label{HypoSigma}
The following hypotheses concern a set $\Sigma \subset \mathbb R^2$ and the set-valued map $S$ whose graph is $R\Sigma$.
\begin{listhypo}[align=left, labelwidth=\widthof{\normalsize (H3)$_\infty$\ }, leftmargin=!]
\item\label{HypoSigma-Zero} $(0, 0) \in \Sigma$.
\item\label{HypoSigma-Exists} $R\Sigma$ contains the graph of a universally measurable function $\varphi$ with linear growth.

\item[\customlabel{HypoSigma-ExistsInfty}{\ref{HypoSigma-Exists}$_\infty$}] $R\Sigma$ contains the graph of a universally measurable function $\varphi$ mapping boun\-ded sets to bounded sets.

\item\label{HypoSigma-Unique} $R\Sigma$ is equal to the graph of a universally measurable function $\varphi$ with linear growth.

\item[\customlabel{HypoSigma-UniqueInfty}{\ref{HypoSigma-Unique}$_\infty$}] $R\Sigma$ is equal to the graph of a universally measurable function $\varphi$ mapping bounded sets to bounded sets.

\item\label{HypoSigma-Damping} For every $(x, y) \in \Sigma$, one has $x y \geq 0$.
\item\label{HypoSigma-StrictDamping} For every $(x, y) \in \Sigma \setminus \{(0, 0)\}$, one has $x y > 0$.
\item\label{HypoSigma-NoDampingAtInfty} One has
\[
\lim_{\substack{\abs{(x, y)} \to +\infty \\ (x, y) \in \Sigma}} \min\left(\frac{x}{y}, \frac{y}{x}\right) = 0.
\]

\item\label{HypoSigma-SectorZero} $\Sigma$ satisfies \ref{HypoSigma-Damping} and there exist positive constants $M, a, b$ such that
\[
a \abs{x} \leq \abs{y} \leq b \abs{x}, \qquad \text{ for every } (x, y) \in \Sigma \cap B(0, M).
\]

\item\label{HypoSigma-SectorInfty} $\Sigma$ satisfies \ref{HypoSigma-Damping} and there exist positive constants $M, a, b$ such that
\[
a \abs{x} \leq \abs{y} \leq b \abs{x}, \qquad \text{ for every } (x, y) \in \Sigma \setminus B(0, M).
\]

\item\label{HypoSigma-gUpper} $\Sigma$ satisfies \ref{HypoSigma-Damping} and there exist a positive constant $M$ and a function $q \in \mathcal C^1(\mathbb R_+, \mathbb R_+)$ with $q(0) = 0$, $0 < q(x) < x$, and $\abs{q^\prime(x)} < 1$ for every $x>0$ such that
\[
q(\abs{x}) \leq \abs{y} \quad \text{ and } \quad q(\abs{y}) \leq \abs{x}, \qquad \text{ for every } (x, y) \in \Sigma \cap B(0, M).
\]

\item\label{HypoSigma-gLower} $\Sigma$ satisfies \ref{HypoSigma-Damping} and there exist a positive constant $M$ and a function $q \in \mathcal C^1(\mathbb R_+, \mathbb R_+)$ with $q(0) = 0$, $0 < q(x) < x$, and $\abs{q^\prime(x)} < 1$ for every $x>0$ such that
\[
\abs{y} \leq q(\abs{x}) \quad \text{ or } \quad \abs{x} \leq q(\abs{y}), \qquad \text{ for every } (x, y) \in \Sigma \cap B(0, M).
\]
\end{listhypo}
\end{hypotheses}

Throughout the paper we will often assume $\Sigma$ to be a \emph{damping set}, whose definition is given next.

\begin{definition}[Damping set]
\label{def-damping}
A set $\Sigma \subset \mathbb R^2$ is called a \emph{damping set} (or simply \emph{damping}) if it satisfies \ref{HypoSigma-Zero}, \ref{HypoSigma-Exists}, and \ref{HypoSigma-Damping}. It is said to be \emph{strict} when one requires \ref{HypoSigma-StrictDamping} to be satisfied instead of \ref{HypoSigma-Damping}.
\end{definition}

By a slight abuse of notation, we will also refer to the set-valued function $S$ whose graph is $R \Sigma$ as a (resp.\ \emph{strict}) \emph{damping} when $\Sigma$ is a (resp.\ strict) damping.

Assumption \ref{HypoSigma-Zero} is used to guarantee that $z \equiv 0$ is a solution of \eqref{eq:wave}. When $\Sigma$ is the graph of a function $\sigma$, then \ref{HypoSigma-Zero} reduces to $\sigma(0) = 0$.

In the case where $\Sigma$ is the graph of a \emph{linear} function $\sigma(x) = \alpha x$, it is standard that a necessary and sufficient condition for the existence of solutions of \eqref{eq:wave} is $\alpha \neq -1$, i.e., that $R \Sigma$ is not the vertical axis $x = 0$. Hypotheses \ref{HypoSigma-Exists}--\ref{HypoSigma-UniqueInfty} prevent this phenomenon of nonexistence of solutions of \eqref{eq:wave} by imposing a positive distance between the vertical axis $x = 0$ and some points of $R\Sigma$ outside of a neighborhood of $0$. Moreover, we will show in Theorem~\ref{TheoExist} that \ref{HypoSigma-Exists} (resp.\ \ref{HypoSigma-ExistsInfty}) is a sufficient condition for the existence of solutions of \eqref{eq:wave} in $\mathsf X_p$ for $p$ finite (resp.\ $p = +\infty$) and the necessity of the linear growth condition (resp.\ the condition of mapping bounded sets to bounded sets).

Regarding uniqueness, we have a more precise result, namely that \ref{HypoSigma-Unique} (resp.\ \ref{HypoSigma-UniqueInfty}) is necessary and sufficient for $p$ finite (resp.\ for $p = +\infty$), as shown in Theorem~\ref{TheoUnique}. Note that a standard assumption in the literature for obtaining uniqueness of solutions of \eqref{eq:wave} is that $\sigma$ or $\id+\sigma$ are monotone, cf.\ for instance \cite[Proposition~1]{pierre2000strong} and also Proposition~\ref{PropIdPlusSigmaMonotone} below. Either of these properties implies conditions \ref{HypoSigma-Unique} and \ref{HypoSigma-UniqueInfty} on $\Sigma$.

Condition \ref{HypoSigma-Damping} is a generalization of the damping assumption on a function $\sigma$, which states that $s \sigma(s) \geq 0$ for every $s \in \mathbb R$, and which implies that the $\mathsf X_2$ norm of solutions of \eqref{eq:wave} is nonincreasing. Similarly, \ref{HypoSigma-StrictDamping} is a strict version of \ref{HypoSigma-Damping} and generalizes the condition of strict damping for a function $\sigma$, i.e., $s \sigma(s) > 0$ for every $s \in \mathbb R^\ast$.

Hypothesis \ref{HypoSigma-NoDampingAtInfty} is used in the sequel to show that the stability concept of UGAS does not hold in general in $\mathsf X_p$ for finite $p$
and it can be restated, in the case where $\Sigma$ is the graph of a continuous function $\sigma:\mathbb R\to\mathbb R$, as 
\[
\lim_{\abs{s} \to +\infty} \min\left(\frac{\sigma(s)}{s}, \frac{s}{\sigma(s)}\right)=0.
\]
In particular, Hypothesis \ref{HypoSigma-NoDampingAtInfty} is verified if $\sigma$ is either a saturation function (see Figure~\ref{FigSaturation} below) or has a superlinear growth at infinity.

Hypotheses \ref{HypoSigma-SectorZero} and \ref{HypoSigma-SectorInfty} are generalizations of linear sector conditions in neighborhoods of the origin and infinity respectively, which are classical in the case where $\Sigma$ is the graph of a continuous function $\sigma:\mathbb R\to\mathbb R$ and which can be stated in that case respectively as
\[
0 < \liminf_{s \to 0} \frac{\sigma(s)}{s} \leq \limsup_{s \to 0} \frac{\sigma(s)}{s} < +\infty
\]
and
\[
0 <  \liminf_{\abs{s} \to +\infty} \frac{\sigma(s)}{s} \leq \limsup_{\abs{s} \to +\infty} \frac{\sigma(s)}{s}< +\infty.
\]

We next translate these hypotheses in equivalent statements when one replaces $\Sigma$ by $R \Sigma$ and get the following proposition.

\begin{proposition}
\label{PropHTilde}
Let $\Sigma\subset \mathbb R^2$ and consider the list given in 
Hypotheses \ref{HypoSigma}. Then Hypothesis \ref{HypoSigma-Zero} is equivalent to the same statement when replacing $\Sigma$ by $R \Sigma$. As for \ref{HypoSigma-Damping}, \ref{HypoSigma-StrictDamping}, \ref{HypoSigma-NoDampingAtInfty}, \ref{HypoSigma-SectorZero}, \ref{HypoSigma-SectorInfty}, \ref{HypoSigma-gUpper}, and \ref{HypoSigma-gLower} they can be expressed in terms of $R \Sigma$ (or, equivalent, of $S$) respectively as follows:
\begin{itemize}[align=left, labelwidth=\widthof{\normalsize $\widetilde{\text{(H10)}}$\ }, leftmargin=!]
\item[\customlabel{HypoSigma-DampingTilde}{$\widetilde{\text{\ref{HypoSigma-Damping}}}$}] For every $x \in \mathbb R$ and $y \in S(x)$, one has $\abs{y} \leq \abs{x}$,

\item[\customlabel{HypoSigma-StrictDampingTilde}{$\widetilde{\text{\ref{HypoSigma-StrictDamping}}}$}] For every $x \in \mathbb R$ and $y \in S(x)$ with $(x, y) \neq (0, 0)$, one has $\abs{y} < \abs{x}$,

\item[\customlabel{HypoSigma-NoDampingAtInftyTilde}{$\widetilde{\text{\ref{HypoSigma-NoDampingAtInfty}}}$}] One has
\[
\lim_{\substack{\abs{(x, y)} \to +\infty \\ y \in S(x)}} \abs*{\frac{y}{x}} = 1.
\]

\item[\customlabel{HypoSigma-SectorZeroTilde}{$\widetilde{\text{\ref{HypoSigma-SectorZero}}}$}] $\Sigma$ satisfies \ref{HypoSigma-DampingTilde} and there exist $M > 0$ and $\mu \in (0, 1)$ such that
\[
\abs{y} \leq \mu \abs{x}, \qquad \text{ for every } (x, y) \in R \Sigma \cap B(0, M).
\]

\item[\customlabel{HypoSigma-SectorInftyTilde}{$\widetilde{\text{\ref{HypoSigma-SectorInfty}}}$}] $\Sigma$ satisfies \ref{HypoSigma-DampingTilde} and there exist $M > 0$ and $\mu \in (0, 1)$ such that
\[
\abs{y} \leq \mu \abs{x}, \qquad \text{ for every } (x, y) \in R \Sigma \setminus B(0, M).
\]

\item[\customlabel{HypoSigma-gUpperTilde}{$\widetilde{\text{\ref{HypoSigma-gUpper}}}$}] $\Sigma$ satisfies \ref{HypoSigma-DampingTilde} and there exist a positive constant $M$ and a function $Q \in \mathcal C^1(\mathbb R_+, \mathbb R_+)$ with $Q(0) = 0$, $0 < Q(x) < x$, and $Q^\prime(x) > 0$ for every $x>0$ such that
\[
\abs{y} \leq Q(\abs{x}), \qquad \text{ for every } (x, y) \in R \Sigma \cap B(0, M).
\]

\item[\customlabel{HypoSigma-gLowerTilde}{$\widetilde{\text{\ref{HypoSigma-gLower}}}$}] $\Sigma$ satisfies \ref{HypoSigma-DampingTilde} and there exist a positive constant $M$ and a function $Q \in \mathcal C^1(\mathbb R_+, \mathbb R_+)$ with $Q(0) = 0$, $0 < Q(x) < x$, and $Q^\prime(x) > 0$ for every $x>0$ such that
\[
\abs{y} \geq Q(\abs{x}), \qquad \text{ for every } (x, y) \in R \Sigma \cap B(0, M).
\]
\end{itemize}
\end{proposition}
\begin{proof} This proposition follows after immediate computations, after noticing though that the equivalences between \ref{HypoSigma-gUpper} and \ref{HypoSigma-gUpperTilde} and between \ref{HypoSigma-gLower} and \ref{HypoSigma-gLowerTilde} are obtained via the follow explicit relation between the functions $q$ and $Q$:
\begin{equation}
\label{Relation-Q-q}
Q(x) = \frac{1}{\sqrt{2}} \left[2(q + \id)^{-1}(\sqrt{2} x) - \sqrt{2} 
x\right],
\end{equation}
where indeed $x\mapsto x+q(x)$ is an increasing bijection from $\R_+$ to $\R_+$ when \ref{HypoSigma-gUpper} or \ref{HypoSigma-gLower} holds true.
\end{proof}

\begin{remark}
Hypotheses \ref{HypoSigma-gUpper} and \ref{HypoSigma-gLower} are borrowed from \cite{V-Martinez2000}, where they are only considered in the case where $\Sigma$ is the graph of a function. These hypotheses can be seen as nonlinear sector conditions, which can be clearly understood when written in terms of $R\Sigma$, see Figure~\ref{FigH10} below. Indeed, these conditions naturally arise when one considers, instead of $\Sigma$, the set 
$\Sigma^{-1} = \{(y, x) \suchthat (x, y) \in \Sigma\}$, which generalizes the situation where $\Sigma$ is the graph of an invertible function $\sigma$ and in which case $\Sigma^{-1}$ becomes the graph of $\sigma^{-1}$ and is obtained from $\Sigma$ by the symmetry with respect to the diagonal line $y = x$.
Instead, $R\Sigma^{-1}$ is obtained from $R\Sigma$ by the symmetry with respect to the axis $y = 0$ which corresponds to a simple sign change for $S$. Hence, all the issues related to iterated sequences associated with $S$ (such as existence, uniqueness and asymptotic behavior) remain unchanged when they are considered for $-S$. 
\end{remark}

We next provide figures illustrating the region $R \Sigma$ for different choices of $\Sigma$. The first one, provided in Figure~\ref{FigH4}, is an example of a set $\Sigma$ satisfying \ref{HypoSigma-Damping} together with the corresponding set $R \Sigma$, which satisfies \ref{HypoSigma-DampingTilde}.

\begin{figure}[ht]
\centering
\begin{tikzpicture}

\newcommand{\SigmaHQ}{(0, 0) -- (1, 0) to[out=0, in=-135] (1, 1) to[out=45, in=180] (2.5, 1) to[out=0, in=-90] (3, 2) -- (3, 3) -- (2, 3) to[out=180, in=0] (1.5, 2) -- (0.75, 2) to[out=-180, in=60] (0, 0) -- (-0.5, {-sqrt(3)/2}) to[out=-120, in=45] (-0.5, -2) -- (-1.5, -3) -- (-3, -3) -- (-3, -1.5) -- (-2, -0.5) to[out=45, in=-180] (-1, -0.5) to[out=0, in=-180] (0, 0)}

\fill[blue!20!white] (-3, 0) -- (3, 0) -- (3, 3) -- (0, 3) -- (0, -3) -- (-3, -3) -- cycle;

\fill[blue!60!white] \SigmaHQ;

\node at (2.25, 2) {$\Sigma$};

\draw[semithick, -Stealth] (-3, 0) -- (3, 0);
\draw[semithick, -Stealth] (0, -3) -- (0, 3);

\draw[semithick, -Stealth] (3.5, -0.25) -- node[midway, above] {$R$} (4.5, -0.25);

\fill[blue!20!white] (5, 3) -- (11, -3) -- (11, 3) -- (5, -3) -- cycle;

\begin{scope}[shift={(8, 0)}, rotate=-45]
\clip (0, {3*sqrt(2)}) -- ({3*sqrt(2)}, 0) -- (0, {-3*sqrt(2)}) -- ({-3*sqrt(2)}, 0) -- cycle;

\fill[blue!60!white] \SigmaHQ;
\end{scope}

\draw[semithick, -Stealth] (5, 0) -- (11, 0);
\draw[semithick, -Stealth] (8, -3) -- (8, 3);

\node at (9.5, 0.5) {$R \Sigma$};
\end{tikzpicture}
\caption{A set $\Sigma$ satisfying \ref{HypoSigma-Damping} and the corresponding set $R \Sigma$ satisfying \ref{HypoSigma-DampingTilde}.}
\label{FigH4}
\end{figure}
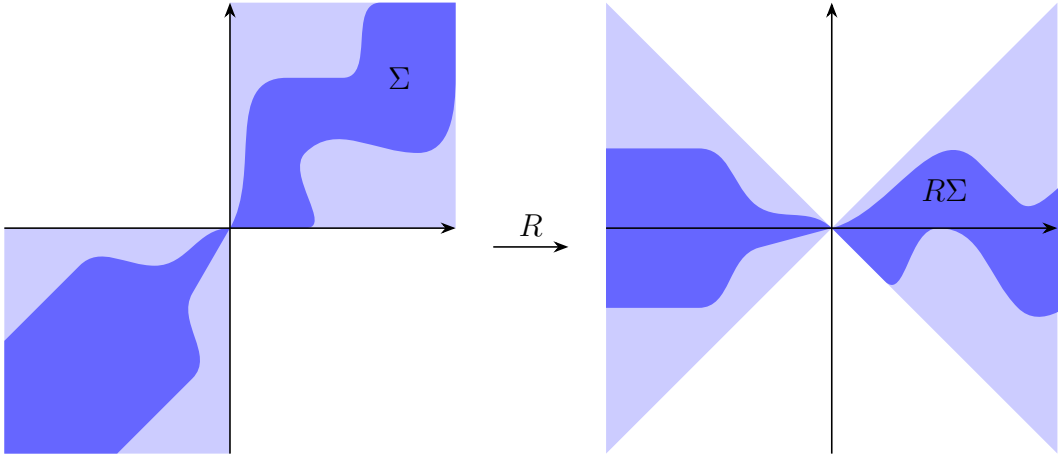

Figure~\ref{FigH10} provides the regions where a set $\Sigma$ must be included in order to satisfy either \ref{HypoSigma-gUpper} (in light blue) or \ref{HypoSigma-gLower} (in light red). The figure also represents the corresponding regions for $R \Sigma$ to satisfy \ref{HypoSigma-gUpperTilde} or \ref{HypoSigma-gLowerTilde}.

\begin{figure}[ht]
\centering
\begin{tikzpicture}

\fill[blue!20!white] (0, 0) to[out=0, in=-135] (3, 1.5) -- (3, 3) -- (1.5, 3) to[out=-135, in=90] (0, 0) to[out=180, in=45] (-3, -1.5) -- (-3, -3) -- (-1.5, -3) to[out=45, in=-90] (0, 0);
\fill[red!20!white] (0, 0) to[out=0, in=-135] (3, 1.5) -- (3, 0) -- cycle;
\fill[red!20!white] (1.5, 3) to[out=-135, in=90] (0, 0) -- (0, 3) -- cycle;
\fill[red!20!white] (0, 0) to[out=180, in=45] (-3, -1.5) -- (-3, 0) -- cycle;
\fill[red!20!white] (-1.5, -3) to[out=45, in=-90] (0, 0) -- (0, -3) -- cycle;
\draw[very thick, blue] (0, 0) to[out=0, in=-135] (3, 1.5) (1.5, 3) to[out=-135, in=90] (0, 0) to[out=180, in=45] (-3, -1.5) (-1.5, -3) to[out=45, in=-90] (0, 0);

\draw[semithick, -Stealth] (-3, 0) -- (3, 0);
\draw[semithick, -Stealth] (0, -3) -- (0, 3);

\draw[semithick, -Stealth] (3.5, -0.25) -- node[midway, above] {$R$} (4.5, -0.25);

\begin{scope}[shift={(8, 0)}, rotate=-45]
\clip (0, {3*sqrt(2)}) -- ({3*sqrt(2)}, 0) -- (0, {-3*sqrt(2)}) -- ({-3*sqrt(2)}, 0) -- cycle;

\fill[blue!20!white] (0, 0) to[out=0, in=-135] (3, 1.5) -- (3, 3) -- (1.5, 3) to[out=-135, in=90] (0, 0) to[out=180, in=45] (-3, -1.5) -- (-3, -3) -- (-1.5, -3) to[out=45, in=-90] (0, 0);
\fill[red!20!white] (0, 0) to[out=0, in=-135] (3, 1.5) -- (4.5, 0) -- cycle;
\fill[red!20!white] (1.5, 3) to[out=-135, in=90] (0, 0) -- (0, 4.5) -- cycle;
\fill[red!20!white] (0, 0) to[out=180, in=45] (-3, -1.5) -- (-4.5, 0) -- cycle;
\fill[red!20!white] (-1.5, -3) to[out=45, in=-90] (0, 0) -- (0, -4.5) -- cycle;
\draw[very thick, blue] (0, 0) to[out=0, in=-135] (3, 1.5) (1.5, 3) to[out=-135, in=90] (0, 0) to[out=180, in=45] (-3, -1.5) (-1.5, -3) to[out=45, in=-90] (0, 0);
\end{scope}

\draw[semithick, -Stealth] (5, 0) -- (11, 0);
\draw[semithick, -Stealth] (8, -3) -- (8, 3);

\draw (3, 1.5) node[above left] {$q$};
\draw (1.5, 3) node[below right] {$q^{-1}$};

\draw (11, 1.5) node[left] {$Q$};
\draw (11, -1.5) node[left] {$-Q$};

\end{tikzpicture}
\caption{Regions for generalized sector conditions for $\Sigma$ and $R \Sigma$.}
\label{FigH10}
\end{figure}
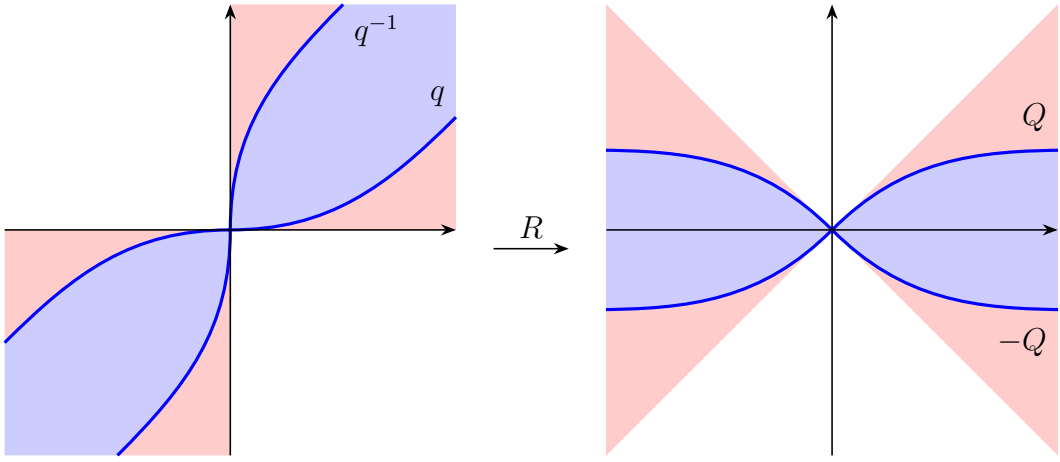

We represent in Figure~\ref{FigSaturation} a set $\Sigma$ as a saturation-type 
sector, i.e., a region comprised between the graphs of two piecewise linear saturation functions. The latter are defined as functions $f$ of the form $f(x) = \lambda x$ for $\abs{x} \leq M$ and $f(x) = \lambda M \frac{x}{\abs{x}}$ for $\abs{x} \geq M$, for some positive constants $\lambda$ and $M$. More generally,
we call \emph{saturation function} any continuous function whose graph is contained in a saturation-type sector. 

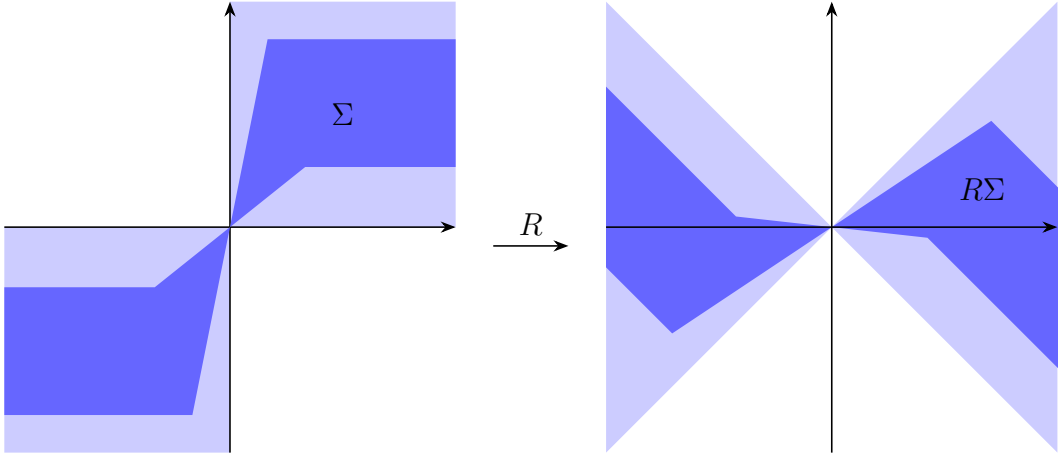
\begin{figure}[ht]
\centering
\begin{tikzpicture}

\fill[blue!20!white] (-3, 0) -- (3, 0) -- (3, 3) -- (0, 3) -- (0, -3) -- (-3, -3) -- cycle;

\draw[semithick, -Stealth] (-3, 0) -- (3, 0);
\draw[semithick, -Stealth] (0, -3) -- (0, 3);


\fill[blue!60!white] (0, 0) -- (1, 0.8) -- (3, 0.8) -- (3, 2.5) -- (0.5, 2.5) -- (-0.5, -2.5) -- (-3, -2.5) -- (-3, -0.8) -- (-1, -0.8) -- cycle;

\draw[semithick, -Stealth] (3.5, -0.25) -- node[midway, above] {$R$} (4.5, -0.25);

\fill[blue!20!white] (5, 3) -- (11, -3) -- (11, 3) -- (5, -3) -- cycle;

\begin{scope}[shift={(8, 0)}, rotate=-45]
\clip (0, {3*sqrt(2)}) -- ({3*sqrt(2)}, 0) -- (0, {-3*sqrt(2)}) -- ({-3*sqrt(2)}, 0) -- cycle;


\fill[blue!60!white] (0, 0) -- (1, 0.8) -- (4.5, 0.8) -- (4.5, 2.5) -- (0.5, 2.5) -- (-0.5, -2.5) -- (-4.5, -2.5) -- (-4.5, -0.8) -- (-1, -0.8) -- cycle;
\end{scope}

\draw[semithick, -Stealth] (5, 0) -- (11, 0);
\draw[semithick, -Stealth] (8, -3) -- (8, 3);

\node at (1.5, 1.5) {$\Sigma$};

\node at (10, 0.5) {$R\Sigma$};

\end{tikzpicture}
\caption{Saturation-type sectors.}
\label{FigSaturation}
\end{figure}

Finally, Figure~\ref{FigSign} represents the set $\Sigma_M$, $M > 0$, given by the graph of the sign set-valued map $\sign_M: \mathbb R \rightrightarrows \mathbb R$ defined by
\begin{equation}
\label{eq:defi-sign}
\sign_M(x) = \begin{dcases*}
\{-M\}, & if $x < 0$, \\
[-M, M], & if $x = 0$, \\
\{M\}, & if $x > 0$,
\end{dcases*}
\end{equation}
i.e.,
\begin{equation}
\label{eq:Sigma-Sign}
\Sigma_M = \left(\mathbb R_- \times \{-M\}\right) \cup \left(\{0\} \times [-M, M]\right) \cup \left(\mathbb R_+ \times \{M\}\right).
\end{equation}
Note that $\Sigma_M$ is not the graph of a single-valued function, but $R \Sigma_M$ is.

\begin{figure}[ht]
\centering
\begin{tikzpicture}

\fill[blue!20!white] (-3, 0) -- (3, 0) -- (3, 3) -- (0, 3) -- (0, -3) -- (-3, -3) -- cycle;

\draw[semithick, -Stealth] (-3, 0) -- (3, 0);
\draw[semithick, -Stealth] (0, -3) -- (0, 3);

\draw[very thick, blue] (-3, -1.5) -- (0, -1.5) -- (0, 1.5) -- (3, 1.5);

\draw[semithick, -Stealth] (3.5, -0.25) -- node[midway, above] {$R$} (4.5, -0.25);

\fill[blue!20!white] (5, 3) -- (11, -3) -- (11, 3) -- (5, -3) -- cycle;

\draw[semithick, -Stealth] (5, 0) -- (11, 0);
\draw[semithick, -Stealth] (8, -3) -- (8, 3);

\begin{scope}[shift={(8, 0)}, rotate=-45]
\clip (0, {3*sqrt(2)}) -- ({3*sqrt(2)}, 0) -- (0, {-3*sqrt(2)}) -- ({-3*sqrt(2)}, 0) -- cycle;

\draw[very thick, blue] (-3, -1.5) -- (0, -1.5) -- (0, 1.5) -- (3, 1.5);
\end{scope}

\end{tikzpicture}
\caption{Sign function.}
\label{FigSign}
\end{figure}
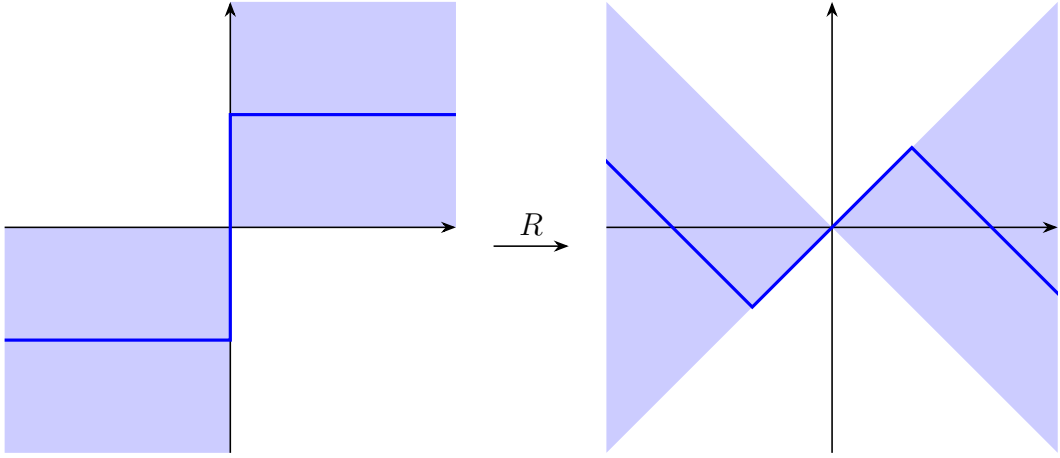

\subsection{Other models of wave equations with set-valued boun\-dary damping}

The flexibility of the viewpoint consisting of translating the study of solutions of \eqref{eq:wave}
into that of iterated sequences of the set-valued map $S$ is well illustrated when considering the wave equation
\begin{equation}\label{eq:wave-D+D}
\left\{
\begin{aligned}
& z_{tt}(t, x) = z_{xx}(t, x), & \qquad & (t, x) \in \mathbb R_+ \times [0, 1], \\
& (z_t(t, 0), z_x(t, 0)) \in \Sigma_0, & & t \in \mathbb R_+, \\
& (z_t(t, 1), -z_x(t, 1)) \in \Sigma_1, & & t \in \mathbb R_+, \\
& z(0, x) = z_0(x), & & x \in [0, 1], \\
& z_t(0, x) = z_1(x), & & x \in [0, 1],
\end{aligned}
\right.
\end{equation}
where $\Sigma_0$ and $\Sigma_1$ are subsets of $\mathbb R^2$. We have taken this example from \cite{V-Martinez2000} where precise decay rates of solutions of \eqref{eq:wave-D+D}, as time tends to infinity, have been given in the case where both $\Sigma_0$ and $\Sigma_1$ are graphs of functions $\sign(x)\abs{x}^{1+\alpha}$, $\alpha>0$. 

Similarly to Proposition~\ref{PropGn}, we aim at characterizing a correspondence between the solutions of \eqref{eq:wave-D+D} and the iterated sequences of a discrete-time dynamical system. For that purpose, we first provide the appropriate counterpart to Definition~\ref{DefiSeqI}.
\begin{definition}\label{DefiSeqII} 
For $p \in [1, +\infty]$, set $\mathsf Z_p = L^p(-1, 0)$.

\begin{enumerate}
\item We use $\Seq_0: L^p_{\loc}(-1, +\infty) \to \mathsf Z_p^{\mathbb N}$ to denote the bijection which associates, with each $h \in L^p_{\loc}(-1,\allowbreak +\infty)$, the sequence $(h_n)_{n \in \mathbb N}$ in $\mathsf Z_p$ defined by $h_n(s) = h(s + n)$ for $n \in \mathbb N$ and a.e.\ $s \in [-1, 0]$.

\item We use $\mathfrak I_2: \mathsf X_p\to \mathsf Z_p\times \mathsf Z_p$ to denote the isometry which associates, with each $(z_0, z_1) \in \mathsf X_p$, the element $(h_0,g_0) \in \mathsf Z_p\times \mathsf Z_p$ defined by
\[
\begin{pmatrix}
g_0(-s) \\
-h_0(s-1) \\
\end{pmatrix} = R \begin{pmatrix}
z_1(s) \\
-z_0^\prime(s) \\
\end{pmatrix}, \qquad \text{for a.e.\ } s \in [0, 1].
\]
\end{enumerate}
\end{definition}

We also assume in the sequel of this subsection that Definition~\ref{DefWeakSolution} is suitably modified in order to take into account the new boundary condition at $x = 0$ and the fact that $z(0, 0)$ is not necessarily zero, and that the contents of Definition~\ref{def:iterated} are extended to set-valued maps defined on $\mathbb{R}^2$ in order to deal with discrete-time dynamical systems on $\mathbb{R}^2$ and $\mathsf Z_p\times \mathsf Z_p$.

\begin{proposition}\label{prop:D+D}
Let $p \in [1, +\infty]$, $\Sigma_0, \Sigma_1$ be two subsets of $\mathbb R^2$, and $S_i: \mathbb R \rightrightarrows \mathbb R$ be the set-valued map whose graph is equal to $R \Sigma_i$ for $i \in \{0, 1\}$. Define the set-valued map $S$ on $\mathbb{R}^2$ which associates with every $(x,y)\in\mathbb{R}^2$ the Cartesian product of the subsets $(-S_1)(y)$ and $(-S_0)(x)$ of the real line.
\begin{enumerate}
\item\label{ItemEquivSols1-TwoBoundariesD+D} Let $z$ be a solution of \eqref{eq:wave-D+D} with initial condition $(z_0, z_1) \in \mathsf X_p$, 
$f \in L_{\loc}^p(0, +\infty)$ and $g \in L_{\loc}^p(-1, +\infty)$ be the corresponding functions from Definition~\ref{DefWeakSolution}, and the sequences $(h_n)_{n \in \mathbb N} = \Seq_0(f(\cdot+1))$ and $(g_n)_{n \in \mathbb N} = \Seq_0(g)$. Then $(h_n, g_n)_{n \in \mathbb N}$ is an iterated sequence for $S$ on $\mathsf Z_p \times \mathsf Z_p$ starting at $(h_0,g_0) = \mathfrak I_2(z_0, z_1)$.
\item\label{ItemEquivSols2-TwoBoundariesD+D} Conversely, let $(h_n,g_n)_{n \in \mathbb N}$ be an iterated sequence for $S$ on $\mathsf Z_p\times \mathsf Z_p$ starting at some $(h_0,g_0) \in\mathsf Z_p\times \mathsf Z_p$. Let
$f(\cdot+1)=\Seq_0^{-1}((h_n)_{n \in \mathbb N})$ and $g = \Seq^{-1}((g_n)_{n \in \mathbb N})$. Then $f \in L^p_{\loc}(0, +\infty)$, $g\in L^p_{\loc}(-1, +\infty)$
and, if $z$ is defined from $f$ and $g$ as in the first equation of \eqref{eq:defWeakSol}, one gets that $z$ is a solution of \eqref{eq:wave-D+D} in $\mathsf X_p$ with initial condition $(z_0, z_1) = \mathfrak I_2^{-1}(h_0,g_0)$.

\item Let $z$, $f$, $g$, and $(h_n, g_n)_{n \in \mathbb N}$ be as in \ref{ItemEquivSols1-TwoBoundariesD+D} or \ref{ItemEquivSols2-TwoBoundariesD+D}. Then, for all $t \in \mathbb R_+$,
\begin{align*}
e_p^p(z)(t) & = \norm{f(t+1 + \cdot)}_{\mathsf Z_p}^p + \norm{g(t + \cdot)}_{\mathsf Z_p}^p, \qquad \text{ if } p < +\infty, \\
e_\infty(z)(t) & = \max(\norm{f(t+1 + \cdot)}_{\mathsf Z_\infty}, \norm{g(t + \cdot)}_{\mathsf Z_\infty}),
\end{align*}
and, in particular, for all $n \in \mathbb N$,
\begin{align*}
e_p^p(z)(n) & = \norm{h_n}_{\mathsf Z_p}^p + \norm{g_n}_{\mathsf Z_p}^p, \qquad \text{ if } p < +\infty, \\
e_\infty(z)(n) & = \max(\norm{h_n}_{\mathsf Z_\infty}, \norm{g_n}_{\mathsf Z_\infty}).
\end{align*}
\end{enumerate}
\end{proposition}

\begin{proof}
The proof for the proposition exactly follows the line of arguments that led to Proposition~\ref{PropGn}. The only difference appears when replacing the Dirichlet boundary condition at $x=0$ in \eqref{eq:wave} by the boundary condition at $x=0$ given by $(z_t(t, 0), z_x(t, 0)) \in \Sigma_0$ in \eqref{eq:wave-D+D}.
One then replaces by \eqref{relation-f-g} and \eqref{recurrence-g} by the equations
\begin{equation}\label{eq:relationsD+D}
-g(t)\in S_0(f(t)),\quad -f(t+1) \in S_1(g(t-1)), \quad \text{ for every }t\geq 0,
\end{equation}
respectively.
\end{proof}

At the light of the above proposition, one can see that the issues of existence and uniqueness of solutions of \eqref{eq:wave-D+D} boil down to the study of the set-valued map $S$, the latter question being equivalent to the separate study 
of $S_0$ and $S_1$. On the other hand, asymptotic stability of solutions of \eqref{eq:wave-D+D} and related issues are addressed through the study of iterated sequences associated with $S$ in $\mathbb{R}^2$. Due to the structure of $S$, the latter question can be reduced to the study of real iterated sequences associated with the set valued maps $(-S_1)\circ (-S_0)$ or $(-S_0)\circ (-S_1)$ since, combining the two equations of \eqref{eq:relationsD+D} yields that $(h_{n})_{n\in\mathbb N}$ is an iterated sequence for $(-S_1) \circ (-S_0)$ in $\mathsf Z_p$ while $(g_{n})_{n\in\mathbb N}$ is an iterated sequence for $(-S_0) \circ (-S_1)$ in $\mathsf Z_p$.

When $S_0$ is single-valued, the sequence $(h_n)_{n \in \mathbb N}$ suffices to describe solutions $z$ of \eqref{eq:wave-D+D}, since the corresponding sequence $(g_n)_{n \in \mathbb N}$ is uniquely determined by the first relation of \eqref{eq:relationsD+D} and the second relation in \eqref{eq:relationsD+D} can be expressed solely in terms of $(h_n)_{n \in \mathbb N}$. If moreover the single-valued map $S_0$ is invertible, we may alternatively describe solutions $z$ of \eqref{eq:wave-D+D} in terms of $(g_n)_{n \in \mathbb N}$ only, since $(h_n)_{n \in \mathbb N}$ can be computed using the first relation of \eqref{eq:relationsD+D}. This is precisely what is done for \eqref{eq:wave} in this paper, in which case $S_0 = \id$ and we can use this simple expression of $S_0$ to further simplify the relation between $e_p(z)(\cdot)$ and the norms of the elements of the sequence $(g_n)_{n \in \mathbb N}$. A similar remark applies when $S_1$ is single-valued.

\begin{remark} Any boundary condition involving only $z_t$ and $z_x$ at the same endpoint can be recovered by our formalism. Indeed, an homogeneous Neumann condition reads $\Sigma$ equal to the horizontal axis $\mathbb R \times \{0\}$, while an homogeneous Dirichlet condition can be seen as taking $\Sigma$ equal to the vertical axis $\{0\} \times \mathbb R$. For instance, proceeding in such a way at the extremity $x=0$ with an homogeneous Dirichlet boundary condition yields that $S_0 = \id$, from which we recover the set-valued map $S$ of Proposition~\ref{PropGn} as equal to $S_1$.
\end{remark}

\subsection{Additional auxiliary results}

We now present some important auxiliary results providing properties of the solutions of \eqref{eq:wave} and of the set $\Sigma$ that will be used several times in the paper. We start with the following result on a decomposition of solutions of \eqref{eq:wave}, whose proof is immediate.

\begin{proposition}
\label{prop:dis-sup}
Let $p \in [1, +\infty]$, $\Sigma \subset \mathbb R^2$, and $S$ be the set-valued function whose graph is $R \Sigma$. Assume that $S(0) = \{0\}$.

Let $(f^{(k)})_{k \in \mathbb N}$ be a sequence in $\mathsf Y_p$ so that $f^{(k)}$ and $f^{(k')}$ have disjoint supports for every integers $k\neq k'$ and assume moreover that $f = \sum_{k\geq 0} f^{(k)}$ belongs to $\mathsf Y_p$. Hence, for every $n \geq 0$,
\begin{equation}\label{eq:dis-sup}
S^{[n]}(f(s)) = \sum_{k\geq 0} S^{[n]}(f^{(k)}(s)), \quad \text{ for a.e.\ } s\in [-1,1]. 
\end{equation}
Moreover, if $(f_n)_{n\geq 0}$ is any iterated sequence for $S$ starting at $f$, then, for every $k \geq 0$, there exists an iterated sequence $(f_n^{(k)})_{n \in \mathbb N}$ for $S$ starting at $f^{(k)}$ such that $f_n=\sum_{k\geq 0}f_n^{(k)}$ for every $n \geq 0$. In particular, if $p$ is finite, it holds
\begin{equation}\label{eq:dis-sup-energie-pfini}
\norm{f_n}^p_p = \sum_{k\geq 0}\norm{f_n^{(k)}}^p_p,
\end{equation}
and, if $p=+\infty$, one has 
\begin{equation}\label{eq:dis-sup-energie-pinfini}
\norm{f_n}_{\infty}=\sup_{k\geq 0}\norm{f_n^{(k)}}_{\infty}.
\end{equation}
\end{proposition}

The next proposition gathers some elementary properties of set-valued dampings, as given in Definition~\ref{def-damping}.

\begin{proposition}
\label{PropSDamping}
Let $S: \mathbb R \rightrightarrows \mathbb R$ be a damping.
\begin{enumerate}
\item\label{ItemS0} $S(0) = \{0\}$ and, if $S$ is a strict damping, then $x \in S(x)$ if and only if $x = 0$.
\item\label{ItemSNonEmpty} For every $x \in \mathbb R$, $S(x)$ is nonempty.
\item\label{ItemSxCompact} If $S$ is closed, then $S(x)$ is compact for every $x \in \mathbb R$.
\item\label{ItemCompositionDamping} If $T$ is a damping, then $S \circ T$ is also a damping, and it is strict if $S$ or $T$ is strict. In particular, the iterates $S^{[n]}$, $n \in \mathbb N^\ast$, are also dampings, and they are strict if $S$ is strict.
\item\label{ItemS2IteratedSequence} Assume that, for every real iterated sequence $(x_n)_{n \in \mathbb N}$ in $\mathbb R$ for $S$ which is not identically zero, there exists $n \in \mathbb N$ such that $\abs{x_n} < \abs{x_0}$. Then $S^{[2]}$ is a strict damping.
\item\label{ItemS2StrictDamping} $S^{[2]}$ is a strict damping if and only if there does not exist $x \in \mathbb R^\ast$ such that either \begin{enumerate*}[label={(\roman*)}]\item $x \in S(x)$ or \item $-x \in S(x)$ and $x \in S(-x)$.\end{enumerate*}
\end{enumerate}
\end{proposition}

\begin{proof}
By \ref{HypoSigma-Zero}, one has $(0,0)\in R\Sigma$ and thus $0 \in S(0)$ since the graph of $S$ is $R\Sigma$. One can deduce easily that \ref{HypoSigma-DampingTilde} implies $S(0) = \{0\}$. The second part of \ref{ItemS0} follows immediately from \ref{HypoSigma-StrictDampingTilde}. The fact that $R \Sigma$ contains the graph of a function, which is a consequence of \ref{HypoSigma-Exists}, also immediately implies \ref{ItemSNonEmpty}. The closedness of $S$ implies that, for every $x \in \mathbb R$, $\left(\{x\} \times \mathbb R\right) \cap R \Sigma$ is closed, and \ref{HypoSigma-DampingTilde} yields that this set is also bounded, showing that it is compact. Then $S(x)$ is compact as the image of $(\{x\} \times \mathbb R) \cap R \Sigma$ through the projection onto the second coordinate, proving \ref{ItemSxCompact}.

To prove \ref{ItemCompositionDamping}, notice that $S \circ T$ trivially satisfies \ref{HypoSigma-Zero}, and \ref{HypoSigma-DampingTilde} for $S \circ T$ follows immediately from the corresponding properties for $S$ and $T$, with $S \circ T$ satisfying \ref{HypoSigma-StrictDampingTilde} as soon as one of $S$ or $T$ satisfies this assumption. Finally, if $\varphi_S: \mathbb R \to \mathbb R$ and $\varphi_T: \mathbb R \to \mathbb R$ are universally measurable functions with linear growth contained in the graphs of $S$ and $T$, respectively, one immediately verifies that $\varphi_S \circ \varphi_T$ is a function with linear growth contained in the graph of $S\circ T$. This function is also universally measurable as the composition of two universally measurable functions (see Proposition~\ref{PropCompositionUnivMeas} in Appendix~\ref{AppUniversally}), showing that $S \circ T$ also satisfies \ref{HypoSigma-Exists}, as required.

Let us show \ref{ItemS2IteratedSequence} by contraposition. By \ref{ItemCompositionDamping}, $S^{[2]}$ is a damping and, if it is not strict, then there exists $x \in \mathbb R^\ast$, $y \in S(x)$, and $z \in S(y)$ such that $\abs{z} = \abs{y} = \abs{x}$. If $x = y$ or $y = z$, then the sequence $(x_n)_{n \in \mathbb N}$ defined by $x_n = y$ for every $n \in \mathbb N$ is a real iterated sequence for $S$ with $\abs{x_n} = \abs{x_0}$ for every $n \in \mathbb N$. Otherwise, one has $z = -y = x$ and we define the sequence $(x_n)_{n \in \mathbb N}$ by $x_n = (-1)^n x$ for $n \in \mathbb N$. It is then clearly a real iterated sequence for $S$ with $\abs{x_n} = \abs{x_0}$ for every $n \in \mathbb N$, leading thus to the proof of the desired result.

Finally, to prove \ref{ItemS2StrictDamping}, notice that, if the damping $S^{[2]}$ is not strict, then, letting $x, y, z$ be as in the proof of \ref{ItemS2IteratedSequence}, the previous argument shows that one has either $y \in S(y)$ or $z = -y = x$, in which case $-x \in S(x)$ and $x \in S(-x)$, as required. Conversely, if there exists $x \in \mathbb R^\ast$ such that $x \in S(x)$, or such that $-x \in S(x)$ and $x \in S(-x)$, then, in both cases, $x \in S^{[2]}(x)$, showing that $S^{[2]}$ is not a strict damping.
\end{proof}

In the case where $\Sigma$ is the graph of a continuous function, items \ref{ItemS0}--\ref{ItemSxCompact} have been already obtained in \cite[Lemma~1]{pierre2000strong}.

We conclude this section with a technical result providing a necessary and sufficient condition on the set-valued map $S$ to be the graph of a continuous function when $\Sigma$ is the graph of a continuous function.

\begin{proposition}
\label{PropIdPlusSigmaMonotone}
Assume that $\Sigma$ is the graph of a continuous function $\sigma:\mathbb R\to\mathbb R$. Then the set-valued function $S$ is single-valued and continuous if and only if $\id+\sigma:\mathbb{R}\to \mathbb{R}$ is a bijection. Moreover, if $\sigma$ is a damping function, then $S$ is single-valued and continuous if and only if $\id + \sigma$ is strictly monotone.
\end{proposition}

\begin{proof}
Since $\Sigma = \{(x, \sigma(x)) \suchthat x \in \mathbb R\}$, one gets that
\[
R \Sigma = \left\{\left(T(x), T(x) - \sqrt{2} x\right) \suchthat x \in \mathbb R\right\}
\]
with $T = \frac{\id + \sigma}{\sqrt{2}}$.

If $\id + \sigma$ is a bijection, then $T$ is invertible and thus $R \Sigma$ is the graph of the continuous function $x \mapsto x - \sqrt{2} T^{-1}(x)$ defined on $\mathbb R$. Conversely, assuming that $R \Sigma$ is the graph of a single-valued function $\varphi: \mathbb R \to \mathbb R$, one immediately deduces that $T: \mathbb R \to \mathbb R$ is surjective. If $x_1, x_2 \in \mathbb R$ are such that $T(x_1) = T(x_2)$, then $T(x_1) - \sqrt{2} x_1 = \varphi(T(x_1)) = \varphi(T(x_2)) = T(x_2) - \sqrt{2} x_2$, implying that $x_1 = x_2$. Thus, $T$ is also injective. Hence $T$ is bijective, as required.

The last assertion of the proposition follows from the fact that, if $\sigma$ is a damping function, then
\[\lim_{s \to -\infty} (\id + \sigma)(s) = -\infty \qquad \text{ and } \qquad \lim_{s \to +\infty} (\id + \sigma)(s) = +\infty,\]
showing that $\id + \sigma$ is surjective. Hence $\id + \sigma$ is bijective if and only if it is strictly monotone.
\end{proof}

\section{Existence and uniqueness of solutions in \texorpdfstring{$\mathsf X_p$}{Xp}}
\label{SecExistUnique}

In this section, we will show that the hypotheses \ref{HypoSigma-Exists}--\ref{HypoSigma-UniqueInfty} introduced in Hypotheses \ref{HypoSigma} actually yield necessary and sufficient conditions for existence and uniqueness of solutions of \eqref{eq:wave} in $\mathsf X_p$, $p \in [1, +\infty]$. We start with the following existence result.

\begin{theorem}
\label{TheoExist}
Let $\Sigma \subset \mathbb R^2$ and $p \in [1, +\infty]$.
\begin{enumerate}
\item\label{ItemTheoExistDirect} If \ref{HypoSigma-Exists} holds and $p < +\infty$, or if \ref{HypoSigma-ExistsInfty} holds and $p = +\infty$, then, for every $(z_0, z_1) \in \mathsf X_p$, there exists a solution of \eqref{eq:wave} in $\mathsf X_p$ with initial condition $(z_0, z_1)$.

\item\label{ItemTheoExistConverse} Assume that, for every $(z_0, z_1) \in \mathsf X_p$, there exists a solution of \eqref{eq:wave} in $\mathsf X_p$ with initial condition $(z_0, z_1)$. Then $R \Sigma$ contains the graph of a Lebesgue measurable function. Moreover, $R \Sigma$ also contains the graph of a function with linear growth if $p < +\infty$ or the graph of a function mapping bounded sets to bounded sets if $p = +\infty$.
\end{enumerate}
\end{theorem}

\begin{proof}
According to Remark~\ref{RemkExistUnique}, the existence, for every initial condition $(z_0, z_1) \in \mathsf X_p$, of a solution of \eqref{eq:wave} in $\mathsf X_p$ with initial condition $(z_0, z_1)$, is equivalent to the following statement: for every $g \in \mathsf Y_p$, there exists $h \in \mathsf Y_p$ such that $h(s) \in S(g(s))$ for a.e.\ $s \in [-1, 1]$, where $S$ denotes the set-valued map whose graph is $R \Sigma$.

We start by showing Item \ref{ItemTheoExistDirect}. It is clear that \ref{HypoSigma-Exists} (resp.\ \ref{HypoSigma-ExistsInfty}) is sufficient to get the statement for $p$ finite (resp.\ for $p = +\infty$) by taking $h = \varphi \circ g$, where $\varphi$ is the function whose existence is asserted in \ref{HypoSigma-Exists} (resp.\ \ref{HypoSigma-ExistsInfty}). Indeed, since $\varphi$ is universally measurable and using Proposition~\ref{PropUnivMeasurable} in Appendix~\ref{AppUniversally}, one gets that $h$ is measurable, and the linear growth assumption (resp.\ the assumption of mapping bounded sets to bounded sets) on $\varphi$ guarantees that $h \in \mathsf Y_p$.

We next turn to an argument for Item \ref{ItemTheoExistConverse}. Assume that, for every $g \in \mathsf Y_p$, there exists $h \in \mathsf Y_p$ such that $h(s) \in S(g(s))$ for a.e.\ $s \in [-1, 1]$. Notice first that, for every $x \in \mathbb R$, $S(x)$ is nonempty. Indeed, given $x \in \mathbb R$, consider $g$ identically equal to $x$ and apply the working hypothesis on $g$. Then clearly $h(s) \in S(x)$ for a.e.\ $s \in [-1, 1]$, and hence $S(x)$ is nonempty.

Let $g: (-1, 1) \to \mathbb R$ be a diffeomorphism. Thanks to Lemma~\ref{LemmaYpMeasurable} in Appendix~\ref{AppLemmas}, there exists a measurable function $h: (-1, 1) \to \mathbb R$ such that $h(s) \in S(g(s))$ for a.e.\ $s \in [-1, 1]$. Let $\varphi = h \circ g^{-1}$, which is measurable since $g$ is a diffeomorphism. Then $\varphi(x) \in S(x)$ for a.e.\ $x \in \mathbb R$, and this inclusion can be made to hold everywhere up to modifying $\varphi$ on set of measure zero.

We finally prove the last parts of the statement regarding linear growth for $1\leq p < +\infty$ and the condition of mapping bounded sets to bounded sets for $p = +\infty$. Reasoning by contradiction yields the existence of a sequence $(x_n)_{n \in \mathbb N}$ such that, for $p < +\infty$, one has
\begin{equation}
\label{ContradictionExistence}
\abs{y} > n (\abs{x_n} + 1) \qquad \text{ for every } n \in \mathbb N \text{ and } y \in S(x_n),
\end{equation}
while, for $p = +\infty$, $(x_n)_{n \in \mathbb N}$ is bounded and
\begin{equation}
\label{ContradictionExistenceInfty}
\abs{y} > n \qquad \text{ for every } n \in \mathbb N \text{ and } y \in S(x_n).
\end{equation}
Let $\{A_n\}_{n \in \mathbb N}$ be a family of disjoint measurable subsets of $[-1, 1]$ of positive Lebesgue measure $\alpha_n$. In the case $p < +\infty$, we further require that $\alpha_n = \frac{1}{(\abs{x_n} + 1)^p (n + 1)^2}$, which is possible since $\sum_{n=0}^\infty \alpha_n \leq 2$. Consider the measurable function $g = \sum_{n = 0}^\infty x_n \chi_{A_n}$. Since one has that
\[
\norm{g}_p^p = \sum_{n=0}^\infty \alpha_n \abs{x_n}^p \quad \text{ if } p < +\infty \qquad \text{ and } \qquad \norm{g}_\infty = \sup_{n \in \mathbb N} \abs{x_n},
\]
then $g \in \mathsf Y_p$. Let $h \in \mathsf Y_p$ be such that $h(s) \in S(g(s))$ for a.e.\ $s \in [-1, 1]$. Then, for $p < +\infty$, one deduces by \eqref{ContradictionExistence} that $\abs{h(s)} > n (\abs{x_n} + 1)$ for every $n \in \mathbb N$ and a.e.\ $s \in A_n$ and therefore
\[
\norm{h}_p^p \geq \sum_{n=0}^\infty \alpha_n n^p (\abs{x_n} + 1)^p = \sum_{n=0}^\infty \frac{n^p}{(n + 1)^2} \geq \sum_{n=0}^\infty \frac{n}{(n + 1)^2} = +\infty,
\]
which contradicts the fact that $h \in \mathsf Y_p$. Similarly, for $p = +\infty$, one gets by \eqref{ContradictionExistenceInfty} that $\abs{h(s)} > n$ for every $n \in \mathbb N$ and a.e.\ $s \in A_n$ and, since $A_n$ has positive measure for every $n \in \mathbb N$, one has $\norm{h}_\infty = +\infty$, yielding the required contradiction also in that case.
\end{proof}
 
We next provide our result on the uniqueness of solutions of \eqref{eq:wave}.

\begin{theorem}
\label{TheoUnique}
Let $\Sigma \subset \mathbb R^2$ and $p \in [1, +\infty]$. Then, for every $(z_0, z_1) \in \mathsf X_p$, there exists a unique solution of \eqref{eq:wave} in $\mathsf X_p$ with initial condition $(z_0, z_1)$ if and only if either \ref{HypoSigma-Unique} holds and $p < +\infty$, or \ref{HypoSigma-UniqueInfty} holds and $p = +\infty$.
\end{theorem}

\begin{proof}
As in the proof of Theorem~\ref{TheoExist}, we use Remark~\ref{RemkExistUnique} to equivalently reformulate the existence and uniqueness, for every initial condition $(z_0, z_1) \in \mathsf X_p$, of a solution of \eqref{eq:wave} in $\mathsf X_p$ with initial condition $(z_0, z_1)$, as the following statement: for every $g \in \mathsf Y_p$, there exists a unique $h \in \mathsf Y_p$ such that $h(s) \in S(g(s))$ for a.e.\ $s \in [-1, 1]$.

We first suppose that $\ref{HypoSigma-Unique}$ or \ref{HypoSigma-UniqueInfty} holds. Notice that \ref{HypoSigma-Unique} implies \ref{HypoSigma-Exists} and \ref{HypoSigma-UniqueInfty} implies \ref{HypoSigma-ExistsInfty} and hence, assuming either \ref{HypoSigma-Unique} and $p < +\infty$ or \ref{HypoSigma-UniqueInfty} and $p = +\infty$, one has, from Theorem~\ref{TheoExist}, existence of solutions to \eqref{eq:wave}. Moreover, both \ref{HypoSigma-Unique} and \ref{HypoSigma-UniqueInfty} imply that $S$ is the graph of a function $\varphi: \mathbb R \to \mathbb R$ and hence having $h(s) \in S(g(s))$ for a.e.\ $s \in [-1, 1]$ is equivalent to having $h = \varphi \circ g$, which uniquely determines $h$.

Conversely, by Theorem~\ref{TheoExist}\ref{ItemTheoExistConverse}, $R \Sigma$ contains the graph of a measurable function $\varphi: \mathbb R \to \mathbb R$. If $R \Sigma$ were not equal to the graph of $\varphi$, there would exist $x \in \mathbb R$ such that $S(x)$ contains more than one element, and thus, by considering the initial condition $g$ constant equal to $x$, one would construct two different solutions to \eqref{eq:wave}, contradicting the uniqueness assumption. Hence $R \Sigma$ is equal to the graph of $\varphi$. Applying once again Theorem~\ref{TheoExist}\ref{ItemTheoExistConverse}, one deduces that $\varphi$ is necessarily a function with linear growth in the case $p < +\infty$ or a function mapping bounded sets to bounded sets in the case $p = +\infty$.

One is left to show that $\varphi$ is universally measurable. This follows by first using Lemma~\ref{LemmaYpMeasurable} from Appendix~\ref{AppLemmas} to conclude that $\varphi$ preserves Lebesgue measurability by left composition and then Proposition~\ref{PropUnivMeasurable} from Appendix~\ref{AppUniversally}.
\end{proof}

\begin{remark}
It is clear from Theorems~\ref{TheoExist} and \ref{TheoUnique} that existence and uniqueness of solutions of \eqref{eq:wave} in $\mathsf X_p$ solely depends on the fact that the set $R \Sigma$ contains (or is equal to) the graph of a function with appropriate properties. Thanks to this fact, it is possible to obtain existence and uniqueness of solutions of \eqref{eq:wave} for a large variety of sets $\Sigma$. For instance, this allows us to consider $\Sigma$ as graphs of discontinuous functions completed by vertical segments at jump discontinuities without facing the usual issues addressed in \cite{Filippov1988Differential, Xu2019Saturated}. An example of the above discussion is the sign function, illustrated in Figure~\ref{FigSign} and treated in more details in Section~\ref{SecSign}.
\end{remark}

\begin{remark}
Hypothesis~\ref{HypoSigma-Exists} falls short of being necessary for existence of solutions of \eqref{eq:wave} in $\mathsf X_p$ for $p$ finite. Indeed, it would be the case if not only the measurable function and that of linear growth provided by Theorem~\ref{TheoExist}\ref{ItemTheoExistConverse} would be equal to the same function $\varphi$, but also if this function $\varphi$ were universally measurable. Note that the further assumption of uniqueness of solutions of \eqref{eq:wave}, together with Proposition~\ref{PropUnivMeasurable}, implies such properties on $\varphi$ and yields Theorem~\ref{TheoUnique}. We conjecture that \ref{HypoSigma-Exists} is actually necessary for the existence of solutions of \eqref{eq:wave} in $\mathsf X_p$ for $p$ finite. Similar comments can be made on \ref{HypoSigma-ExistsInfty} in the case $p = +\infty$.
\end{remark}

\begin{remark}
The notion of solution of \eqref{eq:wave} introduced in Definition~\ref{DefWeakSolution} covers only the case of solutions which are global in time. One can easily adapt Definition~\ref{DefWeakSolution} to allow for solutions of \eqref{eq:wave} defined locally in time: given $T > 0$, it suffices to require that $f \in L^p_{\loc}(0, T+1)$ and $g \in L^p_{\loc}(0, T)$ and that \eqref{eq:defWeakSol} is satisfied for $t \in [0, T)$ instead of for $t \in \mathbb R_+$. The local formulation allows one to also consider solutions of \eqref{eq:wave}  that blow up in finite time, but this topic is outside the scope of the present paper, which focuses instead on the convergence to zero of solutions as time tends to infinity.

Theorem~\ref{TheoExist} concerns the existence of global solutions of \eqref{eq:wave}, but, since one does not necessarily have uniqueness of solutions, there may also exist local solutions blowing up in finite time under the assumptions of Theorem~\ref{TheoExist}\ref{ItemTheoExistDirect}. For finite $p$, one may prevent the existence of such local solutions by requiring the linear growth assumption on $S$ instead of only on the universally measurable function $\varphi$ from \ref{HypoSigma-Exists}, i.e., by requiring that there exist $a, b \in \mathbb R_+$ such that $\norm{S(x)} \leq a \abs{x} + b$ for every $x \in \mathbb R$. Under this assumption, any local solution of \eqref{eq:wave} in $\mathsf X_p$ can be extended to a global solution. A similar remark holds in the case $p = +\infty$ by requiring $\bigcup_{x \in A} S(x)$ to be bounded for every bounded set $A \subset \mathbb R$ (which is in particular satisfied when $S$ has linear growth).

Notice that the linear growth condition on $S$ is satisfied under \ref{HypoSigma-Unique} (and, similarly, in the case $p = +\infty$, the condition of $S$ being bounded on bounded sets is satisfied under \ref{HypoSigma-UniqueInfty}), meaning that, given an initial condition, the unique solution of \eqref{eq:wave} from Theorem~\ref{TheoUnique} is unique not only on the class of global solutions, but also on the class of local solutions. In the sequel of the paper, we will mostly often work with sets $\Sigma \subset \mathbb R^2$ which are dampings, in which case \ref{HypoSigma-Unique} or \ref{HypoSigma-UniqueInfty} are not necessarily satisfied, but the linear growth assumption trivially holds due to \ref{HypoSigma-DampingTilde}.
\end{remark}

\begin{remark}
Since \ref{HypoSigma-Unique} is independent of $p$, it follows from Theorem~\ref{TheoUnique} that existence and uniqueness of solutions of \eqref{eq:wave} in $\mathsf X_p$ for any initial condition in $\mathsf X_p$ for \emph{some} finite $p$ is equivalent to having the same property for \emph{every} finite $p$. However, existence and uniqueness in the case $p = +\infty$ are equivalent to the weaker assumption \ref{HypoSigma-UniqueInfty}. A way to interpret this fact is to remark that the linear growth assumption from \ref{HypoSigma-Unique} is designed to avoid blow-up in finite time of solutions  of \eqref{eq:wave} in $L^p$ norm for finite $p$, but, for bounded initial conditions, this blow-up phenomenon can be prevented with the weaker assumption from \ref{HypoSigma-UniqueInfty}.
\end{remark}

\begin{remark}
\label{RemkExistSimple}
If $g_0$ is a simple function (i.e., a function whose range is a finite set), one may wonder whether there exists a solution of \eqref{eq:wave} with initial condition $\mathfrak I^{-1}(g_0)$. It turns out that a necessary and sufficient condition for the existence of such a solution for every simple function $g_0$ is that $S$ is a multi-valued function (i.e., $S(x) \neq \emptyset$ for every $x \in \mathbb R$). Indeed, in this case, starting from a simple function $g_0$, one builds at once a sequence $(g_n)_{n \in \mathbb N}$ satisfying \eqref{eq:dynSystGn} where, for every $n \geq 0$, $g_n$ is a simple function which is constant on every set on which $g_0$ is constant.

The same question may be asked in the more general case where $g_0 \in \mathsf Y_p$ has countable range. The necessary and sufficient condition for the existence of solutions  of \eqref{eq:wave} starting at $\mathfrak I^{-1}(g_0)$ for every such $g_0$ is simply now that $R \Sigma$ contains the graph of a function with linear growth in the case $p < +\infty$ or that of a function mapping bounded sets to bounded sets in the case $p = +\infty$.

Both statements follow immediately from the arguments provided in the proof of Theorem~\ref{TheoExist}.
\end{remark}

\section{Asymptotic behavior}
\label{SecAsymptotic}

After having provided a suitable notion of weak solution  of \eqref{eq:wave} and established conditions for existence and uniqueness of corresponding solutions, we consider in this section their asymptotic behavior. We start in Subsection~\ref{SecAsymptoticBasic} by identifying \ref{HypoSigma-Damping} as a necessary and sufficient condition for the energy of a solution to be nonincreasing and providing suitable definitions of stability. We then provide necessary and sufficient conditions for these notions of stability in Subsection~\ref{SecRealIterated} in terms of the behavior of real iterated sequences for the set-valued map $S$ whose graph is $R \Sigma$, and relate them to properties of $S^{[2]}$ in Subsection~\ref{SecS2}. A detailed study of the decay rates of solutions  of \eqref{eq:wave} and their optimality is then presented in Subsection~\ref{SecDecayRates}, and the section is concluded in Subsection~\ref{SecSloooooooow} by an example showing that the decay of solutions can be arbitrarily slow when $\Sigma$ is of saturation type, answering a conjecture of \cite{V-Martinez2000}.

\subsection{Basic results and definitions}
\label{SecAsymptoticBasic}

We start with the following basic property of a damping set $\Sigma$ as regards the behavior of $e_p$ along solutions of \eqref{eq:wave}. When $\Sigma$ is the graph of a function, this result is classical for $p = 2$ and has been essentially given for $p \in [1, +\infty]$ in the case of internal distributed damping in \cite{haraux1D}.

\begin{proposition}
\label{PropDampingNonStrict}
Let $\Sigma \subset \mathbb R^2$, $p \in [1, +\infty]$, $S$ be the set-valued map whose graph is $R \Sigma$, and assume that $S$ is a multi-valued function. Then $t \mapsto e_p(z)(t)$ is nonincreasing for every solution $z$ of \eqref{eq:wave} in $\mathsf X_p$ if and only if \ref{HypoSigma-Damping} holds.
\end{proposition}

\begin{proof}
Assume that \ref{HypoSigma-Damping} holds and recall that it is equivalent to \ref{HypoSigma-DampingTilde}. Let $z$ be a solution of \eqref{eq:wave} in $\mathsf X_p$ and $g$ be as in Definition~\ref{DefWeakSolution}. By \ref{HypoSigma-DampingTilde} and \eqref{recurrence-g}, one deduces that $\abs{g(t+1)} \leq \abs{g(t-1)}$ for every $t \geq 0$, and the result follows immediately in the case $p < +\infty$ by \eqref{eq:deriv-energy}. If $p = +\infty$, notice that $e_\infty(z)(t) = \lim_{q \to +\infty} e_q(z)(t)$ for every $t \geq 0$. Since $e_q(z)(\cdot)$ is nonincreasing for every $q \in [1, +\infty)$, the same holds true for $e_\infty(z)(\cdot)$.

Conversely, reasoning by contraposition, assume that \ref{HypoSigma-Damping} does not hold. Therefore, there exist $(x, y) \in R \Sigma$ with $\abs{y} > \abs{x}$. Fix $p \in [1, +\infty]$ and let $g_0, g_1$ in $\mathsf Y_p$ be the constant functions equal to $x$ and $y$, respectively. Let $(g_n)_{n \geq 1}$ be the sequence in $\mathsf Y_p$ defined in Proposition~\ref{PropGn} corresponding to a solution of \eqref{eq:wave} in $\mathsf X_p$ with initial condition $\mathfrak I^{-1}(g_1)$, which exists thanks to Remark~\ref{RemkExistSimple} since $g_1$ is a simple function. Then the sequence $(g_n)_{n \geq 0}$ corresponds, in the sense of Proposition~\ref{PropGn}, to a solution $z$ of \eqref{eq:wave} in $\mathsf X_p$ with initial condition $\mathfrak I^{-1}(g_0)$. Using \eqref{eq:normGn}, one has
\[
e_p(z)(2) - e_p(z)(0) = \norm{g_1}_{p} - \norm{g_0}_{p} = 2^{\frac{1}{p}} \left(\abs{y} - \abs{x}\right) > 0,
\]
where $2^{\frac{1}{p}} = 1$ for $p = +\infty$. This shows that the function $t\mapsto e_p(z)(t)$ is not nonincreasing, concluding then the proof of our result.
\end{proof}

For the rest of the section, $\Sigma$ will be assumed to be a damping set (recall Definition~\ref{def-damping}) and we aim at understanding the asymptotic behavior of solutions of \eqref{eq:wave} in $\mathsf X_p$ for $p \in [1, +\infty]$. Taking into account the previous proposition, we next aim at providing necessary and sufficient conditions on $\Sigma$ so that all solutions of \eqref{eq:wave} in $\mathsf X_p$ converges to zero as $t \to +\infty$. Since we are in an infinite-dimensional setting, there are several meaningful definitions of convergence to zero, and we state in the next definition the ones that will be of interest in this paper.

\begin{definition} Let $\Sigma \subset \mathbb R^2$ satisfy \ref{HypoSigma-Zero}, $p \in [1, +\infty]$, and assume that, for every initial condition $(z_0, z_1) \in \mathsf X_p$, there exists a solution $z$ of \eqref{eq:wave} starting at $(z_0, z_1)$.

\begin{description}
\item[\bf{Strong stability}] The wave equation defined by \eqref{eq:wave} in $\mathsf X_p$ is said to be \emph{strongly stable} if, for every $(z_0, z_1)\in \mathsf X_p$ and any solution $z$ of \eqref{eq:wave} starting at $(z_0, z_1)$, one has 
$$
\lim_{t\to+\infty} e_p(z)(t) = 0.
$$

\item[\bf{UGAS}] The wave equation defined by \eqref{eq:wave} in $\mathsf X_p$ is said to be \emph{uniformly globally asymptotically stable} (UGAS) if there exists $\beta\in {\mathcal{KL}}$ such that, for every $(z_0, z_1)\in \mathsf X_p$ and any solution $z$ of \eqref{eq:wave} starting at $(z_0, z_1)$, one has 
\begin{equation}\label{eq:UGAS}
e_p(z)(t) \leq \beta(e_p(z)(0), t), \quad t\geq 0.
\end{equation}

\item[\bf{GES}] The wave equation defined by \eqref{eq:wave} in $\mathsf X_p$ is said to be \emph{globally exponentially stable} (GES) if it is UGAS with $\beta(\xi, t) = C_1 \xi e^{-C_2 t}$ for some positive constants $C_1, C_2$. 
\end{description}
\end{definition}

The definition of strong stability is classical in the context of stabilization of PDEs (see, e.g., the survey \cite{Alabau2012Recent}), while that of UGAS stems from control theory (see, e.g., \cite{Mironchenko2018Characterizations}).

Note that, thanks to Proposition~\ref{PropGn}, these stability concepts in $\mathsf X_p$ admit equivalent statements in terms of the sequences $(g_n)_{n \in \mathbb N}$ in $\mathsf Y_p$ corresponding to solutions of \eqref{eq:wave}. The function $\beta \in \mathcal{KL}$ from the definition of UGAS can be interpreted as a rate of decrease for solutions of \eqref{eq:wave} in $\mathsf X_p$. In the sequel, we will hence say that \eqref{eq:wave} is UGAS in $\mathsf X_p$ with rate $\beta$.

As a first step towards the characterization of the asymptotic behavior of \eqref{eq:wave}, it is useful to consider real iterated sequences for $S$ since they provide particular solutions of \eqref{eq:wave} whose corresponding sequence $(g_n)_{ n \in \mathbb N}$ in the sense of Proposition~\ref{PropGn} is made of constant functions. Thanks to \ref{HypoSigma-DampingTilde} from Proposition~\ref{PropHTilde}, all real iterated sequences are nonincreasing in absolute value as soon as $\Sigma$ is a damping, and we seek extra conditions on $\Sigma$ so that these real iterated sequences converge to $0$. One might think that a natural sufficient condition for that purpose would be that $\Sigma$ is a strict damping set, i.e., it satisfies \ref{HypoSigma-StrictDamping}, but this is not the case, as shown by the following example.

\begin{example}
Let $\Sigma_0$ be a strict damping set, $p \in [1, +\infty]$, and $(a_n)_{n \in \mathbb N}$ be a decreasing sequence in $\mathbb R_+$ converging to a limit $a_\ast > 0$. Let $\Sigma$ be the set such that
\[
R \Sigma = R \Sigma_0 \cup \{(a_n, a_{n+1}) \suchthat n \in \mathbb N\},
\]
i.e., the sequence $(a_n)_{n \in \mathbb N}$ becomes a real iterated sequence for the set-valued map whose graph is $R \Sigma$. Clearly, $\Sigma$ is still a strict damping set since, for every $(x, y) \in R \Sigma$ with $(x, y) \neq (0, 0)$, one has either $(x, y) \in R \Sigma_0$, in which case $\abs{y} < \abs{x}$ since $\Sigma_0$ is a strict damping, or $(x, y) = (a_n, a_{n+1})$ for some $n \in \mathbb N$, in which case $\abs{y} = a_{n+1} < a_n = \abs{x}$ since $(a_n)_{n \in \mathbb N}$ is decreasing.

Consider the sequence $(g_n)_{n \in \mathbb N}$ in $\mathsf Y_p$ such that, for every $n \in \mathbb N$, $g_n$ is constant and equal to $a_n$. Since this sequence clearly satisfies \eqref{eq:dynSystGn}, Proposition~\ref{PropGn} ensures that this sequence corresponds to a solution $z$ of \eqref{eq:wave} in $\mathsf X_p$. The sequence $(g_n)_{n \in \mathbb N}$ converges in $\mathsf Y_p$ to the constant function equal to $a_\ast > 0$, and thus, in particular, $z$ does not converge to $0$ in $\mathsf X_p$.
\end{example}

The previous example furnishes an instance of a set $\Sigma$ which is a strict damping but for which there exists a real iterated sequence not converging to $0$. Conversely, we next provide an example of a set $\Sigma$ which is a nonstrict damping set for which every real iterated sequence converges to $0$.

\begin{example}
\label{ExplStrongStabilityNotNecessary}
Let $\Sigma \subset \mathbb R^2$ be the graph of the function $\sigma: \mathbb R \to \mathbb R$ given by $\sigma(x) = \min(0, x)$ for $x \in \mathbb R$. Then $\Sigma$ is a nonstrict damping set with the corresponding set-valued map $S$ given by $S(x) = \{-\max(x, 0)\}$ for $x \in \mathbb R$. It is immediate to check that every iterated sequence converges to $0$ since $S^{[2]}(x) = \{0\}$ for every $x \in \mathbb R$.
\end{example}

\subsection{Real iterated sequences and stability}
\label{SecRealIterated}

In order to present our results, we introduce the following convergence properties for real iterated sequences.

\begin{definition}
\label{DefConvZero}
Let $S: \mathbb R \rightrightarrows \mathbb R$ be a set-valued map.
\begin{enumerate}
    \item We say that \emph{real iterated sequences for $S$ converge simply to zero} if every real iterated sequence converges to zero.
    \item We say that \emph{real iterated sequences for $S$ converge to zero uniformly (on compact sets)} if, for every $r \geq 0$ and $\varepsilon > 0$, there exists $N \in \mathbb N$ such that, for every $x \in \mathbb R$ satisfying $\abs{x} \leq r$ and every real iterated sequence $(x_n)_{n \in \mathbb N}$ for $S$ starting from $x$, one has $\abs{x_n} < \varepsilon$ for every $n \geq N$.
\end{enumerate}
\end{definition}

Our first result provides necessary and sufficient conditions for the strong stability of \eqref{eq:wave} in $\mathsf X_p$ in terms of convergence properties of real iterated sequences.

\begin{proposition}
\label{PropCvSequence}
Let $\Sigma \subset \mathbb R^2$ be a damping set and $p \in [1, +\infty]$.
\begin{enumerate}
\item\label{ItemCvSequencePFinite} If $p < +\infty$, then the wave equation \eqref{eq:wave} is strongly stable in $\mathsf X_p$ if and only if real iterated sequences for $S$ converge simply to zero.
\item\label{ItemCvSequencePInfinite} For $p = +\infty$, the following statements are equivalent:
\begin{enumerate}
\item\label{ItemItemUGASXInfty} The wave equation \eqref{eq:wave} is UGAS in $\mathsf X_\infty$.
\item\label{ItemItemStrongStabilityInfty} The wave equation \eqref{eq:wave} is strongly stable in $\mathsf X_\infty$.
\item\label{ItemItemRealIteratedCvUnif} Real iterated sequences for $S$ converge uniformly to zero.
\end{enumerate}
\end{enumerate}
\end{proposition}

\begin{proof}
To prove \ref{ItemCvSequencePFinite}, notice first that real iterated sequences for $S$ provide particular solutions of \eqref{eq:wave} for which the corresponding iterated sequence $(g_n)_{n \in \mathbb N}$ from Proposition~\ref{PropGn} is made of constant functions (cf.\ also Remark~\ref{RemkExistSimple}). Hence, convergence of all such sequences to zero is a necessary condition for the strong stability of \eqref{eq:wave}.

Conversely, assume that real iterated sequences for $S$ converge simply to zero. Let $z$ be a solution of \eqref{eq:wave} in $\mathsf X_p$ and consider the corresponding sequence $(g_n)_{n \in \mathbb N}$ in $\mathsf Y_p$ from Proposition~\ref{PropGn}. Then $(g_n(s))_{n \in \mathbb N}$ is a real iterated sequence for $S$ for a.e.\ $s \in [-1, 1]$ and hence $(g_n)_{n \in \mathbb N}$ converges pointwise to zero almost everywhere. Since $\Sigma$ is a damping set, one has $\abs{g_n(s)} \leq \abs{g_0(s)}$ for every $n \in \mathbb N$ and a.e.\ $s \in [-1, 1]$, and one concludes by the dominated convergence theorem that $g_n \to 0$ in $\mathsf Y_p$ as $n$ tends to infinity.

Let us now prove \ref{ItemCvSequencePInfinite}. The implication \ref{ItemItemUGASXInfty} $\implies$ \ref{ItemItemStrongStabilityInfty} is trivial by definition. We prove that \ref{ItemItemStrongStabilityInfty} $\implies$ \ref{ItemItemRealIteratedCvUnif} reasoning by contraposition. Hence, assume that real iterated sequences for $S$ do not converge uniformly to zero. This implies that there exists $r > 0$ and $\varepsilon > 0$ such that, for every $k \in \mathbb N$, there exists $x^{(k)} \in [-r, r]$ and a real iterated sequence $(x^{(k)}_n)_{n \in \mathbb N}$ for $S$ starting at $x^{(k)}$ for which $\abs{x^{(k)}_{N_k}} \geq \varepsilon$ for some $N_k \geq k$. Without loss of generality, one can assume $N_k = k$ by possibly replacing the initial value $x^{(k)}$ by $x^{(k)}_{N_k - k}$, which still belongs to $[-r, r]$. Let $\{A_n\}_{n \in \mathbb N}$ be a family of disjoint measurable subsets of $[-1, 1]$ of positive Lebesgue measure. Consider the sequence of functions $(g_n)_{n \in \mathbb N}$ defined by
\[
g_n = \sum_{k\geq 0} x^{(k)}_n \chi_{A_k}.
\]
Then, by construction, for every $n \in \mathbb N$, one has $g_n \in \mathsf Y_\infty$ and $g_{n+1}(s) \in S(g_n(s))$ for a.e.\ $s \in [-1, 1]$. By Proposition~\ref{prop:dis-sup}, one gets that $\norm{g_n}_\infty \geq \varepsilon$ for every $n \in \mathbb N$, and thus, using Proposition~\ref{PropGn}, the corresponding solution of \eqref{eq:wave} starting from $\mathfrak I^{-1}(g_0)$ does not converge to $0$ in $\mathsf X_\infty$.

Let us finally prove that \ref{ItemItemRealIteratedCvUnif} $\implies$ \ref{ItemItemUGASXInfty}. Assume that \ref{ItemItemRealIteratedCvUnif} holds and let $\beta_0: \mathbb R_+ \times \mathbb R_+ \to \mathbb R_+$ be defined by
\begin{align*}
\beta_0(r, t) = \sup\{\abs{x_n} & \suchthat 
n \geq t \text{ and } (x_k)_{k \in \mathbb N} \text{ is a real iterated }\\ & \hphantom{\suchthat {}} \text{ sequence for } S \text{ with } \abs{x_0} \leq r \}.
\end{align*}
One can verify that $\beta_0(0, \cdot) \equiv 0$, $\beta_0(\cdot, t)$ is nondecreasing for every $t \geq 0$, $\beta_0(r, \cdot)$ is nonincreasing for every $r \geq 0$, and $\beta_0(r, t) \leq r$ for every $(r, t) \in \mathbb R_+ \times \mathbb R_+$. The statement of \ref{ItemItemRealIteratedCvUnif} exactly says that, for every $r \geq 0$, $\beta_0(r, t) \to 0$ as $t \to +\infty$. Moreover, for every real iterated sequence $(x_n)_{n \in \mathbb N}$ for $S$, one has $\abs{x_n} \leq \beta_0(\abs{x_0}, t)$ for $n\geq t$.

Let $z$ be a solution of \eqref{eq:wave} in $\mathsf X_\infty$ and consider the corresponding sequence $(g_n)_{n \in \mathbb N}$ in $\mathsf Y_\infty$ from Proposition~\ref{PropGn}. For a.e.\ $s \in [-1, 1]$, $(g_n(s))_{n \in \mathbb N}$ is a real iterated sequence for $S$ starting from $g_0(s)$, and thus
\[
\abs{g_n(s)} \leq \beta_0(\abs{g_0(s)}, n) \qquad \text{ for a.e.\ } s \in [-1, 1] \text{ and every } n \in \mathbb N.
\]
Using the fact that $\beta_0$ is nondecreasing with respect to its first argument, we deduce that, for every $n \in \mathbb N$,
\[
\norm{g_n}_\infty \leq \beta_0(\norm{g_0}_\infty, n).
\]
From Propositions~\ref{PropGn} and \ref{PropDampingNonStrict}, one has
\[
e_\infty(z)(t) \leq e_\infty(z)\left(2\floor*{t/2}\right) = \norm*{g_{\floor*{t/2}}}_\infty \leq \beta_0\left(e_\infty(z)(0), \floor*{t/2}\right) \qquad \text{ for every } t \geq 0.
\]
One concludes the proof by applying Lemma~\ref{lem:KL} from Appendix~\ref{AppLemmas} to the function $(r, t) \mapsto \beta_0(r, \floor*{t/2})$.
\end{proof}

\begin{remark}
Since the notion of real iterated sequences of $S$ converging simply to zero is independent of $p$, it follows from Proposition~\ref{PropCvSequence}\ref{ItemCvSequencePFinite} that the convergence of solutions of \eqref{eq:wave} to $0$ for $p$ finite is independent of $p$: if solutions of \eqref{eq:wave} converge to $0$ in \emph{some} $\mathsf X_p$ with $p$ finite, then solutions will converge to $0$ in \emph{every} $\mathsf X_p$ with $p$ finite. The situation is different for $p = +\infty$, in which case one must require a uniformity property on the convergence to zero of real iterated sequences for $S$ in order to obtain convergence in the strong topology of $\mathsf X_\infty$. By adapting the proof of Proposition~\ref{PropCvSequence}\ref{ItemCvSequencePFinite}, one can show that simple convergence to zero of real iterated sequences for $S$ is also a necessary and sufficient condition for the convergence to $0$ of solutions of \eqref{eq:wave} in $\mathsf X_\infty$ in the weak-$\ast$ topology of $\mathsf X_\infty$.
\end{remark}

Recall the obvious fact that, for every solution $z$ of \eqref{eq:wave} in $\mathsf X_\infty$, $p \in [1, +\infty]$, and $t \geq 0$, one has
\[e_p(z)(t) \leq 2^{1/p} e_\infty(z)(t),\]
with the convention that $1/p = 0$ for $p = +\infty$. In case the wave equation \eqref{eq:wave} is UGAS in $\mathsf X_\infty$, one deduces at once a UGAS-like estimate of $e_p(z)(t)$ for every finite $p$, involving however, as regards the dependence with the respect to the initial condition, its $\mathsf X_\infty$ norm $e_\infty(z)(0)$ only. In the next result, we refine the above mentioned trivial estimate of $e_p(z)(t)$, with another one involving the $\mathsf X_p$ norm of the initial condition.

\begin{proposition}\label{prop:inf+p-fini} Let $\Sigma \subset \mathbb R^2$ be a damping set and $p \in [1, +\infty)$. Assume that the wave equation \eqref{eq:wave} associated with $\Sigma$ is UGAS in $\mathsf X_\infty$ with rate $\beta$. Then, for every solution $z$ of \eqref{eq:wave} in $\mathsf X_\infty$, one has, for every $t \geq 0$,
\begin{equation}\label{eq:UGASp-inf}
e_p(z)(t) \leq 2^{1/p} \beta\left(e_p^{1/2}(z)(0), 2\floor*{\tfrac{t}{2}}\right) + e_p^{1/2}(z)(0) \beta\left(\max(e_\infty(z)(0), e_p^{1/2}(z)(0)), 2\floor*{\tfrac{t}{2}}\right).
\end{equation}
\end{proposition}

\begin{proof}
Let $z$ be a solution of \eqref{eq:wave} in $\mathsf X_\infty$ and consider the corresponding sequence $(g_n)_{n \in \mathbb N}$ in $\mathsf Y_\infty$ in the sense of Proposition~\ref{PropGn}. Let us denote $Z_{q} = e_q(z)(0)$ for $q \in \{p, \infty\}$. Consider the partition of $[-1, 1]$ in the disjoint subsets
\begin{align*}
E_1&=\{s\in [-1,1] \suchthat \abs{g_0(s)} < Z_p^{1/2}\},\\
E_2&=\{s\in [-1,1] \suchthat \abs{g_0(s)} \geq Z_p^{1/2}\}.
\end{align*}

Let $\chi_i$ and $\alpha_i$, $i\in\{1,2\}$, be the characteristic 
functions and Lebesgue measures associated with $E_i$,
respectively. One clearly has, for every $n \in \mathbb N$,
\begin{align*}
g_n & = g_n\chi_1 + g_n\chi_2, \\
\norm{g_n}^p_p & = \norm{g_n\chi_1}^p_p + \norm{g_n\chi_2}^p_p.
\end{align*}
Let $i\in\{1,2\}$. Then
$$
\abs{g_n(s) \chi_i(s)} \leq \chi_i(s) \beta(\norm{g_0 \chi_i}_\infty, 2n), \qquad \text{ for a.e.\ } s \in [-1, 1].
$$
For $i=1$, since $\beta(\cdot,t)$ is increasing for all $t\geq 0$, one deduces that
\begin{equation}\label{eq:gchiLp1}
\norm{g_n \chi_1}_p^p \leq 2 \Bigl(\beta(Z_p^{1/2}, 2n)\Bigr)^p, 
\end{equation}
and, for $i=2$, one gets
\begin{equation}\label{eq:gchiLp2}
\norm{g_n \chi_2}_p^p \leq \alpha_2 \Bigl(\beta(\max(Z_\infty, Z_p^{1/2}), 2n)\Bigr)^p.
\end{equation}
By Chebyshev's inequality, it follows that $\alpha_2 \leq Z_p^{p/2}$. Putting together \eqref{eq:gchiLp1}, \eqref{eq:gchiLp2} and the above estimate of $\alpha_2$, one obtains 
\[
\norm{g_n}_{p} \leq \left(2 \beta(Z_p^{1/2}, 2n)^p + Z_p^{p/2}
\beta(\max(Z_\infty, Z_p^{1/2}), 2n)^p\right)^{\frac{1}{p}}.
\]
Since $p\geq 1$, one deduces that 
\begin{equation}\label{eq:ep-fini-infini1}
\norm{g_n}_{p} \leq 2^{1/p} \beta(Z_p^{1/2}, 2n) + Z_p^{1/2}
\beta(\max(Z_\infty, Z_p^{1/2}), 2n).
\end{equation}
We conclude using the fact that $t \mapsto e_p(z)(t)$ is nonincreasing.
\end{proof}

At the light of the equivalence between \ref{ItemItemUGASXInfty} and \ref{ItemItemStrongStabilityInfty} from Proposition~\ref{PropCvSequence}, one may wonder if such an equivalence holds in $\mathsf X_p$ with $p$ finite. The answer is negative, which is a consequence of the next proposition.

\begin{proposition}
\label{PropPFiniteNotUGAS}
Let $\Sigma \subset \mathbb R^2$ be a damping set and $p \in [1, +\infty)$. If \ref{HypoSigma-NoDampingAtInfty} holds, then the wave equation \eqref{eq:wave} is not UGAS in $\mathsf X_p$.
\end{proposition}

\begin{proof}
We prove that \eqref{eq:wave} is not UGAS in $\mathsf X_p$ reasoning by contradiction. Assume then that \eqref{eq:wave} is UGAS in $\mathsf X_p$ with rate $\beta$, i.e., there exists some $\beta \in \mathcal{KL}$ such that, for every solution $z$ of \eqref{eq:wave}, one has
\begin{equation}
\label{ProofNotUGAS}
e_p(z)(t) \leq \beta(e_p(z)(0), t), \qquad \text{ for every } t \geq 0.
\end{equation}
Since $\beta \in \mathcal{KL}$, there exists $N \in \mathbb N$ such that $\beta(1, 2N) < 1/2$.

Assumption \ref{HypoSigma-NoDampingAtInftyTilde} implies that there exists $M > 1$ such that, for every $(x, y) \in \mathbb R^2$, if $\abs{(x, y)} > M$ and $y \in S(x)$, then $\abs*{\frac{y}{x}} > \frac{1}{2^{1/N}}$. Let $x_0 \in \mathbb R$ be such that $\abs{x_0} \geq 2M$ and consider a real iterated sequence $(x_n)_{n \in \mathbb N}$ for $S$ starting from $x_0$. Then, for every $n \in \{0, \dotsc, N\}$, one has
\begin{equation}
\label{ToShowByInduction}
\abs*{\frac{x_n}{x_0}} \geq \frac{1}{2^{n/N}}.
\end{equation}
Indeed, \eqref{ToShowByInduction} trivially holds for $n = 0$. If $n \in \{0, \dotsc, N-1\}$ is such that \eqref{ToShowByInduction} holds, then $\abs{x_n} \geq \frac{\abs{x_0}}{2^{n/N}} > M$ and thus, since $x_{n+1} \in S(x_n)$, we deduce that $\abs*{\frac{x_{n+1}}{x_n}} > \frac{1}{2^{1/N}}$, yielding that $\abs*{\frac{x_{n+1}}{x_0}} = \abs*{\frac{x_{n+1}}{x_n}} \abs*{\frac{x_n}{x_0}} \geq \frac{1}{2^{(n+1)/N}}$, showing that \eqref{ToShowByInduction} holds for $n+1$. Hence, by induction, \eqref{ToShowByInduction} is established for every $n \in \mathbb N$.

Consider now the iterated sequence of functions $(g_n)_{n \in \mathbb N}$ for $S$ in $\mathsf Y_p$ such that $g_n = x_n \chi_{A}$, where $A \subset [-1, 1]$ is an interval of length $1/\abs{x_0}^p$. Then clearly $\norm{g_n}_p = \abs*{\frac{x_n}{x_0}}$. Letting $z$ be the solution of \eqref{eq:wave} corresponding to $(g_n)_{n \in \mathbb N}$ in the sense of Proposition~\ref{PropGn}, one deduces from \eqref{eq:normGn} that
\[
e_p(z)(2N) = \norm{g_N}_p = \abs*{\frac{x_N}{x_0}} \geq \frac{1}{2}.
\]
This is a contradiction since, from \eqref{ProofNotUGAS}, one also has that
\[
e_p(z)(2N) \leq \beta(e_p(z)(0), 2N) = \beta(1, 2N) < \frac{1}{2}.
\]
This contradiction establishes the desired result.
\end{proof}

\begin{remark}
A particular instance of the above proposition has been established in \cite[Theorem~4.1.1]{V-Martinez2000} where it is shown that \eqref{eq:wave} is not GES in the case $p = 2$ and $\Sigma$ is the graph of a saturation function.
\end{remark}

Proposition~\ref{PropPFiniteNotUGAS} raises the question of whether one can provide conditions on $\Sigma$ under which \eqref{eq:wave} is UGAS for finite $p$. Our next result identifies the assumption \ref{HypoSigma-SectorInfty} as such a condition.

\begin{theorem}
\label{thm-ugas}
Let $\Sigma \subset \mathbb R^2$ be a damping satisfying \ref{HypoSigma-SectorInfty} and $p \in [1, +\infty)$. Assume moreover that the wave equation defined in \eqref{eq:wave} is UGAS in $\mathsf X_\infty$ with rate $\beta$. Then it is also UGAS in $\mathsf X_p$.
\end{theorem}

\begin{proof}
Let $M > 0$ and $\mu \in (0, 1)$ be given by \ref{HypoSigma-SectorInftyTilde} and set $\eta = 1/\mu$. Let $z$ be a solution of \eqref{eq:wave} in $\mathsf X_p$ and denote by $(g_n)_{n \in \mathbb N}$ the corresponding sequence in $\mathsf Y_p$ in the sense of Proposition~\ref{PropGn}. For simplicity of notation, we set $g = g_0$, $Z_p = \norm{g}_p$, $Z = \max(M, Z_p^{1/2})$.

Consider the partition of $[-1, 1]$ in disjoints subsets defined by
\begin{align*}
E&=\{s\in [-1,1] \suchthat \abs{g(s)}< Z_p^{1/2}\},\\
F&=\{s\in [-1,1] \suchthat Z_p^{1/2}\leq\abs{g(s)}<Z\},\\
E_k&=\{s\in [-1,1] \suchthat \eta^{k} Z \leq \abs{g(s)}<\eta^{k+1}Z\},
\ k\geq 0,
\end{align*}
Let $\chi_E$ and $\alpha_E$, $\chi_F$ and $\alpha_F$, $\chi_k$ and $\alpha_k$, $k\geq 0$, be the characteristic functions and Lebesgue measures associated 
with $E$, $F$ and $E_k$, $k\geq 0$, respectively. One clearly has
\begin{align*}
g&=g\chi_E+g\chi_F+\sum_{k\geq 0}g\chi_k,\\
Z_p^p&=\norm{g\chi_E}^p_p+\norm{g\chi_F}^p_p+\sum_{k\geq 0}\alpha_k\xi_k^p,\quad \text{ where } \xi_k\in [\eta^{k}Z,\eta^{k+1}Z] \text{ for } k\geq 0.
\end{align*}
Moreover, for every $n \in \mathbb N$, one has
\begin{equation}\label{eq:ep-fini-infini}
\norm{g_n}_p^p = \norm{g_n \chi_E}_p^p + \norm{g_n \chi_F}_p^p + \sum_{k \geq 0} \norm{g_n \chi_k}_p^p.
\end{equation}

Note that, for every set $\mathcal S \in \{E, F, E_0, E_1, \dotsc\}$, the sequence $(g_n \chi_{\mathcal S})_{n \in \mathbb N}$ corresponds, in the sense of Proposition~\ref{PropGn}, to a solution $z_{\mathcal S}$ of \eqref{eq:wave} with initial condition $\mathfrak I^{-1}(g \chi_{\mathcal S})$. In particular, since $g \chi_{\mathcal S} \in \mathsf Y_\infty$, $z_{\mathcal S}$ is a solution of \eqref{eq:wave} in $\mathsf X_\infty$ and one has, for every $n \in \mathbb N$,
\begin{equation}\label{eq:estSETstepN}
\abs{g_n(s) \chi_{\mathcal S}(s)} \leq \chi_{\mathcal S}(s) \beta(\norm{g \chi_{\mathcal S}}_\infty, 2n), \qquad \text{ for a.e.\ } s \in [-1, 1].
\end{equation}
since \eqref{eq:wave} is UGAS in $\mathsf X_\infty$ with rate $\beta$.

One deduces for ${\mathcal{S}}=E$ that 
\begin{equation}\label{eq:gchiLpE}
\norm{g_n \chi_{E}}_p^p \leq \alpha_E \Bigl(\beta(Z_p^{1/2}, 2n)\Bigr)^p
\end{equation}
and, for ${\mathcal{S}}=F$
$$
\norm{g_n \chi_{F}}_p^p \leq \alpha_F \Bigl(\beta(Z, 2n)\Bigr)^p.
$$ 
By Chebyshev's inequality, we have that $\alpha_F \leq \frac{Z_p^p}{Z_p^{p/2}} \leq Z_p^{p/2}$. Then, we obtain
\begin{equation}\label{eq:gchiLpF}
\norm{g_n \chi_{F}}_p^p \leq Z_p^{p/2} \Bigl(\beta(Z, 2n)\Bigr)^p.
\end{equation}

We next estimate the sum appearing in the right-hand side of \eqref{eq:ep-fini-infini}. For $k \geq 0$ and $s\in E_k$, by applying \ref{HypoSigma-SectorInftyTilde}, it holds that $\norm{S(g(s))} \leq \eta^k Z$ and, more generally, an immediate inductive argument using \ref{HypoSigma-SectorInftyTilde} shows that $\norm{S^{[n]}(g(s))} \leq \eta^{\max(0, k+1-n)} Z$ for every $n \geq 0$. Then, for every $k \in \mathbb N$ and $n \in \mathbb N$, one has
\[
\norm{g_n \chi_k}_p^p \leq \alpha_k \eta^{p \max(0, k+1-n)} Z^p \leq \frac{\alpha_k \xi_k^p}{\eta^{p (k - \max(0, k+1-n))}}.
\]
For $k \geq \floor*{n/2} + 1$, it follows that $k - \max(0, k+1-n) \geq \frac{n-1}{2}$ and hence
\begin{equation*}
\norm{g_n \chi_k}_p^p \leq \frac{\alpha_k \xi_k^p}{\sqrt{\eta}^{p (n-1)}}.
\end{equation*}
Since $\sum_{k \geq 0} \alpha_k \xi_k^p \leq Z_p^p$, we deduce that
\begin{equation}
\label{eq:estim-sum-infty}
\sum_{k = \floor*{n/2}+1}^\infty \norm{g_n \chi_k}_p^p \leq \frac{Z_p^p}{\sqrt{\eta}^{p (n-1)}}.
\end{equation}
For $k \in \left\{0, \dotsc, \floor*{n/2}\right\}$ and $n \geq 1$, note that, since $g_{k+1}(s) \in S^{[k+1]}(g(s))$, one has $\abs{g_{k+1}(s)} \leq Z$ for $s \in E_k$. Using the facts that $(g_{k+1+j} \chi_{k})_{j \in \mathbb N}$ is an iterated sequence for $S$ in $\mathsf X_\infty$ and that \eqref{eq:wave} is UGAS in $\mathsf X_\infty$, we deduce that $\norm{g_{k+1+j} \chi_k}_\infty \leq \beta(Z, 2j)$ for every $j \in \mathbb N$, and thus
\begin{equation*}
\norm{g_n \chi_k}_p^p \leq \alpha_k \Bigl(\beta(Z, 2(n-1-k))\Bigr)^p \leq \alpha_k \Bigl(\beta(Z, n-1)\Bigr)^p,
\end{equation*}
since $\beta(Z, \cdot)$ is nonincreasing. Using the facts that $\xi_k \geq \eta^k Z$ and that $\sum_{k \geq 0} \alpha_k \xi_k^p \leq Z_p^p$, one deduces that
\begin{equation}
\label{eq:estim-sum-n2}
\begin{aligned}
\sum_{k=0}^{\floor*{n/2}} \alpha_k \Bigl(\beta(Z, n-1)\Bigr)^p & \leq \frac{\Bigl(\beta(Z, n-1)\Bigr)^p}{Z^p} \sum_{k=0}^{\floor*{n/2}} \frac{\alpha_k \xi_k^p}{\eta^{k p}} \\ 
& \leq \frac{\Bigl(\beta(Z, n-1)\Bigr)^p}{Z^p} Z_p^p \leq Z_p^{p/2} \Bigl(\beta(Z, n-1)\Bigr)^p.
\end{aligned}
\end{equation}
Putting together \eqref{eq:gchiLpE}, \eqref{eq:gchiLpF}, \eqref{eq:estim-sum-infty}, and \eqref{eq:estim-sum-n2}, and using the fact that $\alpha_E \leq 2$, one deduces that
\[
\norm{g_n}_p^p \leq 2 \Bigl(\beta(Z_p^{1/2}, 2n)\Bigr)^p + Z_p^{p/2} \Bigl(\beta(Z, 2n)\Bigr)^p + Z_p^{p/2} \Bigl(\beta(Z, n-1)\Bigr)^p + \frac{Z_p^p}{\sqrt{\eta}^{p (n-1)}}.
\]
We conclude using the fact that $t\mapsto e_p(z)(t)$ is nonincreasing and by Lemma~\ref{lem:KL} as in the proof of Proposition~\ref{PropCvSequence}.
\end{proof}

Note that Proposition~\ref{prop:inf+p-fini} and Theorem~\ref{thm-ugas} both assume that \eqref{eq:wave} is UGAS in $\mathsf X_\infty$, this property being characterized in terms of real iterated sequences for $S$ in Proposition~\ref{PropCvSequence}. Proposition~\ref{PropDefiRho}\ref{ItemIteratedUniform} in the next subsection provides an easy to check sufficient condition to have the latter property.

\subsection{Stability properties based on properties of \texorpdfstring{$S^{[2]}$}{S2}}
\label{SecS2}

Proposition~\ref{PropCvSequence} provides necessary and sufficient conditions for the strong stability of \eqref{eq:wave} in $\mathsf X_p$ in terms of convergence to zero of \emph{real} iterated sequences for the multi-valued map $S: \mathbb R \rightrightarrows \mathbb R$. These necessary and sufficient conditions are not yet satisfactory because they are difficult to verify. Instead, we provide in the sequel necessary or sufficient conditions on $S$ that are simpler to check. For that purpose, we introduce the function $\rho: \mathbb R_+ \to \mathbb R_+$ defined next which will be useful for several results in the sequel of the paper.

\begin{proposition}
\label{PropDefiRho}
Let $\Sigma \subset \mathbb R^2$ be a damping set. Let $\rho: \mathbb R_+ \to \mathbb R_+$ be the function defined by
\begin{equation}
\label{eq:defiRho}
\rho(r) = \lim_{\eta \to 0^+} \sup_{\abs{x} \leq r + \eta} \norm{S^{[2]}(x)}.
\end{equation}
Then the following properties hold.

\begin{enumerate}
\item\label{ItemRhoEasy} $\rho(0) = 0$, $\rho$ is nondecreasing, and $\rho(r) \leq r$ for every $r \in \mathbb R_+$.

\item\label{ItemRhoUSC} $\rho$ is upper semi-continuous.

\item\label{ItemRhoStrictDamping} Assume that $S^{[2]}$ is closed. Then $\rho(r) < r$ for every $r > 0$ if and only if $S^{[2]}$ is a strict damping.

\item\label{ItemIterateEstimateRho} For every $n \in \mathbb N$ and $r \in \mathbb R_+$, one has
\begin{equation}\label{eq:rho-iterate}
\sup_{\abs{x} \leq r} \norm{S^{[n]}(x)} \leq \rho^{[\floor{n/2}]}(r).
\end{equation}
In particular, for every real iterated sequence $(x_n)_{n \in \mathbb N}$ for $S$, one has
\begin{equation}
\label{eq:iterated-rho}
\abs{x_n} \leq \rho^{[\floor{n/2}]}(\abs{x_0}).
\end{equation}

\item\label{ItemIteratedUniform} If $\rho(r) < r$ for every $r > 0$, then real iterated sequences for $S$ converge uniformly to zero.
\end{enumerate}
\end{proposition}

\begin{proof}
Since $\Sigma$ is a damping, one has at once that $\norm{S^{[2]}(x)} \leq \abs{x}$ for every $x \in \mathbb R$ and hence \ref{ItemRhoEasy} follows at once.

To prove \ref{ItemRhoUSC}, notice that, by definition of $\rho$, for every $r \in \mathbb R_+$ and $\varepsilon > 0$, there exists $\delta > 0$ such that
\[
\sup_{\abs{x} \leq r + \delta} \norm{S^{[2]}(x)} \leq \rho(r) + \varepsilon,
\]
implying that, for every $\eta \in (0, \delta)$, $\rho(r + \eta) \leq \rho(r) + \varepsilon$. This shows that $\lim_{\eta \to 0^+} \rho(r + \eta) \leq \rho(r)$ for every $r \in \mathbb R_+$, which is equivalent to the upper semi-continuity of $\rho$ since $\rho$ is nondecreasing.

Let us now prove \ref{ItemRhoStrictDamping}. The only nontrivial implication is that $\rho(r) < r$ for every $r > 0$ as soon as $S^{[2]}$ is closed and a strict damping. Fix $r > 0$. By definition of $\rho(r)$, there exist sequences $(x_n)_{n \in \mathbb N}$ and $(y_n)_{n \in \mathbb N}$ with $y_n \in S^{[2]}(x_n)$ for every $n \in \mathbb N$ and such that, up to extracting subsequences, $(x_n)_{n \in \mathbb N}$ converges to some $x_\ast \in [-r, r]$ and $(y_n)_{n \in \mathbb N}$ converges to some $y_\ast \in \mathbb R$ with $\rho(r) = \abs{y_\ast}$. Since $S^{[2]}$ is closed, one has $y_\ast \in S^{[2]}(x_\ast)$, which implies, using the fact that $S^{[2]}$ is a strict damping, that $(x_\ast, y_\ast) = (0, 0)$ or $\abs{y_\ast} < \abs{x_\ast}$. In both cases, one deduces that $\rho(r) < r$.

To prove \ref{ItemIterateEstimateRho}, notice that it suffices to prove \eqref{eq:rho-iterate} for even integers since $\norm{S(x)} \leq \abs{x}$ for every $x \in \mathbb R$, and that \eqref{eq:rho-iterate} is trivially true for $n = 0$ and $n = 2$. The argument goes on by induction: let $n \geq 4$ be an even integer so that \eqref{eq:rho-iterate} holds for even integers $m < n$ and set $k = \frac{n}{2}$. For every $r \geq 0$, $y \in [-r, r]$, and $z \in S^{[2(k-1)]}(y)$, one has
\[
\norm{S^{[2]}(z)} \leq \rho(\abs{z}) \leq \rho\left(\norm{S^{[2(k-1)]}(y)}\right) \leq \rho(\rho^{[k-1]}(\abs{y})) \leq \rho^{[k]}(r),
\]
where we used the definition of $\rho$ in the first inequality, the fact that $\rho$ is nondecreasing in the second, third, and fourth inequalities, and the induction hypothesis in the third inequality. Hence \eqref{eq:rho-iterate} follows at once. It is clear that \eqref{eq:iterated-rho} is an immediate consequence of \eqref{eq:rho-iterate}.

Let us finally prove \ref{ItemIteratedUniform}. By taking into account \eqref{eq:iterated-rho}, it is enough to prove that, for every $r > 0$, $\rho^{[n]}(r)$ converges to $0$ as $n \to +\infty$. Since $\rho(x) \leq x$ for every $x \in \mathbb R_+$, $(\rho^{[n]}(r))_{n \in \mathbb N}$ is nonincreasing, and hence it admits a limit $r_\ast \in \mathbb R_+$. Since $\rho^{[n+1]}(r) = \rho(\rho^{[n]}(r))$ for every $n \in \mathbb N$, one deduces, letting $n \to +\infty$ and using the upper semi-continuity of $\rho$, that $r_\ast = \lim_{n \to +\infty} \rho(\rho^{[n]}(r)) \leq \rho(r_\ast) \leq r_\ast$, showing that $\rho(r_\ast) = r_\ast$. Since $\rho(x) < x$ for every $x > 0$, one then has necessarily that $r_\ast = 0$.
\end{proof}

\begin{remark}
Notice that, if $S$ is closed, then $S^{[2]}$ is closed, and therefore the conclusion of item \ref{ItemRhoStrictDamping} holds true. This is the case, in particular, if $S$ is the graph of a continuous function.
\end{remark}

We can now state necessary and sufficient conditions on $S$ only (and not relying on real iterated sequences) for strong stability of \eqref{eq:wave} in $\mathsf X_p$.

\begin{theorem}
\label{TheoStrongStability}
Let $\Sigma \subset \mathbb R^2$ be a damping set and $p \in [1, +\infty]$. If \eqref{eq:wave} is strongly stable in $\mathsf X_p$, then $S^{[2]}$ is a strict damping. 

Conversely, assume that $S^{[2]}$ is a strict damping and that $\rho$ defined in Proposition~\ref{PropDefiRho} satisfies $\rho(r) < r$ for every $r > 0$. Then \eqref{eq:wave} is strongly stable in $\mathsf X_p$.
\end{theorem}

\begin{proof}
The first part of the statement follows immediately from Propositions~\ref{PropSDamping}\ref{ItemS2IteratedSequence} and \ref{PropCvSequence}, while the second part is an immediate consequence of Propositions~\ref{PropDefiRho}\ref{ItemIteratedUniform} and \ref{PropCvSequence}.
\end{proof}

\begin{remark}
In the case where the damping set $\Sigma$ is the graph of a function $\sigma: \mathbb R \to \mathbb R$, the second part of  Theorem~\ref{TheoStrongStability} has been essentially already obtained in \cite{pierre2000strong} in the case $p = 2$ under the assumption that either $\sigma(s) > 0$ for all $s > 0$ or $\sigma(s) < 0$ for all $s < 0$ (referred to in that reference as a unilateral condition), which is stronger than requiring $S^{[2]}$ to be a strict damping.
\end{remark}

\begin{remark}
Recall that, for $p = +\infty$, Proposition~\ref{PropCvSequence} ensures that strong stability and UGAS are equivalent and \eqref{eq:iterated-rho} immediately implies that, for every solution $z$ of \eqref{eq:wave} in $\mathsf X_\infty$, the corresponding sequence $(g_n)_{n \in \mathbb N}$ in $\mathsf Y_\infty$ from Proposition~\ref{PropGn} satisfies, for every $n \geq 0$,
\begin{equation}
\norm{g_n}_\infty \leq \rho^{[\floor{n/2}]}(\norm{g_0}_\infty),
\end{equation}
which yields
\begin{equation}
e_\infty(z)(t) \leq e_\infty(z)(2 \floor{t/2}) \leq \rho^{[\floor{t/4}]}(e_\infty(z)(0)).
\end{equation}
If now $\rho$ verifies that $\rho(r) < r$ for every $r > 0$, we have UGAS for solutions of \eqref{eq:wave} in $\mathsf X_\infty$ and one can build a corresponding $\mathcal{KL}$ function $\beta$ by applying Lemma~\ref{lem:KL} to the function $(r, t) \mapsto \rho^{[\floor{t/4}]}(r)$.
\end{remark}

\begin{remark}
\label{remk:rho-n}
One could have replaced $S^{[2]}$ by $S^{[n]}$ for $n \geq 1$ in Proposition~\ref{PropDefiRho} to define functions $\rho_n$ similar to $\rho$ ($= \rho_2$) satisfying the same properties. Note that $(\rho_n)_{n \in \mathbb N^\ast}$ is a nonincreasing sequence of functions. We focus on the case $n = 2$ due to Proposition~\ref{PropSDamping}\ref{ItemS2IteratedSequence} as well as to the fact that, from Theorem~\ref{TheoStrongStability}, $S^{[2]}$ being a strict damping is a necessary condition for the strong stability of \eqref{eq:wave}, which is not the case for $S$ due to Example~\ref{ExplStrongStabilityNotNecessary}.

An ultimate justification for sticking to $n = 2$ is the fact that we were not able to come up with a result interesting enough to justify the use of $\rho_n$ with $n > 2$, even though, for $n \geq 3$, the condition $\rho_n(r) < r$ for every $r > 0$ is strictly weaker than the corresponding condition with $n = 2$.
\end{remark}

As an immediate consequence of Proposition~\ref{PropDefiRho}\ref{ItemRhoStrictDamping} and Theorem~\ref{TheoStrongStability}, one deduces the following result.

\begin{corollary}
\label{CoroS2Closed}
Let $\Sigma \subset \mathbb R^2$ be a damping set, $p \in [1, +\infty]$, and assume that $S^{[2]}$ is closed. Then \eqref{eq:wave} is strongly stable in $\mathsf X_p$ if and only if $S^{[2]}$ is a strict damping.
\end{corollary}

\subsection{Decay rates and their optimality}
\label{SecDecayRates}

In the previous sections, we have considered the convergence to zero of solutions of \eqref{eq:wave} and the speed of convergence to zero (or decay rate) was in some cases upper bounded by a $\mathcal{KL}$ function $\beta$ arising from the stability concept of UGAS. In this section, we intend to be more explicit on the dependence in time of $\beta$ and also to provide lower bounds for the decay rate in some cases. More precisely, we say that a $\mathcal{KL}$ function $\beta$ is \emph{optimal} for \eqref{eq:wave} if the latter is UGAS with rate $\beta$ and there exists a nonzero initial condition $(z_0, z_1)\in \mathsf X_p$ such that 
\begin{equation}\label{eq:optimal}
\liminf_{t\to+\infty}\frac{\norm{(z(t,\cdot),z_t(t,\cdot))}_{\mathsf X_p}}
{\beta(\norm{(z_0, z_1)}_{\mathsf X_p},t)}>0.
\end{equation}
In the sequel, we will essentially work with Hypotheses \ref{HypoSigma-gUpperTilde} or \ref{HypoSigma-gLowerTilde}. Our results in this section are decomposed in two parts: first, in Section~\ref{sec:decay-sequence}, we present some preliminary results concerning the decay rates of real iterated sequences, and then, in Section~\ref{sec:decay-solution}, we apply these results to solutions of \eqref{eq:wave}.

\subsubsection{Decay rates for real iterated sequences}
\label{sec:decay-sequence}

The main results of this section are the next two propositions dealing with real iterated sequences for $S$. They will be crucial in order to establish our subsequent results for decay rates of solutions of \eqref{eq:wave} and their optimality in $\mathsf X_p$, $p \in [1, +\infty]$. The first proposition translates \ref{HypoSigma-gUpperTilde} and \ref{HypoSigma-gLowerTilde} in terms of upper and lower bounds, respectively, on real iterated sequences for $S$, relying on iterates of the function $Q$ from \ref{HypoSigma-gUpperTilde} and \ref{HypoSigma-gLowerTilde}. Its proof is an immediate consequence of the formulation of \ref{HypoSigma-gUpperTilde} and \ref{HypoSigma-gLowerTilde}.

\begin{proposition}
\label{prop:ineq-sequences}

Let $\Sigma \subset \mathbb R^2$ be a damping set and $S$ be the corresponding set-valued map whose graph is equal to $R \Sigma$.
\begin{enumerate}
\item Assume that \ref{HypoSigma-gUpper} holds and let $Q$ be as in \ref{HypoSigma-gUpperTilde}. Then, for every real iterated sequence $(x_n)_{n \in \mathbb N}$ for $S$ with $\abs{x_0} \leq M/\sqrt{2}$, one has
\begin{equation}\label{eq:real-iterated-Upper}
\abs{x_n} \leq Q^{[n]}(\abs{x_0}), \qquad \text{ for every } n \in \mathbb N.
\end{equation}

\item\label{ItemIneqSequenceGeq} Assume that \ref{HypoSigma-gLower} holds and let $Q$ be as in \ref{HypoSigma-gLowerTilde}. Then, for every real iterated sequence $(x_n)_{n \in \mathbb N}$ for $S$ with $\abs{x_0} \leq M/\sqrt{2}$, one has
\begin{equation}\label{eq:real-iterated-Lower}
\abs{x_n} \geq Q^{[n]}(\abs{x_0}), \qquad \text{ for every } n \in \mathbb N.
\end{equation}
\end{enumerate}
\end{proposition}

The next proposition provides an explicit asymptotic behavior for the sequence $(Q^{[n]}(\abs{x_0}))_{n\in \mathbb N}$, which serves as either an upper or a lower bound in the previous proposition.

\begin{proposition}\label{prop:decay-rate}
Let $q \in \mathcal C^1(\mathbb R_+, \mathbb R_+)$ be as in \ref{HypoSigma-gUpper} or \ref{HypoSigma-gLower}, i.e., $q(0) = 0$, $0 < q(x) < x$, and $\abs{q^\prime(x)} < 1$ for every $x>0$, and $Q$ be defined from $q$ as in \eqref{Relation-Q-q}. Let $x_0 \in \mathbb R_+^\ast$.

\begin{enumerate}
\item\label{ItemQPrime0} Assume that $q^\prime(0) = 0$ and let $F: (0, x_0] \to \mathbb R_+$ be the diffeomorphism defined by
\[
F(z) = \int_{z}^{x_0} \frac{\diff \xi}{\overline q(\xi)},
\]
where $\overline q(s) = \sqrt{2} q(\sqrt{2} s)$ for $s \in \mathbb R_+$. Then $F(Q^{[n]}(x_0)) \sim n$ as $n \to +\infty$. If moreover there exists $C > 0$ such that
\begin{equation}
\label{EqConditionEquiv}
\frac{F(z) \overline q(z)}{z} \leq C, \qquad \text{ for every } z \in (0, x_0],
\end{equation}
then $$Q^{[n]}(x_0) \sim F^{-1}(n)\text{ as }n \to +\infty.$$

\item\label{ItemQPrimeBetween} Assume that $q^\prime(0) \in (0, 1)$ and let $\lambda = 2\artanh(q^\prime(0))$. Then $\ln Q^{[n]}(x_0) \sim -\lambda n$ as $n \to +\infty$. If moreover one has
\begin{equation}\label{eq:bizarre}
\sum_{k=0}^\infty \psi(e^{-\frac{\lambda}{2} k}) < +\infty,
\end{equation}
where $\psi(r) = \sup_{s \in (0, r]} \abs*{\frac{q(s)}{s} - q^\prime(0)}$, then there exists $C > 1$ such that, for every $n\in\mathbb N$, $$C^{-1} e^{-\lambda n} \leq Q^{[n]}(x_0) \leq C e^{-\lambda n}.$$
\item\label{ItemQPrime1} If $q^\prime(0) = 1$, then for every $C>0$, $\lim_{n\to+\infty}e^{Cn}Q^{[n]}(x_0)=0$. If moreover, there exist positive constants $C_\ast,\alpha,x_*$ such that $C_\ast x_*^{\alpha}<1$ and $|q(x) - x|\leq C_\ast 2^{-\frac{\alpha}{2}} |x|^{1+\alpha}$ for $|x|\leq x_*$, then there exists a positive constant $\mu_*$ such that 
$$Q^{[n]}(x_0)\leq C_\ast^{-\frac1\alpha}e^{-\mu_*(1+\alpha)^n},$$
for $n$ large enough.
\end{enumerate}
\end{proposition}

\begin{proof}
By \eqref{Relation-Q-q}, one immediately gets that
\[Q^\prime(0) = \frac{1 - q^\prime(0)}{1 + q^\prime(0)}.\]
For $n \in \mathbb N$, set $x_n = Q^{[n]}(x_0)$ and remark that $(x_n)_{n \in \mathbb N}$ is a sequence of positive real numbers decreasing to zero.

We start by establishing Item~\ref{ItemQPrime0}. Notice that $\overline q$ is of class $\mathcal C^1$ and one has $\overline q(0) = \overline q^\prime(0) = 0$. Set $\varphi(x) = x - Q(x)$ and notice that $\varphi$ is $\mathcal C^1$ and verifies $\varphi(0) = \varphi^\prime(0) = 0$ and $0 < \varphi(x) < x$ for every $x > 0$. The sequence $(x_n)_{n \in \mathbb N}$ is then solution of the recurrence relation $x_{n+1} = x_n - \varphi(x_n)$ for $n \in \mathbb N$. Moreover, for $x > 0$ and setting $y = (q + \id)^{-1}(x)$, one has
\[\frac{\overline q(x)}{\varphi(x)} = \frac{q(y + q(y))}{q(y)} = 1 + \frac{1}{q(y)} \int_y^{y + q(y)} q^\prime(\xi) \diff \xi \to 0 \qquad \text{ as } y \to 0,\]
implying that $\varphi(x) \sim \overline q(x)$ as $x \to 0$.

For $n \in \mathbb N$, define $t_n = F(x_n)$. We first prove that $t_n \sim n$ as $n \to +\infty$. Notice that, for $n \geq 1$, one has 
\[
\frac{t_n}{n} = \frac{F(x_n)}{n} = \frac{1}{n} \sum_{k=0}^{n-1}\int_{x_{k+1}}^{x_k} \frac{\diff \xi}{\overline q(\xi)},
\]
hence, by Cesàro's theorem, the claim follows if one shows that
\[
\lim_{a \to 0} \int_{a - \varphi(a)}^{a} \frac{\diff \xi}{\overline q(\xi)} = 1,
\]
i.e., that
\[
\lim_{a \to 0} \frac{1}{\overline q(a)} \int_{a - \varphi(a)}^a \frac{\overline q(a) - \overline q(\xi)}{\overline q(\xi)} \diff \xi = 0,
\]
which would follow if
\begin{equation}
\label{EqUniformConvergence}
\lim_{a \to 0} \max_{\xi \in [a - \varphi(a), a]} \abs*{\frac{\overline q(a) - \overline q(\xi)}{\overline q(\xi)}} = 0.
\end{equation}
Notice that, since, for every $\xi \in [a - \varphi(a), a]$, $\overline q(\xi) = \overline q(a) + (\xi - a) \overline q^\prime(\xi_a)$, for some $\xi_a \in [\xi, a]$, one gets that
\[
\abs{\overline q(a) - \overline q(\xi)} \leq \varphi(a) \max_{\zeta \in [0, a]} \abs{\overline q^\prime(\zeta)}, \qquad \overline q(\xi) \geq \varphi(a) \left(\frac{\overline q(a)}{\varphi(a)} - \max_{\zeta \in [0, a]} \abs{\overline q^\prime(\zeta)}\right).
\]
One deduces that, for every $\xi \in [a - \varphi(a), a]$,
\[
\abs*{\frac{\overline q(a) - \overline q(\xi)}{\overline q(\xi)}} \leq \frac{\max_{\zeta \in [0, a]} \abs{\overline q^\prime(\zeta)}}{\frac{\overline q(a)}{\varphi(a)} - \max_{\zeta \in [0, a]} \abs{\overline q^\prime(\zeta)}}.\]
Since $\max_{\zeta \in [0, a]} \abs{\overline q^\prime(\zeta)}$ tends to $0$ as $a \to 0$, one immediately deduces \eqref{EqUniformConvergence}.

Let us now show that, under \eqref{EqConditionEquiv}, one has the stronger conclusion that $x_n \sim F^{-1}(n)$ as $n \to +\infty$. Notice that $x_n = F^{-1}(t_n) = F^{-1}(n) - (t_n - n) \overline q(F^{-1}(\xi_n))$ for some $\xi_n$ between $n$ and $t_n$. Set $z_n = F^{-1}(\xi_n)$, which tends to zero as $n \to +\infty$. Then, for $n \in \mathbb N$, 
\[
\frac{x_n}{F^{-1}(n)} = 1 - \rho_n \frac{z_n}{F^{-1}(n)} \qquad \text{ with } \rho_n = \frac{F(z_n)\overline q(z_n)}{z_n} \frac{t_n - n}{\xi_n}.
\]
In addition, one has that
\[
\abs*{\frac{t_n - n}{\xi_n}} = \left(\frac{\max(t_n, n)}{\min(t_n, n)} - 1\right) \frac{\min(t_n, n)}{\xi_n}\leq \left(\frac{\max(t_n, n)}{\min(t_n, n)} - 1\right).
\]
Since $t_n \sim n$ as $n \to +\infty$ and using \eqref{EqConditionEquiv}, one deduces that $\rho_n \to 0$ as $n \to +\infty$. Moreover, since $F^{-1}$ is decreasing, for every $n \in \mathbb N$, there exists $\alpha_n \in [0, 1]$ such that $z_n = (1 - \alpha_n) F^{-1}(n) + \alpha_n x_n$. Then, for $n$ large enough, one deduces, after an elementary computation, that
\[
\frac{x_n}{F^{-1}(n)} = \frac{1 - \rho_n (1 - \alpha_n)}{1 + \rho_n \alpha_n},
\]
yielding that $x_n \sim F^{-1}(n)$ as $n \to +\infty$.

We next turn to an argument for Item~\ref{ItemQPrimeBetween}. Notice first that $\lambda = -\ln Q^\prime(0)$. For $n \in \mathbb N$, one has
\[
x_{n+1}=Q'(0)x_n(1+Q_1(x_n))), 
\]
where $Q_1(x)=\frac{Q(x)}{x Q^\prime(0)} - 1$ for $x > 0$. Notice that $Q_1(x)$ tends to zero as $x$ tends to zero since $Q$ is differentiable at zero. Moreover, there exists $n_0\in\mathbb N$ so that $|Q_1(x_n)|<1$ for $n\geq n_0$. One deduces that, for $n\geq n_0$,
\[
x_ne^{\lambda n}=C(x_0,n_0)\prod_{k=n_0}^{n-1}(1+Q_1(x_k)),
\]
where $C(x_0,n_0)>0$ only depends on $x_0$ and $n_0$. Hence, 
\begin{equation}\label{eq:x_nBetween}
\frac{\ln(x_n)}n+\lambda =\frac{\ln C(x_0,n_0)}n +\frac1n\sum_{k=n_0}^{n-1}\ln(1+Q_1(x_k)).
\end{equation}
Since $Q_1(x_n) \to 0$ as $n$ tends to infinity, the first part of 
Item~\ref{ItemQPrimeBetween} follows at once. In particular, it implies that, for $n$ large enough, one has the estimate
\begin{equation}
\label{eq:exponential-estimate-xn}
e^{-\frac32\lambda n}\leq x_n\leq e^{-\frac23\lambda n}.
\end{equation}

Assume moreover that \eqref{eq:bizarre} holds true. Notice that, for $x$ small enough, one has
\begin{equation}
\label{eq:estimQxOverX}
\abs*{\frac{Q(x)}{x} - Q^\prime(0)} \leq 2 \abs*{\frac{q(y)}{y} - q^\prime(0)} \leq 2 \psi(\sqrt{2} x),
\end{equation}
where $y = (\id + q)^{-1}(\sqrt{2} x) \leq \sqrt{2} x$. Indeed, for $x > 0$, one has 
\begin{align*}
\frac{Q(x)}{x} - Q^\prime(0) & = 2 \left(\frac{y}{y + q(y)} - \frac{1}{1 + q^\prime(0)}\right) \\
& = \frac{2}{\left(1 + \frac{q(y)}{y}\right)\left(1 + q^\prime(0)\right)} \left(q^\prime(0) - \frac{q(y)}{y}\right),
\end{align*}
and \eqref{eq:estimQxOverX} follows since \ref{HypoSigma-gUpper} or \ref{HypoSigma-gLower} are supposed to hold true.

The second part of Item~\ref{ItemQPrimeBetween} is equivalent to the fact that $\ln(x_n)+\lambda n$ remains bounded as $n$ tends to infinity. By \eqref{eq:x_nBetween}, it is then enough to prove that the series of general term $\abs{Q_1(x_n)}$ is convergent. By \eqref{eq:exponential-estimate-xn} and \eqref{eq:estimQxOverX}, for $n$ large enough, one has
\[
\abs{Q_1(x_n)} \leq \frac{2 \psi(\sqrt{2} x_n)}{Q^\prime(0)} \leq \frac{2 \psi(\sqrt{2} e^{- \frac{2}{3} \lambda n})}{Q^\prime(0)} \leq \frac{2 \psi(e^{-\frac{\lambda}{2} n})}{Q^\prime(0)},
\]
where we also use the fact that $\psi$ is nondecreasing. The conclusion follows from \eqref{eq:bizarre}.

We finally prove Item~\ref{ItemQPrime1}. Note that $Q'(0)=0$ and hence $\frac{x_{n+1}}{x_n}$ tends to zero as $n$ tends to infinity.
Fix $C>0$. For every $n\geq 0$, one has 
\[e^{Cn}x_n=x_0\prod_{k=0}^{n-1}y_k,\qquad y_n:=e^C\frac{x_{n+1}}{x_n}.
\]
The first part of Item~\ref{ItemQPrime1} immediately follows since $\lim_{n\to+\infty}y_n=0$. For the second part, we start by noticing that, for every $x > 0$, one has $Q(x) = \frac{1}{\sqrt{2}} \left(y - q(y)\right)$ with $y = (\id + q)^{-1}(\sqrt{2} x)$. For $x \leq \frac{x_\ast}{\sqrt{2}}$ and since one has $y \leq \sqrt{2} x \leq x_\ast$, one deduces by assumption that
\[
\abs{Q(x)} \leq \frac{C_\ast}{\sqrt{2}} 2^{-\frac{\alpha}{2}} \abs{y}^{1 + \alpha} \leq C_\ast \abs{x}^{1 + \alpha}.
\]
Let $n_2\in\mathbb N$ be such that $x_n\leq \frac{x_*}{\sqrt{2}}$ for $n\geq n_2$. Then one has, for $n\geq n_2$,
\[
x_{n+1}\leq C_\ast x_n^{1+\alpha}.
\]
By setting $z_n = \ln\left(x_n C_\ast^{1/\alpha}\right)$, an elementary computation yields that $z_{n+1} \leq (1+\alpha) z_n$, hence, for $n \geq n_2$, one has $z_n \leq (1+\alpha)^{n - n_2} z_{n_2}$ and one gets the desired conclusion after setting
\[
\mu_* = -\frac{\ln(x_* C_\ast^{1/\alpha})}{(1+\alpha)^{n_2}}>0. \qedhere
\]
\end{proof}

\begin{remark}
Note that $F^{-1}: \mathbb R_+ \to (0, x_0]$ is equal to the solution $V$ of the Cauchy problem
\[
\frac{\diff }{\diff t} V(t) = -\overline q(V(t)), \qquad V(0) = x_0.
\]
This is how $F^{-1}$ is introduced in \cite{V-Martinez2000}. We postpone to Remark~\ref{RemkComparison} a comparison of our results from Proposition~\ref{prop:decay-rate}\ref{ItemQPrime0} with the corresponding ones of \cite{V-Martinez2000}.

Condition \eqref{EqConditionEquiv} is satisfied in several cases considered in the literature (see, e.g., \cite[Theorem~1.7.12, Examples~1 to 4]{Alabau2012Recent} and \cite{V-Martinez2000}), such as polynomial feedbacks of the form $q(s) = s \abs{s}^{p-1}$ for $p > 1$, or when $q$ is an odd function given for $s > 0$ by
$q(s) = e^{-\beta(s)}$ with $\beta(s) = 1/s^p$ for $p > 0$ or $\beta(s) = e^{1/s}$.

Concerning \eqref{eq:bizarre}, it holds true if for instance there exist positive $C,\alpha$ such that $\abs{q^\prime(x) - q^\prime(0)} \leq C\abs{x}^{\alpha}$ in a neighborhood of zero. Moreover one has quasi-optimality of \eqref{eq:bizarre} in the following sense: if the series of general term $Q_1(x_n)$ is unbounded, then $\limsup_{n\to+\infty}e^{\lambda n}Q^{[n]}(x_0) = +\infty$ or $\liminf_{n \to +\infty} e^{\lambda n} Q^{[n]}(x_0) = 0$.
\end{remark}

\begin{remark}
\label{remk:optimal-rate}
In the case where $q^\prime(0) = 0$ but \eqref{EqConditionEquiv} is not satisfied, Proposition~\ref{prop:decay-rate} ensures that $F(Q^{[n]}(x_0)) \sim n$ as $n \to +\infty$ but no direct information is provided on the asymptotic behavior of $Q^{[n]}(x_0)$. On the other hand, the techniques used in the proof of Proposition~\ref{prop:decay-rate}, together with standard techniques in analysis, propose a simple strategy to derive the asymptotic behavior of $Q^{[n]}(x_0)$ for any choice of $q$ with $q^\prime(0) = 0$, especially when \eqref{EqConditionEquiv} is not satisfied. To illustrate that strategy, consider the case of the function $q(x) = \frac{x}{(-\ln x)^p}$ for $x \in (0, M]$, where $p,M>0$ and $0 < q(x) < x$ and $\abs{q^\prime(x)} < 1$ for every $x > 0$, which appears in \cite[Theorem~1.7.12, Example~5]{Alabau2012Recent}. In that reference, the optimality of the decay rate associated with such a function $q$ is left open, as discussed in \cite{Alabau2012Recent} after its Theorem~1.7.16. In Appendix~\ref{app:decay}, we provide an answer to that open problem, by determining the precise decay rate of the corresponding the sequence $Q^{[n]}(x_0)$.
\end{remark}

The next proposition gathers some useful additional properties of $Q^{[n]}(x_0)$ under the assumptions of Proposition~\ref{prop:decay-rate}\ref{ItemQPrime0}.

\begin{proposition}
\label{prop:fnplusn0}
Let $q \in \mathcal C^1(\mathbb R_+, \mathbb R_+)$ be as in the statement of Proposition~\ref{prop:decay-rate}\ref{ItemQPrime0}, i.e., $q(0) = q^\prime(0) = 0$, $0 < q(x) < x$, and $\abs{q^\prime(x)} < 1$ for every $x>0$. Let $x_0 \in \mathbb R_+^\ast$ and $Q$ be defined from $q$ as in \eqref{Relation-Q-q}.
\begin{enumerate}
\item\label{ItemTimeTranslation} For every $n_0 \in \mathbb N$, one has $Q^{[n + n_0]}(x_0) \sim Q^{[n]}(x_0)$ as $n \to +\infty$.

\item\label{ItemDifferentInitialPoint} For every $y_0 \in \mathbb R_+^\ast$, one has $Q^{[n]}(x_0) \sim Q^{[n]}(y_0)$ as $n \to +\infty$.

\item\label{ItemSlowerThanExponential} For every $\varepsilon > 0$, we have
\[
\lim_{n \to +\infty} e^{\varepsilon n} Q^{[n]}(x_0) = +\infty.
\]
\end{enumerate}
\end{proposition}

\begin{proof}
To show \ref{ItemTimeTranslation}, note that, for $n \in \mathbb N$,
\[
\frac{Q^{[n + n_0]}(x_0)}{Q^{[n]}(x_0)}=\prod_{j=0}^{j=n_0-1}\frac{Q(Q^{[n+j]}(x_0))}{Q^{[n+j]}(x_0)}=\prod_{j=0}^{j=n_0-1}\Big(1-\frac{\varphi(Q^{[n+j]}(x_0))}{Q^{[n+j]}(x_0)}\Big),
\]
where $\varphi$ is defined as in the proof of Proposition~\ref{prop:decay-rate}\ref{ItemQPrime0}. Since $Q^{[n+j]}(x_0) \to 0$ as $n \to +\infty$ for every $j\in\{0,\cdots,n_0-1\}$ and $\varphi(0)=\varphi'(0)=0$, one gets the conclusion.

In order to prove \ref{ItemDifferentInitialPoint}, assume, with no loss of generality, that $y_0 \leq x_0$. Then there exists $n_0 \in \mathbb N$ such that $Q^{[n_0 + 1]}(x_0) < y_0 \leq Q^{[n_0]}(x_0)$, yielding that $Q^{[n + n_0 + 1]}(x_0) < Q^{[n]}(y) \leq Q^{[n + n_0]}(x_0)$, and one gets the conclusion from \ref{ItemTimeTranslation}.

Finally, to show \ref{ItemSlowerThanExponential}, let $F$ and $\overline q$ be defined as in the statement of Proposition~\ref{prop:decay-rate}\ref{ItemQPrime0}. Since $\overline q^\prime(0) = 0$, one has $0 < \overline q(s) < \frac{\varepsilon}{4} s$ for every $s \in (0, s_0)$, for some $s_0 > 0$. Hence $F(z) \geq \frac{3}{\varepsilon}\ln\left(\frac{1}{Q^{[n]}(x_0)}\right)$. Since $F(Q^{[n]}(x_0)) \sim n$ and $Q^{[n]}(x_0) \to 0$ as $n \to +\infty$, one deduces that, for $n$ large enough, $n \geq \frac{2}{\varepsilon}\ln\left(\frac{1}{Q^{[n]}(x_0)}\right)$, which is equivalent to $Q^{[n]}(x_0) \geq e^{-\frac{n \varepsilon}{2}}$, yielding the conclusion.
\end{proof}

A quick look at the argument in Proposition~\ref{prop:decay-rate}\ref{ItemQPrime1} for the decrease of $(x_n)_{n \in \mathbb N}$ faster than any exponential in the case $q^\prime(0) = 1$ may let one think that more precise decay rates can be obtained. However, it is not really the case, as explained in the following proposition, whose proof is given in Appendix~\ref{AppProofProp}.

\begin{proposition}
\label{prop:faster-than-exponential-but-not-so-fast}
Let $\varphi: \mathbb R_+ \to \mathbb R_+$ be an increasing function such that $\lim_{x\rightarrow +\infty} \varphi(x) \allowbreak = +\infty$. Then there exists $q \in \mathcal C^1(\mathbb R_+, \mathbb R_+)$ satisfying $q(0) = 0$, $0 < q(x) < x$, $\abs{q^\prime(x)} < 1$ for every $x > 0$, and $q^\prime(0) = 1$, such that, for every $x_0 > 0$,
\begin{equation}\label{eq:q'=0-x}
\liminf_{n \to +\infty} e^{n \varphi(n)} Q^{[n]}(x_0) > 0,
\end{equation}
where $Q$ is defined from $q$ as in \eqref{Relation-Q-q}.
\end{proposition}

\subsubsection{Decay rates for solutions}
\label{sec:decay-solution}

We now use Section~\ref{sec:decay-sequence} to derive results regarding decay rates for solutions of \eqref{eq:wave}. We start with the following consequence of Proposition~\ref{prop:ineq-sequences}.

\begin{theorem}\label{th:Qn} Let $\Sigma\subset \mathbb{R}^2$ be a damping set. 
\begin{enumerate}
\item\label{Itemh10} Suppose that $\Sigma$ satisfies \ref{HypoSigma-gUpper}. Let $M_0 > 0$ and $Q$ be the constant and the function whose existences are asserted in \ref{HypoSigma-gUpperTilde}.

\begin{enumerate}
\item\label{ItemItemXInftySmall} For $g_0\in \mathsf Y_\infty$ satisfying $\norm{g_0}_{\infty}\leq \frac{M_0}{\sqrt{2}}$, one has that, every solution $z$ of \eqref{eq:wave} starting at $\mathfrak I^{-1}(g_0)$, 
\begin{equation}\label{eq:h10-loc}
e_{\infty}(z)(t)\leq Q^{[\floor*{t/2}]}(e_{\infty}(z)(0)),\qquad \forall t\geq 0,
\end{equation}
and, for $p \in [1, +\infty)$,
\begin{equation}
\label{eq:h10-loc-p}
e_p(z)(t) \leq 2^{\frac1p}Q^{[\floor{t/2}]}(Z_p^{1/2}) + Z_p^{1/2} Q^{[\floor{t/2}]}(\max(Z_\infty, Z_p^{1/2})),\qquad \forall t\geq 0,
\end{equation}
where $Z_q = e_q(z)(0)$ for $q \in \{p, \infty\}$.

\item\label{ItemItemBoundUGASXinfty} Assume that the wave equation defined in \eqref{eq:wave} is UGAS in $\mathsf X_\infty$ with rate $\beta$. Then, for every solution $z$ of \eqref{eq:wave} in $\mathsf X_{\infty}$, one has
\begin{equation}\label{eq:h10-glob}
e_{\infty}(z)(t)\leq Q^{\left[\floor*{\frac{t-t_z}2}\right]}
\left(\min\Bigl(e_{\infty}(z)(0),\frac{M_0}{\sqrt{2}}\Bigr)\right),\qquad \forall t\geq t_z,
\end{equation}
where $t_z\geq 0$ is the first nonnegative time so that 
\[
\beta(e_{\infty}(z)(0), t_z) = \min\Bigl(e_{\infty}(z)(0),\allowbreak \frac{M_0}{\sqrt{2}}\Bigr).
\]

\item\label{ItemItemXp} Assume that \ref{HypoSigma-SectorInfty} holds and that the wave equation defined in \eqref{eq:wave} is UGAS in $\mathsf X_\infty$. Then, for every $p \in [1, +\infty]$ and every solution $z$ of \eqref{eq:wave} in $\mathsf X_p$, one has, for $t$ large enough,
\begin{equation}
\label{eq:estim-toto59}
e_p(z)(t) \leq 2 Q^{[\floor{t - t_1}]}(Z_p^{1/2}) + 2 Z_p^{1/2} Q^{\left[\floor*{t/2} - t_2\right]}(Z) + Z_p \mu^{\frac{t}{4} - \frac{1}{2}},
\end{equation}
where $Z_p = e_p(z)(0)$, $Z = \max(M, Z_p^{1/2})$, $M$ and $\mu$ are provided by \ref{HypoSigma-SectorInftyTilde}, and the times $t_1,t_2$ only depend on $Z_p$ and $Z$.
\end{enumerate}

\item\label{Itemh11} Suppose that $\Sigma$ satisfies \ref{HypoSigma-gLower}. Let $p\in [1, +\infty]$ and $z$ be a nontrivial solution of \eqref{eq:wave} in $\mathsf X_p$ starting at $\mathfrak I^{-1}(g_0)$ for some $g_0 \in \mathsf Y_p$. Assume moreover that either the Lebesgue measure of $\left\{s \in [-1, 1] \suchthat 0 < \abs{g_0(s)} \leq \frac{M}{\sqrt{2}}\right\}$ is positive or $\Sigma \cap \Delta = \{(0, 0)\}$, where $\Delta = \{(x, x) \suchthat x \in \mathbb R\}$. Then there exists positive constants $C_1, C_2$ only depending on the initial condition such that 
\begin{equation}\label{eq:h11-glob}
e_p(z)(t)\geq C_1 Q^{[\floor*{t/2}]}(C_2), \qquad \forall t\geq 0.
\end{equation}
\end{enumerate}
\end{theorem}

\begin{proof} To get the first part of Item~\ref{ItemItemXInftySmall}, simply notice by Proposition~\ref{PropDampingNonStrict} that
\[
e_{\infty}(z)(t)\leq e_{\infty}(z)(2\floor*{t/2}) =
\norm*{g_{\floor*{t/2}}}_{\infty},
\]
and, by \eqref{eq:real-iterated-Upper},
\begin{equation}
\label{eq:iterated-Upper}
\abs{g_n(s)} \leq Q^{[n]}(\abs{g_0(s)}), \qquad \text{ for } n \in \mathbb N \text{ and a.e.\ } s \in [-1, 1].
\end{equation}
The conclusion follows at once since $Q^{[n]}$ is increasing for every integer $n$. As regards the second part of Item~\ref{ItemItemXInftySmall}, we simply follow the argument of Proposition~\ref{prop:inf+p-fini} replacing the bounds using the $\mathcal{KL}$ function $\beta(\cdot, t)$ by $Q^{[\floor{t/2}]}(\cdot)$ and the conclusion follows from \eqref{eq:ep-fini-infini1}.

Since $e_\infty(z)(t_z) \leq \frac{M_0}{\sqrt{2}}$, Item~\ref{ItemItemBoundUGASXinfty} follows immediately by applying Item~\ref{ItemItemXInftySmall} to $g_0 = \mathfrak I(z(t_z, \cdot), z_t(t_z, \cdot))$.

As regards Item~\ref{ItemItemXp}, the argument consists in reproducing the proof of Theorem~\ref{thm-ugas} with modifications taking into account Item~\ref{ItemItemBoundUGASXinfty}. In the sequel, we use the notations of the proof of Theorem~\ref{thm-ugas}. Thanks to Item~\ref{ItemItemBoundUGASXinfty}, one can replace the estimate \eqref{eq:gchiLpE} by
\[
\norm{g_n \chi_E}_p \leq \alpha_E^{1/p} Q^{\left[\floor*{\frac{2n - t_1}{2}}\right]}\left(\min\left(Z_p^{1/2}, \frac{M_0}{\sqrt{2}}\right)\right),
\]
where $t_1 \geq 0$ only depends on $Z_p$. We use again Item~\ref{ItemItemBoundUGASXinfty} to replace equations \eqref{eq:gchiLpF} and \eqref{eq:estim-sum-n2} by
\[
\norm{g_n \chi_F}_p \leq Z_p^{1/2} Q^{\left[\floor*{\frac{2n - t_2}{2}}\right]}\left(\min\left(Z, \frac{M_0}{\sqrt{2}}\right)\right),
\]
and
\[
\sum_{k=0}^{\floor{n/2}} \norm{g_n \chi_k}_p^p \leq Z_p^{p/2} \left(Q^{\left[\floor*{\frac{n - 1 - t_2}{2}}\right]}\left(\min\left(Z, \frac{M_0}{\sqrt{2}}\right)\right)\right)^p,
\]
where $t_2 \geq 0$ only depends on $Z$. Putting together \eqref{eq:estim-sum-infty} and the previous inequalities, and up to increasing $t_1$ and $t_2$, we deduce \eqref{eq:estim-toto59}.

We finally provide an argument for Item \ref{Itemh11}. Denote by $(g_n)_{n \in \mathbb N}$ the sequence in $\mathsf Y_p$ corresponding to the solution $z$ in the sense of Proposition~\ref{PropGn}. For every $n \in \mathbb N$, define the measurable set $F_n\subset [-1, 1]$ by
\begin{equation}
\label{eq:Fn}
F_n = \left\{s\in [-1, 1] \suchthat 0 < \abs{g_n(s)} \leq \frac{M}{\sqrt{2}}\right\}.
\end{equation}
Assume first that $F_n$ is of zero Lebesgue measure for every $n\in \mathbb N$. In particular, $\Sigma \cap \Delta = \{(0, 0)\}$, i.e., $0 \notin S(x)$ for every $x \in \mathbb R^\ast$. Let $E = \{s \in [-1, 1] \suchthat g_0(s) \neq 0\}$. Since $z$ is nontrivial, the Lebesgue measure $\alpha_E$ of $E$ is positive. Moreover, our assumption on the sets $F_n$ and the fact that $\Sigma \cap \Delta = \{(0, 0)\}$ yield that, for almost every $s \in E$ and every $n \in \mathbb N$, one has $\abs{g_n(s)} \geq \frac{M}{\sqrt{2}}$. In particular, $\norm{g_n}_p \geq \alpha_E^{1/p} \frac{M}{\sqrt{2}}$, and the conclusion follows with $C_1 = \alpha_E^{1/p} \frac{M}{\sqrt{2}}$, $C_2 = 1$, and by using Propositions~\ref{PropGn} and \ref{PropDampingNonStrict} and the fact that $Q(x) \leq x$ for every $x \geq 0$.

If now there exists $n_0 \in \mathbb N$ so that $F_{n_0}$ has positive Lebesgue measure, then there exists $C_z>0$ and a subset $G_z$ of $F_{n_0}$ of positive measure $\alpha_z$ such that $C_z \leq \abs{g_{n_0}(s)} \leq \frac{M}{\sqrt{2}}$ for $s \in G_z$. By Proposition~\ref{prop:ineq-sequences}\ref{ItemIneqSequenceGeq} and by using the fact that $Q$ is increasing, it follows, for $n\geq n_0$, that $\abs{g_n(s)} \geq Q^{[n - n_0]}(\abs{g_{n_0}(s)}) \geq Q^{[n]}(C_z)$ for $s \in G_z$. Thus $\norm{g_n}_p \geq \alpha_z^{1/p} Q^{[n]}(C_z)$ and, from Propositions~\ref{PropGn} and \ref{PropDampingNonStrict}, one deduces that $e_p(z)(t) \geq \norm{g_{\floor{t/2} + 1}}_p \geq \alpha_z^{1/p} Q^{[\floor{t/2} + 1]}(C_z)$, whence the conclusion with $C_1 = \alpha_z^{1/p}$ and $C_2 = Q(C_z)$.
\end{proof}

\begin{remark}
\label{remk:hypotheses-ugas}
If we assume in \ref{ItemItemXp} that \ref{HypoSigma-SectorInfty} holds with a constant $M > 0$ equal to the constant $M_0 > 0$ from \ref{HypoSigma-gUpper}, then necessarily \eqref{eq:wave} is UGAS in $\mathsf X_\infty$. Indeed, in that case, by combining the bounds from \ref{HypoSigma-SectorInftyTilde} and \ref{HypoSigma-gUpperTilde}, we deduce that $\rho(r) < r$ for every $r > 0$, where $\rho$ is given by \eqref{eq:defiRho}. The fact that \eqref{eq:wave} is UGAS in $\mathsf X_\infty$ then follows from Propositions~\ref{PropCvSequence}\ref{ItemCvSequencePInfinite} and \ref{PropDefiRho}\ref{ItemIteratedUniform}.
\end{remark}

\begin{remark}
\label{Remk-ep}
It is useful to notice that, for $p = +\infty$, we have deduced the estimate \eqref{eq:h10-loc} immediately from \eqref{eq:iterated-Upper}. One may wonder whether it is possible to deduce from \eqref{eq:iterated-Upper} a similar estimate replacing $e_\infty$ by $e_p$ for finite $p$. This is indeed the case under the extra assumption that $Q$ is concave (or, equivalently, that $q$ is convex): one deduces from Lemma~\ref{LemmaJensen} in Appendix~\ref{AppLemmas} that
\begin{equation}
\label{eq:estim-ep}
e_p(z)(t) \leq 2^{1/p} Q^{[\floor{t/2}]}(2^{-1/p} e_p(z)(0)),
\end{equation}
as soon as $e_\infty(z)(0) \leq \frac{M_0}{\sqrt{2}}$. Note that \eqref{eq:estim-ep} holds true for every solution  of \eqref{eq:wave} in $\mathsf X_p$ as soon as one assumes that \ref{HypoSigma-gUpper} holds \emph{globally}, i.e., $M = +\infty$ in its definition. In that case, the wave equation in \eqref{eq:wave} is UGAS in $\mathsf X_p$. This is in accordance with Theorem~\ref{thm-ugas} since \ref{HypoSigma-SectorInfty} holds if $Q$ is concave in $\mathbb R_+$ and \ref{HypoSigma-gUpper} holds globally.
\end{remark}

Theorem~\ref{th:Qn} together with Propositions~\ref{prop:decay-rate} and \ref{prop:fnplusn0} provide accurate estimates for the behavior of trajectories of \eqref{eq:wave} as time tends to $+\infty$. Moreover, the optimality of such estimates can be addressed using Theorem~\ref{th:Qn}\ref{Itemh11}. Even though this can be done in full generality in the case where $\Sigma$ is the graph of a function $q$ or $q^{-1}$, with $q \in \mathcal C^1(\mathbb R_+, \mathbb R_+)$ satisfying the statements in \ref{HypoSigma-gUpper} or \ref{HypoSigma-gLower}, we only focus in the sequel on the case where $q^\prime(0) = 0$ since it is the one usually addressed in the literature (see, e.g., \cite{V-Martinez2000, Alabau2012Recent}). More precisely, we have the following result.

\begin{corollary}
\label{coro:equiv-einfty}
Let $\Sigma \subset \mathbb R^2$ be a damping set satisfying \ref{HypoSigma-SectorInfty} and $M > 0$ be the constant from \ref{HypoSigma-SectorInfty}. Assume moreover that
\begin{enumerate}
\item $\Sigma \cap \Delta = \{(0, 0)\}$, where $\Delta = \{(x, x) \suchthat x \in \mathbb R\}$; and
\item\label{ItemAssumptionH10AndH11} there exists a function $q \in \mathcal C^1(\mathbb R_+, \mathbb R_+)$ with $q(0) = q^\prime(0) = 0$, $0 < q(x) < x$, and $\abs{q^\prime(x)} < 1$ for every $x>0$ such that
\[
\abs{y} = q(\abs{x}) \quad \text{ or } \quad \abs{x} = q(\abs{y}), \qquad \text{ for every } (x, y) \in \Sigma \cap B(0, M).
\]
\end{enumerate}
Then, for every $p \in [1, +\infty]$ and every nontrivial solution $z$ of \eqref{eq:wave} in $\mathsf X_p$, there exist positive constants $C_1, C_2$ such that, for every $t \geq 0$,
\[C_1 Q^{[\floor{t/2}]}(1) \leq e_p(z)(t) \leq C_2 Q^{[\floor{t/2}]}(1),\]
where $Q$ is the function defined from $q$ in \eqref{Relation-Q-q}.
\end{corollary}

\begin{proof}
Note that assumption \ref{ItemAssumptionH10AndH11} implies that both \ref{HypoSigma-gUpper} and \ref{HypoSigma-gLower} are satisfied with the same function $q$ and the same $M > 0$. Moreover, as noticed in Remark~\ref{remk:hypotheses-ugas}, \ref{ItemAssumptionH10AndH11} and \ref{HypoSigma-SectorInfty} imply that \eqref{eq:wave} is UGAS in $\mathsf X_\infty$. Then, by Theorem~\ref{th:Qn}\ref{ItemItemXp} and \ref{Itemh11}, one deduces that \eqref{eq:estim-toto59} and \eqref{eq:h11-glob} hold, and the conclusion follows by Proposition~\ref{prop:fnplusn0}.
\end{proof}

\begin{remark}
\label{RemkComparison}
Our results given in Proposition~\ref{prop:decay-rate}\ref{ItemQPrime0}, Theorem~\ref{th:Qn} (equations \eqref{eq:h10-loc}, \eqref{eq:h10-glob}, and \eqref{eq:h11-glob} with $p = +\infty$), and Corollary~\ref{coro:equiv-einfty} are directly inspired by a string of results of \cite{V-Martinez2000}, namely Theorem~2.1 and all the results from Section~3 of that reference. Because of the flexibility of our approach, we provide simpler proofs and we are able to relax some assumptions on the function $q$ and the set of initial conditions for which the appropriate estimates hold true.

Note that Theorem~\ref{th:Qn}\ref{ItemItemXp}, together with the assumption that the function $q$ from \ref{HypoSigma-gUpper} satisfies $q^\prime(0) = 0$ and \eqref{EqConditionEquiv}, yields the estimate
\[
e_p(z)(t) \leq 3 (1 + e_p(z)(0)^{1/2}) F^{-1}(t/2)
\]
for every $p \in [1, +\infty]$, every solution $z$ of \eqref{eq:wave} in $\mathsf X_p$, and $t$ large enough, where $F$ is as in Proposition~\ref{prop:decay-rate}\ref{ItemQPrime0}. Indeed, this follows from \eqref{eq:estim-toto59}, Proposition~\ref{prop:decay-rate}\ref{ItemQPrime0}, 
and manipulations similar to those in the argument of Corollary~\ref{coro:equiv-einfty}. This estimate can be compared to that of \cite[Theorem~2.1(b)]{Zuazua1990Uniform}, which obtains a similar estimate for wave equations in space dimension up to $3$ but with stronger assumptions on $q$.
\end{remark}

An interesting instance of the preceding results is the classical linear case, which corresponds to $\Sigma$ satisfying \ref{HypoSigma-SectorZero}. As a consequence of the above results and remarks, we have the following.

\begin{corollary}
\label{thm-decay}
Let $\Sigma \subset \mathbb R^2$ be a damping set.
\begin{enumerate}
\item\label{ItemTheoCvExpoLocal} If $\Sigma$ satisfies \ref{HypoSigma-SectorZero} and \eqref{eq:wave} is UGAS in $X_\infty$, then there exists $\lambda > 0$ such that, for every $p \in [1, +\infty]$ and every initial condition $(z_0,z_1) \in \mathsf X_\infty$, every solution $z$ of \eqref{eq:wave} starting at $(z_0,z_1)$ satisfies
\begin{equation}
\label{eq:exp-decay}
e_p(z)(t) \leq C e^{-\lambda t} e_p(z)(0), \qquad \text{ for every } t\geq 0,
\end{equation}
where the constant $C > 0$ depends only on $\lambda$ and $\norm*{(z_0,z_1)}_{\mathsf X_\infty}$.

\item\label{ItemTheoCvExpoGlobal} If there exist positive constants $a, b$ such that $a \abs{x} \leq \abs{y} \leq b \abs{x}$ for every $(x, y) \in \Sigma$, then \eqref{eq:wave} is GES in $\mathsf X_p$ for every $p \in [1, +\infty]$. More precisely, there exist $C > 0$ and $\lambda > 0$ such that, for every $p \in [1, +\infty]$ and every initial condition $(z_0,z_1) \in \mathsf X_p$, every solution $z$ of \eqref{eq:wave} starting at $(z_0,z_1)$ satisfies
\begin{equation}
\label{eq:exp-decay-bis}
e_p(z)(t) \leq C e^{-\lambda t} e_p(z)(0), \qquad \text{ for every } t\geq 0.
\end{equation}
\end{enumerate}
\end{corollary}

\begin{proof}
Item~\ref{ItemTheoCvExpoLocal} can be deduced by combining the previous results and remarks and noticing that \ref{HypoSigma-SectorZero} is a particular case of \ref{HypoSigma-gUpper} with a linear function $q$. We choose however to provide the following direct ad hoc argument.

Let $M > 0$ and $\mu \in (0, 1)$ be as in \ref{HypoSigma-SectorZeroTilde} and $\beta$ be the $\mathcal{KL}$ function provided by the UGAS assumption in $\mathsf X_\infty$. Let $z$ be a solution of \eqref{eq:wave} in $\mathsf X_\infty$ and consider the corresponding sequence $(g_n)_{n \in \mathbb N}$ in $\mathsf Y_\infty$ from Proposition~\ref{PropGn}. Set $R = \norm{g_0}_\infty$, let $t_R > 0$ be such that $\beta(R, t_R) \leq \frac{M}{\sqrt{2}}$, and define $n_R = \ceil{t_R/2}$. For $n \geq n_R$ and a.e.\ $s \in [-1, 1]$, one has
\[
\abs{g_n(s)} \leq e_\infty(z)(2n) \leq \beta(R, 2n) \leq \beta(R, t_R) \leq \frac{M}{\sqrt{2}}.
\]
It follows from \ref{HypoSigma-SectorZeroTilde} that, for $n \geq n_R$, one has
\begin{equation}
\label{eq:ExpoEstimGn}
\norm{g_n}_p \leq \mu^{n - n_R} \norm*{g_{n_R}}_p \leq \mu^{n - n_R} \norm*{g_0}_p.
\end{equation}
The previous inequality still holds for $n \in \{0, \dotsc, n_R-1\}$ since $(g_n)_{n \in \mathbb N}$ is nonincreasing. One immediately gets the conclusion with $\lambda = - \frac{\ln \mu}{2}$ using Propositions~\ref{PropGn} and \ref{PropDampingNonStrict}.

Part \ref{ItemTheoCvExpoGlobal} follows immediately since, from its assumptions, one deduces that \eqref{eq:ExpoEstimGn} holds for every $n \in \mathbb N$ with $n_R = 0$.
\end{proof}

We next provide an alternative proof to Corollary~\ref{thm-decay}\ref{ItemTheoCvExpoLocal} at the price of strengthening Hypothesis \ref{HypoSigma-SectorZero} to the following hypothesis on a damping set $\Sigma \subset \mathbb R^2$: 
\begin{itemize}[labelwidth=\widthof{\normalsize (H7)$_{\ast}$}, leftmargin=!]
\item[\customlabel{HypoSigma-SectorZeroAst}{\ref{HypoSigma-SectorZero}$_\ast$}] For every $M > 0$, there exists $\mu \in (0, 1)$ such that $\abs{y} \leq \mu \abs{x}$ for every $(x, y) \in R \Sigma \cap B(0, M)$.
\end{itemize}
It is immediate to see that, under \ref{HypoSigma-SectorZeroAst}, real iterated sequences for $S$ converge uniformly to zero and hence \eqref{eq:wave} is UGAS in $\mathsf X_\infty$, therefore \ref{HypoSigma-SectorZeroAst} is stronger than the assumptions of Corollary~\ref{thm-decay}\ref{ItemTheoCvExpoLocal}.

Nevertheless, we can prove Corollary~\ref{thm-decay}\ref{ItemTheoCvExpoLocal} under \ref{HypoSigma-SectorZeroAst} for $p$ finite with an argument having its own interest, which is given next. Consider the Lyapunov function $V_p$, based on a variant of $e_p$, and defined, for a solution $z$ of \eqref{eq:wave} in $\mathsf X_p$, by
\begin{equation}\label{eq:V_p-1}
V_p(t) =\int_0^1 e^{\nu x} F(f(t+x)) \diff x + \int_0^1 e^{-\nu x} F(g(t-x)) \diff x, 
\quad t\geq 0,
\end{equation}
where $f$ and $g$ are the functions associated with $z$ from Definition~\ref{DefWeakSolution}, $\nu$ is a positive constant to be chosen later, and $F(s) = \abs*{s}^p$ for $s\in\mathbb{R}$. Note that $V_p$ coincides with $e_p^p(z)$ if $\nu = 0$. In the case $p = 2$, the corresponding Lyapunov function $V_2$ has been extensively used in the literature of hyperbolic systems of conservation laws, as detailed in \cite{bastin2016stability}.

We first provide another expression for $V_p$ before taking its time derivative along trajectories of \eqref{eq:wave}. From \eqref{eq:V_p-1} and \eqref{relation-f-g}, one gets for $t \geq 0$ that
\begin{equation}\label{eq:V_p-2}
V_p(t) =e^{-\nu t}\int_{t-1}^{t+1} e^{\nu s} F(g(s)) \diff s.
\end{equation}
Hence $V_p$ is absolutely continuous and its time derivative along the trajectories of \eqref{eq:wave} satisfies
\begin{equation}\label{eq:derVp}
\frac{\diff V_p}{\diff t}(t) = - \nu V_p(t) + e^{\nu} F(g(t+1))-e^{-\nu}F(g(t-1)), \qquad t\geq 0.
\end{equation}
Since $g(t+1)\in S(g(t-1))$ for $t\geq 0$, we deduce from \ref{HypoSigma-SectorZeroAst} that $\abs{g(t+1)}\leq \mu \abs{g(t-1)}$, where $\mu \in (0, 1)$ is obtained by taking $M = \norm{(z_0, z_1)}_{\mathsf X_\infty}$ in \ref{HypoSigma-SectorZeroAst} and $(z_0, z_1)$ is the initial condition of $z$. Then, from \eqref{eq:derVp}, one deduces that
\begin{equation*}
\frac{\diff V_p}{\diff t}(t) \leq  - \nu V_p(t) + e^{-\nu} F(g(t-1)) \Bigl(\mu^p e^{2\nu}-1\Bigr), \quad t\geq 0.
\end{equation*}
Setting $\nu = -\frac{p}{2}\ln(\mu)$, one obtains that
\begin{equation*}
\frac{\diff V_p}{\diff t}(t) \leq -\nu V_p(t),\quad t\geq 1,
\end{equation*}
which implies the exponential convergence of trajectories of \eqref{eq:wave} with a decay rate depending on $\norm{(z_0,z_1)}_{\mathsf X_\infty}$. It is now standard to obtain \eqref{eq:exp-decay} with $\nu$ independent of the initial condition, cf.\ \cite{HaidarLyapunov}.
\begin{remark}
Note that the previous argument stills works without considering the parameter $\nu$ in the definition of $V_p$ in \eqref{eq:V_p-1}, in which case $V_p=e_p^p(z)$. Equation \eqref{eq:derVp} becomes, after taking into account \ref{HypoSigma-SectorZeroAst},
\[\frac{\diff V_p}{\diff t}(t) \leq -(1-\mu^p)F(g(t-1)), \qquad \text{ for a.e.\ }t\geq 0.
\]
Integrating between $t-1$ and $t+1$ for $t\geq 1$, one gets that
\[e_p(z)(t+1)\leq \mu e_p(z)(t-1),
\]
which yields exponential decay. More generally, at the light of Remark~\ref{remk:energy-derivative} and without assuming necessarily  \ref{HypoSigma-SectorZeroAst}, one can follow  a reasoning relying on  a Lyapunov function of the form \eqref{eq:V_p-1} with any positive definite function $F$ with no additional regularity assumption on the solution $z\in \mathsf X_p$ of \eqref{eq:wave}, as soon as a damping and sector conditions are satisfied.
\end{remark}
We close this section by providing a necessary and sufficient condition of GES.

\begin{corollary}
Let $\Sigma \subset \mathbb R^2$ be a damping set, $p \in [1, +\infty]$, and $S$ be the set-valued map whose graph is $R \Sigma$. Then \eqref{eq:wave} is GES in $\mathsf X_p$ if and only if there exists $\mu \in (0, 1)$ and $n_0 \in \mathbb N^\ast$ such that, for every $x \in \mathbb R$, one has $\norm{S^{[n_0]}(x)} \leq \mu \abs{x}$.
\end{corollary}

\begin{proof}
Assume that there exists $\mu \in (0, 1)$ and $n_0 \in \mathbb N^\ast$ such that, for every $x \in \mathbb R$, one has $\norm{S^{[n_0]}(x)} \leq \mu \abs{x}$. Let $z$ be a solution of \eqref{eq:wave} in $\mathsf X_p$ and consider the sequence $(g_n)_{n \in \mathbb N}$ in $\mathsf Y_p$ corresponding to $z$ according to Proposition~\ref{PropGn}. Then, for every $k \in \mathbb N$, one has
\[
\norm{g_{k n_0}}_p \leq \mu^k \norm{g_0}_{p},
\]
and, by using Propositions~\ref{PropGn} and \ref{PropDampingNonStrict}, one gets that \eqref{eq:wave} is GES.

Conversely, assume that \eqref{eq:wave} is GES and let $C, \lambda$ be positive constants such that $e_p(z)(t) \leq C e^{-\lambda t} e_p(z)(0)$ for every $t \geq 0$ and every solution $z$ of \eqref{eq:wave} in $\mathsf X_p$. Let $n_0 \in \mathbb N^\ast$ be such that $C e^{-2 \lambda n_0} < 1$. By considering constant initial conditions, one gets the conclusion with $\mu = C e^{-2 \lambda n_0}$.
\end{proof}

\subsection{Arbitrary slow convergence}
\label{SecSloooooooow}

In this subsection, we positively answer a conjecture posed in \cite[Theorem~4.1, Remark~2]{V-Martinez2000} 
regarding the worst possible decay rate of solutions of \eqref{eq:wave} when $\Sigma$ is of saturation type.

\begin{theorem}
\label{TheoremSloooooooow}
For $C > 0$, define $\mathcal R_C = \{(x, y) \in \mathbb R^2 \suchthat \abs{x} \leq C/\sqrt{2} \text{ or } \abs{y} \leq C/\sqrt{2}\}$. Let $\Sigma \subset \mathbb R^2$ be a damping set such that $\Sigma \subset \mathcal R_C$ for some $C > 0$. Then, for every $p \in [1, +\infty)$ and every decreasing function $\varphi: \mathbb R_+ \to \mathbb R_+^\ast$ tending to $0$ as $t \to +\infty$, there exists $(z_0, z_1) \in \mathsf X_p$ such that every solution $z$ of \eqref{eq:wave} with initial condition $(z_0, z_1)$ satisfies, for every $t \geq 0$,
\[
e_p(z)(t) \geq \varphi(t).
\]
\end{theorem}

A graphical representation of the region $\mathcal R_C$ and the rotated region $R \mathcal R_C$ is provided in Figure~\ref{FigTheoremSlow}. The darker shade represents the intersection between $\mathcal R_C$ and the damping region $\{(x, y) \in \mathbb R^2 \suchthat x y \geq 0\}$.

\begin{figure}[ht]
\centering
\begin{tikzpicture}

\fill[blue!20!white] (-3, -0.8) -- (-0.8, -0.8) -- (-0.8, -3) -- (0.8, -3) -- (0.8, -0.8) -- (3, -0.8) -- (3, 0.8) -- (0.8, 0.8) -- (0.8, 3) -- (-0.8, 3) -- (-0.8, 0.8) -- (-3, 0.8) -- cycle;

\fill[blue!40!white] (-3, -0.8) -- (-0.8, -0.8) -- (-0.8, -3) -- (0, -3) -- (0, 3) -- (0.8, 3) -- (0.8, 0.8) -- (3, 0.8) -- (3, 0) -- (-3, 0) -- cycle;

\draw[semithick, -Stealth] (-3, 0) -- (3, 0);
\draw[semithick, -Stealth] (0, -3) -- (0, 3);

\draw (1.5, 0) node[above] {$\mathcal R_C$};

\draw[semithick, -Stealth] (3.5, -0.25) -- node[midway, above] {$R$} (4.5, -0.25);

\begin{scope}[shift={(8, 0)}, rotate=-45]
\clip (0, {3*sqrt(2)}) -- ({3*sqrt(2)}, 0) -- (0, {-3*sqrt(2)}) -- ({-3*sqrt(2)}, 0) -- cycle;

\fill[blue!20!white] (-4.5, -0.8) -- (-0.8, -0.8) -- (-0.8, -4.5) -- (0.8, -4.5) -- (0.8, -0.8) -- (4.5, -0.8) -- (4.5, 0.8) -- (0.8, 0.8) -- (0.8, 4.5) -- (-0.8, 4.5) -- (-0.8, 0.8) -- (-4.5, 0.8) -- cycle;

\fill[blue!40!white] (-4.5, -0.8) -- (-0.8, -0.8) -- (-0.8, -4.5) -- (0, -4.5) -- (0, 4.5) -- (0.8, 4.5) -- (0.8, 0.8) -- (4.5, 0.8) -- (4.5, 0) -- (-4.5, 0) -- cycle;
\end{scope}

\draw[semithick, -Stealth] (5, 0) -- (11, 0);
\draw[semithick, -Stealth] (8, -3) -- (8, 3);

\node at (9.5, 1.5) {$R \mathcal R_C$};

\end{tikzpicture}
\caption{Regions $\mathcal R_C$ and $R \mathcal R_C$.}
\label{FigTheoremSlow}
\end{figure}
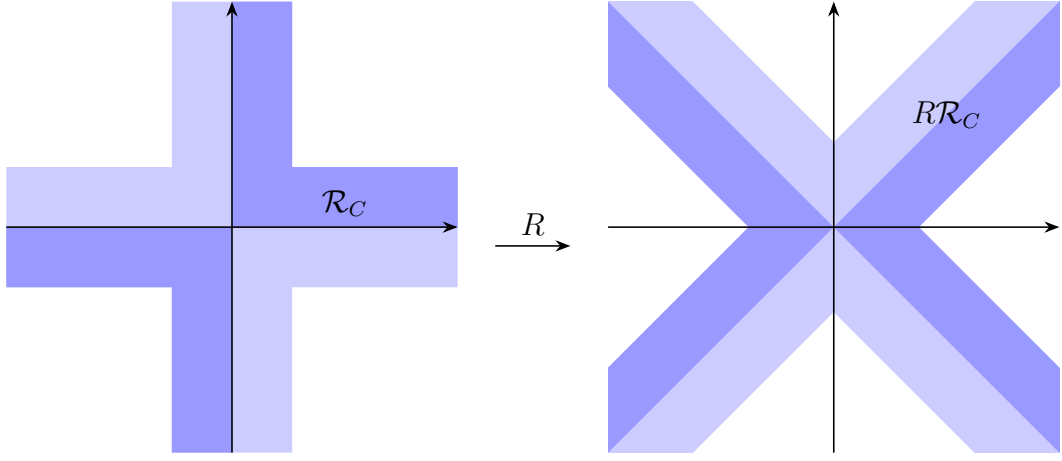

\begin{proof}
Let $S$ be the set-valued map whose graph is $R \Sigma$. The assumptions on $\Sigma$, namely the fact that $\Sigma$ is a damping set and that $\Sigma\subset \mathcal{R}_C$, imply that, for every $x \in \mathbb R$ and $y \in S(x)$, one has $\abs{x} - C \leq \abs{y} \leq \abs{x}$.

From Propositions~\ref{PropGn} and \ref{PropDampingNonStrict}, it is enough to construct $g_0 \in \mathsf Y_p$ such that, for every iterated sequence $(g_n)_{n \in \mathbb N}$ in $\mathsf Y_p$ for $S$ starting at $g_0$, one has
\begin{equation}
\label{IneqDiscreteTime}
\norm{g_n}_{p} \geq \varphi(2(n-1)), \qquad \forall n \in \mathbb N^\ast.
\end{equation}
Up to dividing the sequence $(g_n)_{n \in \mathbb N}$ and the function $\varphi$ by $C$, we assume with no loss of generality that $C = 1$.

Let $C_p = \frac{1}{3^{p-1}}$ and $\phi: \mathbb R_+ \to \mathbb R_+$ be a decreasing function such that $\phi(t) = \frac{1}{C_p} \varphi^p(2(\frac{t^{1/p}}{3}-1))$ for every $t \geq 3^p$. Note that $C_p \phi(3^p n^p) = \varphi^p(2(n-1))$ for every $n \in \mathbb N^\ast$. Let $(a_n)_{n \in \mathbb N}$ and $(b_n)_{n \in \mathbb N}$ be the sequences obtained by applying Lemma~\ref{LemmaSequences} in Appendix~\ref{AppLemmas} to $\phi$. Define $g_0: [-1, 1] \to \mathbb R_+$ by
\begin{equation}
\label{DefiG0}
g_0(s) = \sum_{k=0}^\infty k^{1/p} \chi_{(a_{k+1}, a_k]}(s),
\end{equation}
and consider an iterated sequence $(g_n)_{n \in \mathbb N}$ for $S$ starting at $g_0$. Using the fact that the sequence $(b_n)_{n\in\mathbb N}$ converges to $0$ and the relationship between the sequences $(a_n)_{n\in \mathbb N}$ and $(b_n)_{n\in\mathbb N}$, it is easy to see that
\[
\norm{g_0}_p^p = \sum_{k=0}^\infty k (a_k - a_{k+1}) = \sum_{k=1}^\infty a_k = b_0 < +\infty,
\]
and thus $g_0 \in \mathsf Y_p$, which implies that $g_n \in \mathsf Y_p$ for every $n \in \mathbb N$. Moreover, Proposition~\ref{prop:dis-sup} shows that, for every $n \geq 1$, one has
\[
g_n(s) = \sum_{k=0}^\infty \alpha_{k, n} \chi_{(a_{k+1}, a_k]}(s), \qquad \text{ with } \alpha_{k, n} \in S^{[n]}(k^{1/p}) \text{ for }k\in\mathbb N.
\]

Notice that, for every $n \in \mathbb N$ and $k \geq n^p$, one has $\abs{\alpha_{k, n}} \geq k^{1/p} - n$. Then, for every $n \in \mathbb N$, one has
\begin{equation}
\label{eq:estim-norm-gn-below}
\begin{aligned}
\norm{g_n}_p^p & \geq \sum_{k \geq n^p} (k^{1/p} - n)^p (a_k - a_{k+1}) \\
& = \sum_{k \geq n^p+1} \left[\left(k^{1/p} - n\right)^p - \left((k-1)^{1/p} - n\right)^p\right] a_k.
\end{aligned}
\end{equation}
We next prove that, for every $n \in \mathbb N^\ast$ and $k \geq 3^p n^p$, one has
\begin{equation}
\label{eq:bound-to-be-proved}
\left(k^{1/p} - n\right)^p - \left((k-1)^{1/p} - n\right)^p \geq C_p.
\end{equation}
To see that, we rewrite the left-hand side of \eqref{eq:bound-to-be-proved} as $k (A^p - B^p)$ with
\[
A = 1 - \frac{n}{k^{1/p}}, \qquad B = \left(1 - \frac{1}{k}\right)^{1/p} - \frac{n}{k^{1/p}}.
\]
Using that $k \geq 3^p n^p$, one deduces at once that $\frac{1}{3} \leq B \leq A \leq 1$. On the other hand, there exists $\Gamma \in [B, A] \subset \left[\frac{1}{3}, 1\right]$ such that $A^p - B^p = p (A-B) \Gamma^{p-1}$. Since
\[A - B = 1 - \left(1 - \frac{1}{k}\right)^{1/p} \geq \frac{1}{p k} \qquad \text{ for every } k \in \mathbb N^\ast,\]
one concludes that \eqref{eq:bound-to-be-proved} holds.

It now follows from \eqref{eq:estim-norm-gn-below}, \eqref{eq:bound-to-be-proved}, and the expression of the sequence $(a_k)_{k\in\mathbb N}$ with respect to $(b_k)_{k\in \mathbb N}$ that
\[
\norm{g_n}_p^p \geq C_p b_{\ceil{3^p n^p}-1} \geq C_p \phi(3^p n^p) = \varphi^p (2 (n-1)), \qquad n\geq 1,
\]
as required.
\end{proof}

\begin{remark}
Note that Theorem~\ref{TheoremSloooooooow} does not hold for $p=+\infty$. Indeed, assume that $\Sigma \subset \mathbb R^2$ is the graph of the piecewise linear saturation function $\sigma$ defined by $\sigma(x) = x$ for $\abs{x}\leq 1$ and $\sigma(x) = \frac{x}{\abs{x}}$ for $\abs{x}\geq 1$. Then $R \Sigma$ is the graph of the function $S$ given by $S(x) = 0$ for $\abs{x} \leq \sqrt{2}$ and $S(x) = -x + \sqrt{2}\frac{x}{\abs{x}}$ for $\abs{x} \geq \sqrt{2}$. A straightforward computation using Proposition~\ref{PropGn} shows that, for every $(z_0, z_1) \in \mathsf X_\infty$, we have $e_\infty(z)(t) = 0$ for every $t \geq 2 \ceil*{e_\infty(z)(0) / \sqrt{2}}$, and thus an estimate such as that of Theorem~\ref{TheoremSloooooooow} cannot hold.
\end{remark}

\section{Boundary perturbations}
\label{SecISS}

In this section, we show that the framework developed previously to address the stability of \eqref{eq:wave} can also be applied to handle wave equations with disturbances in the boundary condition $(z_t(t, 1), -z_x(t, 1)) \in \Sigma$. More precisely, the disturbed version of \eqref{eq:wave} we consider in this section is
\begin{equation}
\label{eq:wave-disturbed}
\left\{
\begin{aligned}
& z_{tt}(t, x) = z_{xx}(t, x), & \qquad & (t, x) \in \mathbb R_+ \times [0, 1], \\
& z(t, 0) = 0, & & t \in \mathbb R_+, \\
& (z_t(t, 1), -z_x(t, 1)) \in \Sigma + d(t), & & t \in \mathbb R_+, \\
& z(0, x) = z_0(x), & & x \in [0, 1], \\
& z_t(0, x) = z_1(x), & & x \in [0, 1],
\end{aligned}
\right.
\end{equation}
where $d: \mathbb R_+ \to \mathbb R^2$ is a measurable function representing the disturbance.

Given $p \in [1, +\infty]$ and a disturbance $d$ as above, solutions of \eqref{eq:wave-disturbed} in $\mathsf X_p$ can be defined with an obvious modification of Definition~\ref{DefWeakSolution}, consisting in replacing the set $\Sigma$ in the boundary condition at $x = 1$ in \eqref{eq:defWeakSol} by $\Sigma + d(t)$. Proposition~\ref{prop:charac1-sol} still holds for \eqref{eq:wave-disturbed} after replacing \eqref{recurrence-g} by
\begin{equation}
\label{recurrence-g-disturbed}
(g(s-2), g(s)) \in R \Sigma + R d(s - 1), \qquad \text{for a.e.\ } s \geq 1.
\end{equation}
One may use the one-to-one correspondence $\Seq$ from Definition~\ref{DefiSeqI} between elements $g \in L^p_{\loc}(-1, +\infty)$ and sequences $(g_n)_{n \in \mathbb N}$ in $\mathsf Y_p$ to rewrite \eqref{recurrence-g-disturbed} as
\begin{equation}
\label{eq:gn-disturbed}
(g_n(s), g_{n+1}(s)) \in R \Sigma + \delta_n(s), \qquad n \in \mathbb N, \text{ a.e.\ } s \in [-1, 1],
\end{equation}
where the sequence of measurable functions $(\delta_n)_{n \in \mathbb N}$ is defined from $d$ by
\begin{equation}
\label{eq:delta-from-d}
\delta_n(s) = R d(s + 2n + 1) \qquad \text{for every } n \in \mathbb N \text{ and } s \in [-1, 1].
\end{equation}

As regards existence and uniqueness of solutions of \eqref{eq:wave-disturbed}, similarly to Theorems~\ref{TheoExist} and \ref{TheoUnique}, we deduce in a straightforward manner the following result.

\begin{theorem}
Let $\Sigma \subset \mathbb R^2$ and $p \in [1, +\infty]$.
\begin{enumerate}
\item If \ref{HypoSigma-Exists} holds and $p < +\infty$, or if \ref{HypoSigma-ExistsInfty} holds and $p = +\infty$, then, for every $(z_0, z_1) \in \mathsf X_p$ and every $d \in L^p_{\loc}(\mathbb R_+, \mathbb R^2)$, there exists a solution of \eqref{eq:wave-disturbed} in $\mathsf X_p$ with initial condition $(z_0, z_1)$.

\item For every $(z_0, z_1) \in \mathsf X_p$ and every $d \in L^p_{\loc}(\mathbb R_+, \mathbb R^2)$, there exists a unique solution of \eqref{eq:wave-disturbed} in $\mathsf X_p$ with initial condition $(z_0, z_1)$ if and only if either \ref{HypoSigma-Unique} holds and $p < +\infty$, or \ref{HypoSigma-UniqueInfty} holds and $p = +\infty$.
\end{enumerate}
\end{theorem}

We now turn to the issue of asymptotic behavior of solutions of \eqref{eq:wave-disturbed}. Before stating our results, we need to introduce, for every damping set $\Sigma \subset \mathbb R^2$, the function $\mu: \mathbb R_+ \to \mathbb R_+$ defined by
\begin{equation}
\label{eq:defiMu}
\mu(r) = \lim_{\eta \to 0^+} \sup_{\abs{x} \leq r + \eta} \norm{S(x)},
\end{equation}
where we recall that $S$ is the set-valued map whose graph is $R\Sigma$. Similarly to the function $\rho$ introduced in \eqref{eq:defiRho}, the function $\mu$ satisfies the properties stated in Proposition~\ref{PropDefiRho}, with $S^{[2]}$ replaced by $S$ in \ref{ItemRhoStrictDamping} and $\rho^{[\floor{n/2}]}$ replaced by $\mu^{[n]}$ in \ref{ItemIterateEstimateRho}. We also need to introduce the space of disturbances $\mathcal D_p$, $p \in [1, +\infty)$, given by
\begin{equation}
\label{eq:defi-space-Dp}
\mathcal{D}_p:=\left\{ d:\mathbb R_+\rightarrow \mathbb R^2 \text{ is measurable} \suchthat D(\cdot):=\sum_{n=0}^\infty |d(\cdot+2n+1)|\in \mathsf Y_p \right\}.
\end{equation}
Note that $\mathcal D_p \subset L^p(\mathbb R_+, \mathbb R^2)$ with equality if and only if $p = 1$.

We shall use the next definition of input-to-state stability for \eqref{eq:wave-disturbed}, which can can be seen as a generalization of the UGAS property in presence of disturbances (see, for instance, \cite{Mironchenko2018Characterizations}).

\begin{definition}\label{def:ISS}
Let $p\in[1, +\infty]$. We say that the dynamical system \eqref{eq:wave-disturbed} is \emph{input-to-state stable} (ISS) with respect to the state space $\mathsf X_p$ and the disturbance space $L^p(\mathbb R_+, \mathbb R^2)$ if there exist a $\mathcal{KL}$ function $\beta$ and a $\mathcal{K}$ function $\gamma$ such that, for every trajectory $z$ of \eqref{eq:wave-disturbed} associated with an initial condition $(z_0,z_1)\in \mathsf X_p$ and a disturbance $d\in L^p(\mathbb R_+, \mathbb R^2)$, one has
\begin{equation}\label{eq:iss}
e_p(z)(t)\leq \beta(e_p(z)(0),t)+\gamma(\norm{d}_{L^p(\mathbb R_+,\mathbb R^2)}),\qquad t\geq 0.
\end{equation}
\end{definition}

The main result of this section is the next theorem, which provides conditions on $\mu$ ensuring that \eqref{eq:wave-disturbed} is either strongly stable or ISS.

\begin{theorem}
\label{thm:ISS}
Let $p \in [1, +\infty]$, $\Sigma \subset \mathbb R^2$ be a damping set, and $\mu$ be the corresponding function defined in \eqref{eq:defiMu}. Assume that $\mu(r) < r$ for every $r > 0$.
\begin{enumerate}
\item\label{ItemStrongStab} Let $d \in L^p_{\loc}(\mathbb R_+, \mathbb R^2)$ and $z$ be a solution of \eqref{eq:wave-disturbed} with disturbance $d$. Then $e_p(z)(t) \to 0$ as $t \to +\infty$ if one of the following conditions holds.
\begin{enumerate}
\item\label{ItemItemStrongStabPFinite} $p < +\infty$ and $d \in \mathcal D_p$.
\item\label{ItemItemStrongStabDConverges} $d(t) \to 0$ as $t \to +\infty$ and, for every $\varepsilon > 0$, one has
\[
\inf_{r \geq \varepsilon} r - \mu(r) > 0.
\]
\end{enumerate}
\item\label{ItemISS}
The dynamical system \eqref{eq:wave-disturbed} is ISS with respect to $\mathsf X_p$ and $L^p(\mathbb R_+, \mathbb R^2)$ if one of the following conditions holds.
\begin{enumerate}
\item\label{ItemISS-pinf} $p=+\infty$ and $\lim_{r\to+\infty}r -\mu(r)=+\infty$.
\item\label{ItemISS-pfini} $p<+\infty$ and \ref{HypoSigma-SectorInfty} is satisfied.
\end{enumerate}
\end{enumerate}
\end{theorem}

\begin{proof}
Let $d: \mathbb R_+ \to \mathbb R^2$ be measurable, $z$ be a trajectory of \eqref{eq:wave-disturbed}, and $(g_n)_{n \in \mathbb N}$ and $(\delta_n)_{n \in \mathbb N}$ be the sequences corresponding to $z$ and $d$ according to Proposition~\ref{PropGn} and \eqref{eq:delta-from-d}, respectively. For $n \in \mathbb N$ and $s \in [-1, 1]$, denote $\delta_n(s) = (\delta_{n, 1}(s), \delta_{n, 2}(s)) \in \mathbb R^2$ and set $h_n(s) = g_n(s) - \delta_{n, 1}(s)$. Then, using \eqref{eq:gn-disturbed}, one deduces that, for $n \in \mathbb N$ and a.e.\ $s \in [-1, 1]$,
\begin{equation}
\label{eq:yacine-veut-un-label-ici}
\abs{h_{n+1}(s)} \leq \mu(\abs{h_n(s)}) + \abs{\delta_{n+1, 1}(s) - \delta_{n, 2}(s)}.
\end{equation}

We first prove \ref{ItemStrongStab} under the assumption \ref{ItemItemStrongStabPFinite}. One deduces from \eqref{eq:yacine-veut-un-label-ici} that, for $n \in \mathbb N$ and a.e.\ $s \in [-1, 1]$,
\begin{equation}
\label{eq:yacine-veut-un-label-ici-aussi}
\abs{h_n(s)} \leq \abs{h_0(s)} + 2 D(s),
\end{equation}
where $D$ is the function in $\mathsf Y_p$ associated with $d \in \mathcal D_p$ in the sense of \eqref{eq:defi-space-Dp}. Hence $(h_n)_{n \in \mathbb N}$ is a sequence in $\mathsf Y_p$ and it is dominated by the function $\abs{h_0} + 2 D$, which also belongs to $\mathsf Y_p$. In addition, $\delta_n(s)$ tends to $0$ as $n \to +\infty$ for a.e.\ $s \in [-1, 1]$ since $D \in \mathsf Y_p$. Then, in order to obtain the conclusion of \ref{ItemStrongStab}, it suffices to show that $(h_n)_{n \in \mathbb N}$ tends to zero almost everywhere.

Let $s \in [-1, 1]$ be such that the series of general term $\abs{\delta_{n+1, 1}(s) - \delta_{n, 2}(s)}$ is convergent. The sequence $(\abs{h_n(s)})_{n \in \mathbb N}$ being bounded thanks to \eqref{eq:yacine-veut-un-label-ici-aussi}, it is enough to prove that its only limit point is zero. Reasoning by contradiction, assume that there exists a subsequence $(\abs{h_{n_k}(s)})_{k \in \mathbb N}$ converging to some $r_\ast > 0$. Then, for every $\varepsilon > 0$, there exists $k_0 \in \mathbb N$ such that $\abs[\big]{\abs{h_{n_k}(s)} - r_\ast} < \varepsilon$ for every $k \geq k_0$ and $\sum_{n = n_{k_0}}^{\infty} \abs{\delta_{n+1, 1}(s) - \delta_{n, 2}(s)} < \varepsilon$. Using \eqref{eq:yacine-veut-un-label-ici} and the fact that $\mu$ is increasing, one deduces that
\[
r_\ast - \varepsilon \leq \abs{h_{n_{k + 1}}(s)} \leq \mu(\abs{h_{n_k}(s)}) + \sum_{n = n_k}^{n_{k+1} - 1} \abs{\delta_{n+1, 1}(s) - \delta_{n, 2}(s)} \leq \mu(r_\ast + \varepsilon) + \varepsilon.
\]
Hence $r_\ast \leq \mu(r_\ast + \varepsilon) + 2 \varepsilon$ and, since $\mu$ is upper semi-continuous and $\varepsilon > 0$ is arbitrary, we deduce that $\mu(r_\ast) \geq r_\ast$, yielding the desired contradiction.

We next prove \ref{ItemStrongStab} under the assumption \ref{ItemItemStrongStabDConverges}, which amounts to prove that $(h_n)_{n \in \mathbb N}$ tends to $0$ in $\mathsf Y_p$. For every $\varepsilon > 0$, let $\eta = \min\left(\varepsilon/2, \inf_{r \geq \varepsilon/2} r - \mu(r)\right) > 0$. Let $n_0 \in \mathbb N$ be such that $\abs{\delta_{n+1, 1}(s) - \delta_{n, 2}(s)} < \eta/2$ for every $n \geq n_0$ and $s \in [-1, 1]$. It follows from \eqref{eq:yacine-veut-un-label-ici} that, for $n \geq n_0$ and $s \in [-1, 1]$, one has
\begin{align*}
\abs{h_n(s)} \leq \varepsilon & \implies \abs{h_{n+1}(s)} \leq \varepsilon, \\
\abs{h_n(s)} \geq \frac{\varepsilon}{2} & \implies \abs{h_{n+1}(s)} \leq \abs{h_n(s)} - \frac{\eta}{2}.
\end{align*}
For $n \geq n_0$, let $\chi_n$ denote the characteristic function of the set $\{s \in [-1, 1] \suchthat \abs{h_n(s)} \leq \varepsilon\}$. Then the sequence $(h_n \chi_n)_{n \in \mathbb N}$ takes values in $\mathsf Y_p$ and, for $n \geq n_0$, one has $\norm{h_n \chi_n}_p \leq 2^{1/p} \varepsilon$ (with the convention that $1/p = 0$ for $p = +\infty$). On the other hand, for a.e.\ $s \in [-1, 1]$, the sequence $(\abs{h_n(s)} (1 - \chi_n(s)))_{n \in \mathbb N}$ is nonincreasing for $n \geq n_0$ and converges in a finite number of steps to zero. In particular, by the dominated convergence theorem, $\norm{\abs{h_n} (1 - \chi_n)}_p$ tends to zero as $n \to +\infty$. The conclusion follows.

We next turn to \ref{ItemISS}. Notice first that $e_p^p(z)(t) \leq e_p^p(z)(2n) + e_p^p(z)(2(n+1))$ for $p < +\infty$ and $e_\infty(z)(t) \leq \max(e_\infty(z)(2n), e_\infty(z)(2(n+1)))$ for every $n \in \mathbb N$ and $t \in [2n, 2(n+1)]$. Hence, using also Lemma~\ref{lem:KL} in Appendix~\ref{AppLemmas}, it is enough to show that the inequality in \eqref{eq:iss} holds for $t = 2n$ for every $n \in \mathbb N$ to obtain that \eqref{eq:wave-disturbed} is ISS. Moreover, notice that \ref{HypoSigma-SectorInfty} implies that $\lim_{r \to +\infty} r -  \mu(r) = +\infty$.

First of all, one obtains from Lemma~\ref{LemmKInfty} given in Appendix~\ref{AppLemmas} that there exists a $\mathcal{K}_\infty$ function $\varphi$ lower bounding the function $r\mapsto r - \mu(r)$ and such that $\id - \varphi$ is nondecreasing. In particular, one then has that $\varphi^{-1}: \mathbb R_+ \to \mathbb R_+$ is $\mathcal K_\infty$. Moreover, if \ref{HypoSigma-SectorInfty} holds true, one has in addition that $\id - \varphi$ is concave. From \eqref{eq:gn-disturbed}, one deduces that, for every $n \in \mathbb N$ and $s \in [-1, 1]$, one has
\[\abs{g_{n+1}(s)} \leq (\id - \varphi)(\abs{g_n(s)} + \abs{\delta_{n, 1}(s)}) + \abs{\delta_{n, 2}(s)},\]
and thus
\[
\abs{g_{n+1}(s)} \leq (\id - \varphi)(\abs{g_n(s)}) + \abs{\delta_{n, 1}(s)} + \abs{\delta_{n, 2}(s)},
\]
since, by an immediate computation using the fact that $\varphi$ is nondecreasing, for every $a, b \in \mathbb R_+$, one has $(\id - \varphi)(a + b) \leq (\id - \varphi)(a) + b$. Since $\id - \varphi$ is nondecreasing and it is concave when $p$ is finite, one deduces, applying Lemma~\ref{LemmaJensen} in Appendix~\ref{AppLemmas} in the case $p < +\infty$, that, for every $n \in \mathbb N$,
\[
\norm{g_{n+1}}_p \leq 2^{1/p} (\id - \varphi)(2^{-1/p} \norm{g_n}_p) + 2 \norm{\delta_{n}}_p,
\]
with the convention that $1/p = 0$ for $p = +\infty$. It follows that $\norm{g_n}_p \leq k_n$, where $(k_n)_{n \in \mathbb N}$ is the trajectory of the one-dimensional discrete-time control system defined by
\begin{equation}
\label{eq:DynSyst}
k_{n+1} = 2^{1/p}(\id - \varphi)(2^{-1/p}\abs{k_n}) + \abs{u_n}, \qquad n \in \mathbb N,
\end{equation}
with initial condition and control given respectively by
\[
\begin{aligned}
k_0 & = \norm{g_0}_p, & \qquad & \\
u_n & = 2 \norm{\delta_n}_p, & & n \in \mathbb N.
\end{aligned}
\]
To obtain the conclusion, it is enough to prove that the control system \eqref{eq:DynSyst} is ISS according to the standard definition of ISS for finite-dimensional control systems as given in \cite{Jiang2001Input}. Indeed, in this case, one will deduce that there exist a $\mathcal{KL}$ function $\beta$ and a $\mathcal K$ function $\gamma$ such that
\[
\norm{g_n}_p \leq k_n\leq \beta(k_0, n) + \gamma\biggl(\sup_{n \in \mathbb N} u_n\biggr)=
\beta(\norm{g_0}_p, n) + \gamma\biggl(2 \sup_{n \in \mathbb N} \norm{\delta_n}_p\biggr)
\]
and the conclusion follows by noticing that $\sup_{n \in \mathbb N} \norm{\delta_n}_p \leq \norm{d}_{L^p(\mathbb R_+, \mathbb R^2)}$.

Thanks to \cite[Theorem~1]{Jiang2001Input}, the proof of ISS for \eqref{eq:DynSyst} is reduced to establishing the existence of three $\mathcal K$ functions $\gamma_0, \sigma_1, \sigma_2$ such that, for every trajectory $(k_n)_{n \in \mathbb N}$ of \eqref{eq:DynSyst} starting at $k_0 \in \mathbb R$ and corresponding to a control $(u_n)_{n \in \mathbb N}$, one has
\begin{equation}
\label{eq:iss-limsup}
\limsup_{n \to +\infty} \abs{k_n} \leq \gamma_0\left(\limsup_{n \to +\infty}\abs{u_n}\right)
\end{equation}
and
\begin{equation}
\label{eq:iss-ubibs}
\sup_{n \in \mathbb N} \abs{k_n} \leq \max\left(\sigma_1(\abs{k_0}), \sigma_2\left(\sup_{n \in \mathbb N} \abs{u_n}\right)\right).
\end{equation}

From \eqref{eq:DynSyst} and the fact that $\id-\varphi$ is continuous and nondecreasing, it is immediate to derive that
\[
\limsup_{n \to +\infty} \abs{k_n} \leq 2^{1/p} (\id - \varphi)(2^{-1/p} \limsup_{n \to +\infty} \abs{k_n}) + \limsup_{n \to +\infty} \abs{u_n},
\]
yielding \eqref{eq:iss-limsup} with $\gamma_0(r) = 2^{1/p} \varphi^{-1}(2^{-1/p} r)$ for $r \in \mathbb R_+$.

Finally, \eqref{eq:iss-ubibs} follows by proving that, for every $n \in \mathbb N$,
\[\abs{k_n} \leq \max\left(\abs{k_0}, U + 2^{1/p} \varphi^{-1}(2^{-1/p} U)\right),\]
where $U = \sup_{n \in \mathbb N} \abs{u_n}$. Indeed, using an inductive argument, the above inequality holds trivially for $n = 0$, while the induction step follows by considering separately the cases where $\abs{k_n} < 2^{1/p} \varphi^{-1}(2^{-1/p} U)$ or not.
\end{proof}

\begin{remark} Consider a damping set $\Sigma$ verifying the following  
generalized sector condition, which is a global version of \ref{HypoSigma-gUpper}: there exist a 
positive constant $M$ and a function $q \in \mathcal C^1(\mathbb R_+, \mathbb R_+)$ with $q(0) = 0$, $0 < 
q(x) < x$, and $\abs{q^\prime(x)} < 1$ for every $x>0$ such that $q(\abs{x}) \leq \abs{y}$ and $q(\abs{y}) \leq \abs{x}$ for every $(x, y) \in \Sigma$. In that case, the function $\mu$ defined in \eqref{eq:defiMu} satisfies $\mu \leq Q$, where the function $Q$ is defined from $q$ as in \eqref{Relation-Q-q}. Moreover, conditions on $\mu$ expressed in \ref{ItemItemStrongStabDConverges} and \ref{ItemISS-pinf} are satisfied if $\liminf_{x \to +\infty} q(x) > 0$ and $\lim_{x\to + \infty} q(x) = +\infty$, respectively.
\end{remark}

\begin{remark}\label{rem:loc}
One can weaken the assumption in \ref{ItemISS-pinf} to $p = +\infty$ and $\liminf_{r \to +\infty} r - \mu(r) =: \ell > 0$ and obtain an ISS-type result under additional assumptions on the $L^\infty$ norm of the disturbance $d$. More precisely, in that case, one can adapt the proof of Lemma~\ref{LemmKInfty} in Appendix~\ref{AppLemmas} to provide a $\mathcal K$ function $\varphi$ lower bounding $\id - \mu$ such that $\id - \varphi$ is nondecreasing, but whose range is $[0, \ell)$ instead of $\mathbb R_+$. By following the same lines of the arguments of Item~\ref{ItemISS-pinf}, one deduces an ISS-type estimate for trajectories of \eqref{eq:wave-disturbed} associated with disturbances $d$ in $L^\infty(\mathbb R_+, \mathbb R^2)$ with $\norm{d}_{L^\infty} < \ell/2$, i.e., one has an estimate of the form \eqref{eq:iss} but the function $\gamma$ is defined only on $[0, \ell/2)$.
\end{remark}

\begin{remark}
To the best of our knowledge, few results on \eqref{eq:wave-disturbed} 
have been obtained and the most precise ones can be found in \cite{Xu2019Saturated}, which considers the case where $\Sigma$ is the graph of a saturation function and essentially deals with disturbance rejection. Theorem~\ref{thm:ISS}\ref{ItemStrongStab} improves the results of that reference since the disturbance there is assumed to be matching and regular (both $d(\cdot)$ and $d'(\cdot)$ belong to $\mathcal{D}_1$) and Remark~\ref{rem:loc} shows that we can achieve ISS-type of results in $\mathsf X_\infty$.
\end{remark}

\section{Application to the case of the sign function}
\label{SecSign}

In this section, we apply the framework introduced in the paper to the particular choice of $\Sigma$ as the graph of the sign multi-valued function defined in \eqref{eq:defi-sign}, which is not the graph of a function $\sigma: \mathbb R \to \mathbb R$ and which is not a strict damping. We will illustrate how our techniques easily handle this case, obtaining in particular existence and uniqueness of solutions  of \eqref{eq:wave} and the precise characterization of their asymptotic behavior. This boundary condition has been previously considered in the literature, for instance in \cite{Xu2019Saturated}.

We then consider that $\Sigma = \Sigma_M$, where $\Sigma_M$ is defined in \eqref{eq:Sigma-Sign}, and, for sake of simplicity, we assume in the sequel that $M = \sqrt{2}$. All the results that we will present in this section in that case readily extend to the general case of $\Sigma_M$ for any $M > 0$ by simply remarking that $z$ is a solution of \eqref{eq:wave} associated with the set $\Sigma_M$ if and only if $\sqrt{2} z / M$ is a solution of \eqref{eq:wave} associated with the set $\Sigma_{\sqrt{2}}$.

A straightforward computation shows that $R\Sigma$ is the graph of a function (which we denote by $S$ in a slight abuse of notation) given by
\begin{equation}
\label{eq:S-Sign}
S(x) = \begin{dcases*}
x, & if $\abs{x} \leq 1$, \\
2 - x, & if $x > 1$, \\
-2 - x, & if $x < -1$.
\end{dcases*}
\end{equation}
Note that this set satisfies both \ref{HypoSigma-Unique} and \ref{HypoSigma-UniqueInfty} and thus, as an immediate consequence of Theorem~\ref{TheoUnique}, we obtain the following.

\begin{proposition}
Let $p \in [1, +\infty]$ and consider the wave equation \eqref{eq:wave} with the set $\Sigma$ given by \eqref{eq:Sigma-Sign}. Then, for every $(z_0, z_1) \in \mathsf X_p$, there exists a unique solution of \eqref{eq:wave} in $\mathsf X_p$ with initial condition $(z_0, z_1)$.
\end{proposition}

\begin{remark}
In \cite{Xu2019Saturated}, the sign function has been introduced as the limit of linear saturated feedbacks, and the existence and uniqueness result therein, \cite[Lemma~3]{Xu2019Saturated}, is shown in the Hilbertian framework of $\mathsf X_2$ by proving that the generator of the corresponding equation in $\mathsf X_2$ is a maximal monotone nonlinear operator. With respect to that framework, our techniques allow one to consider solutions of \eqref{eq:wave} in $\mathsf X_p$ for any $p \in [1, +\infty]$.
\end{remark}

We now turn to the asymptotic behavior of trajectories of \eqref{eq:wave} in $\mathsf X_p$ for $p \in [1, +\infty]$ with $\Sigma$ given by \eqref{eq:Sigma-Sign} and $M = \sqrt{2}$. It is immediate to see that $\Sigma$ is a damping set and hence, by Proposition~\ref{PropDampingNonStrict}, the energy $e_p(z)(\cdot)$ is nonincreasing along trajectories $z$. On the other hand, $S$ is continuous, and hence its graph is closed, as well as the graphs of its iterates $S^{[n]}$ for $n \in \mathbb N$. Since $S(x) = x$ for $x \in [-1, 1]$, $S^{[2]}$ is not a strict damping, and thus Corollary~\ref{CoroS2Closed} immediately implies that \eqref{eq:wave} is not strongly stable in $\mathsf X_p$ for any $p \in [1, +\infty]$. A more direct way to see that, which is similar to the argument provided in \cite[Lemma~4]{Xu2019Saturated}, consists of considering a constant function $g_0$ in $\mathsf Y_p$ whose constant value belongs to $[-1, 1]$. The corresponding solution $z$ of \eqref{eq:wave} with initial condition $\mathfrak I^{-1}(g_0)$ corresponds to the sequence $(g_n)_{n \in \mathbb N}$ which is constant and equal to $g_0$.

We next characterize the limits of solutions of \eqref{eq:wave}. We start with the following preliminary result dealing with limits of real iterated sequences for $S$, which is obtained by straightforward computations.

\begin{lemma}
\label{LemmaSign}
Let $S$ be given by \eqref{eq:S-Sign}.
\begin{enumerate}
\item For every $n \in \mathbb N$, one has $S^{[n]}(0) = 0$ and, for $x \neq 0$,
\begin{equation}
\label{eq:Sn-Sign}
S^{[n]}(x) = (-1)^k \frac{x}{\abs{x}} \left(\abs{x} - 2 k\right), \qquad k = \min\left(n, \floor*{\frac{\abs{x} + 1}{2}}\right).
\end{equation}

\item Let $(x_n)_{n \in \mathbb N}$ be a real iterated sequence for $S$ starting at $x_0 \neq 0$ and denote $k_0 = \floor*{\frac{\abs{x_0} + 1}{2}}$. Then $(x_n)_{n \geq k_0}$ is constant and
\[x_n = (-1)^{k_0} \frac{x_0}{\abs{x_0}} \left(\abs{x_0} - 2 k_0\right), \qquad \text{ for } n \geq k_0.\]
\end{enumerate}
\end{lemma}

For sake of simplicity and using an abuse of notation, we will write the right-hand side of \eqref{eq:Sn-Sign} also when $x = 0$, and it should be considered as being equal to zero in that case.

As an immediate consequence of Lemma~\ref{LemmaSign}, we have the following.

\begin{proposition}
\label{PropSign}
Let $g_0: [-1, 1] \to \mathbb R$ be measurable and consider the iterated sequence $(g_n)_{n \in \mathbb N}$ for $S$ starting at $g_0$. Let $K: [-1, 1] \to \mathbb N$ and $g_\infty: [-1, 1] \to \mathbb R$ be defined by $K(s) = \floor*{\frac{\abs{g_0(s)} + 1}{2}}$ and
\[
g_\infty(s) = (-1)^{K(s)} \frac{g_0(s)}{\abs{g_0(s)}} \left(\abs{g_0(s)} - 2 K(s)\right).
\]
Then $S\circ g_\infty = g_\infty$, $g_n$ converges to $g_\infty$ a.e.\ on $[-1, 1]$, and $\abs{g_n} \leq \abs{g_0}$. In particular, if $g_0 \in \mathsf Y_p$ for some $p \in [1, +\infty)$, then $g_n$ converges to $g_\infty$ in $\mathsf Y_p$. In addition, if $g_0 \in \mathsf Y_\infty$, then $g_n = g_\infty$ for every $n \geq \norm{K}_\infty$.
\end{proposition}

Translating the above result in terms of solutions of \eqref{eq:wave}, one immediately gets the following.

\begin{theorem}
\label{TheoSign}
Let $p \in [1, +\infty]$, $(z_0, z_1) \in \mathsf X_p$, and consider the solution $z$ of \eqref{eq:wave} with $\Sigma$ given by \eqref{eq:Sigma-Sign} with $M = \sqrt{2}$ and with initial condition $(z_0, z_1)$. Let $g_0 = \mathfrak I(z_0, z_1) \in \mathsf Y_p$, $g_\infty$ be defined from $g_0$ as in Proposition~\ref{PropSign}, and $z_\infty$ be the solution of \eqref{eq:wave} starting at $\mathfrak I^{-1}(g_\infty)$. Then $z_\infty$ is $2$-periodic and $z(t) - z_\infty(t)$ converges to $0$ in $\mathsf X_p$ as $t \to +\infty$. Moreover, if $p = +\infty$, the convergence takes place in finite time less than or equal to $T = 2 \floor*{\frac{\norm{(z_0, z_1)}_{\mathsf X_\infty} + 1}{2}}$, i.e., $z(t) = z_\infty(t)$ for $t \geq T$.
\end{theorem}

\begin{proof}
We only provide an argument for the $2$-periodicity of $z_\infty$. Indeed, recall that, by Proposition~\ref{PropGn}\ref{ItemEquivSols2}, the Riemannian invariants of $z_\infty$ are built after the $S$-iterates of $g_\infty$. The latter being a fixed point of $S$, the Riemannian invariants of $z_\infty$ are then $2$-periodic. One derives the $2$-periodicity of $z_\infty$ by using the first equation of \eqref{eq:defWeakSol}.
\end{proof}

\begin{remark}
Theorem~\ref{TheoSign} improves the result \cite[Theorem~6]{Xu2019Saturated} in the following directions: firstly, Theorem~\ref{TheoSign} applies to solutions  of \eqref{eq:wave} with initial conditions in $\mathsf X_p$ for some $p \in [1, +\infty]$, whereas the main convergence result in \cite[Theorem~6]{Xu2019Saturated} only consider regular solutions of \eqref{eq:wave} in the Hilbertian setting. Secondly, thanks to Theorem~\ref{TheoSign}, we provide an expression for the limit $z_\infty$ which is more explicit than the one based on Fourier series provided in \cite[Theorem~6]{Xu2019Saturated}.
\end{remark}

\begin{remark}
At the light of the finite-time convergence in $\mathsf X_\infty$ in Theorem~\ref{TheoSign}, one may wonder whether uniform bounds on the convergence rate can also be obtained for finite $p$. Unfortunately, the answer turns out to be negative, since, by Theorem~\ref{TheoremSloooooooow}, the energy of solutions of \eqref{eq:wave} in $\mathsf X_p$ for finite $p$ may decrease arbitrarily slow.
\end{remark}

\appendix

\section{Universally measurable functions}
\label{AppUniversally}

Hypotheses \ref{HypoSigma-Exists}--\ref{HypoSigma-UniqueInfty} used throughout this paper use the notion of universally measurable functions. This appendix provides the definition of this class of functions together with their main properties used in this paper. Interested readers may find further properties of universally measurable sets and functions and their applications in analysis in \cite{Bertsekas1978Stochastic, Cohn2013Measure, Nishiura2008Absolute}. For sake of simplicity, we only consider here real-valued universally measurable functions defined on a (possibly unbounded) interval $I \subset \mathbb R$, since this particular setting is the only one used in the paper. We denote by $\mathfrak B$ and $\mathfrak L$ the $\sigma$-algebras of Borel and Lebesgue measurable subsets of $\mathbb R$, respectively.

\begin{definition}
\label{DefiUniversally}
Let $\mathfrak F$ be a $\sigma$-algebra on $\mathbb R$.
\begin{enumerate}
\item Let $\mu$ be a nonnegative measure on $(\mathbb R, \mathfrak F)$. A set $M \subset \mathbb R$ is said to be \emph{$\mu$-measurable} if there exist $A, B$ in $\mathfrak F$ with $A \subset M \subset B$ such that $\mu(B \setminus A) = 0$.

\item\label{ItemDefiUSet} A set $M \subset \mathbb R$ is said to be \emph{universally measurable with respect to $\mathfrak F$} if, for every probability measure $\mu$ on $(\mathbb R, \mathfrak F)$, $M$ is $\mu$-measurable. The family of all universally measurable subsets of $\mathbb R$ with respect to $\mathfrak F$ is denoted by $\mathfrak F_\ast$.

When $\mathfrak F$ is the $\sigma$-algebra $\mathfrak B$ of Borel subsets of $\mathbb R$, we define $\mathfrak U = \mathfrak B_\ast$ and we say that the elements of $\mathfrak U$ are the \emph{universally measurable} subsets of $\mathbb R$.

\item Let $I \subset \mathbb R$ be an interval and $f: I \to \mathbb R$. We say that $f$ is an \emph{universally measurable function} if, for every $A \in \mathfrak B$, one has $f^{-1}(A) \in \mathfrak U$.
\end{enumerate}
\end{definition}

For every $\sigma$-algebra $\mathfrak F$ of $\mathbb R$, the set $\mathfrak F_\ast$ is also a $\sigma$-algebra and one has further that $(\mathfrak F_\ast)_\ast = \mathfrak F_\ast$. In particular, $\mathfrak U_\ast = \mathfrak U$. The $\sigma$-algebra $\mathfrak U$ is called the \emph{universal $\sigma$-algebra} of $\mathbb R$. Using the fact that the Lebesgue measure is complete with respect to the $\sigma$-algebra $\mathfrak L$, one also immediately checks that $\mathfrak L_\ast = \mathfrak L$.

The above definition implies that $\mathfrak B \subset \mathfrak U \subset \mathfrak L$, and classical counterexamples presented in \cite{Bertsekas1978Stochastic, Cohn2013Measure, Nishiura2008Absolute} show that these inclusions are strict. One deduces from these inclusions that every Borel measurable function is universally measurable and that every universally measurable function is Lebesgue measurable. Note that, as stated, e.g., in \cite{Nishiura2008Absolute}, the requirement on $\mu$ to be a probability measure in Definition~\ref{DefiUniversally}\ref{ItemDefiUSet} may be replaced with the requirement on $\mu$ being finite or also $\sigma$-finite with no change in the definition.

A classical result on universal measurability is the following property, whose proof can be found, for instance, in \cite[Lemma~8.4.6]{Cohn2013Measure}.

\begin{proposition}
\label{PropFuncUniv}
Let $\mathfrak F$ and $\mathfrak G$ be $\sigma$-algebras on $\mathbb R$, $I \subset \mathbb R$ be an interval, and $f: I \to \mathbb R$. Assume that $f^{-1}(A) \in \mathfrak F$ for every $A \in \mathfrak G$. Then $f^{-1}(A) \in \mathfrak F_\ast$ for every $A \in \mathfrak G_\ast$.
\end{proposition}

As an immediate consequence of Proposition~\ref{PropFuncUniv}, one obtains the following alternative characterizations of Lebesgue and universally measurable functions.

\begin{corollary}
\label{CoroCharactLebesgueUniv}
Let $I \subset \mathbb R$ be an interval and $f: I \to \mathbb R$.
\begin{enumerate}
\item The function $f$ is Lebesgue measurable if and only if $f^{-1}(A) \in \mathfrak L$ for every $A \in \mathfrak U$.
\item\label{ItemCharactUniv} The function $f$ is universally measurable if and only if $f^{-1}(A) \in \mathfrak U$ for every $A \in \mathfrak U$.
\end{enumerate}
\end{corollary}

Corollary~\ref{CoroCharactLebesgueUniv}\ref{ItemCharactUniv} implies in particular that universal measurability of functions is preserved under composition, as we state next.

\begin{proposition}
\label{PropCompositionUnivMeas}
Let $f: \mathbb R \to \mathbb R$ and $g: \mathbb R \to \mathbb R$ be universally measurable functions. Then $f \circ g$ is also universally measurable.
\end{proposition}

The major result on universally measurable functions that we need in this paper is the following, which characterizes the set of universally measurable functions as those which preserve Lebesgue measurability by left composition. The statement and the proof presented below were communicated to the authors\footnote{See \url{https://mathoverflow.net/questions/366953/}.} by Mateusz Kwa\'{s}nicki.

\begin{proposition}
\label{PropUnivMeasurable}
Let $f: \mathbb R \to \mathbb R$. Then $f$ is universally measurable if and only if, for every Lebesgue measurable function $g: (-1, 1) \to \mathbb R$, $f \circ g$ is Lebesgue measurable.
\end{proposition}

\begin{proof}
Assuming first that $f$ is universally measurable, one immediately obtains from the characterizations in Corollary~\ref{CoroCharactLebesgueUniv} that $f \circ g$ is Lebesgue measurable for every Lebesgue measurable function $g: (-1, 1) \to \mathbb R$.

Let us now assume that $f$ is not universally measurable. We will construct a continuous function $g: (-1, 1) \to \mathbb R$ such that $f \circ g$ is not Lebesgue measurable.

Since $f$ is not universally measurable, there exists $B \in \mathfrak B$ such that $A = f^{-1}(B) \notin \mathfrak U$. Thus, there exists a probability measure $\mu$ on $(\mathbb R, \mathfrak B)$ such that $A$ is not $\mu$-measurable. Let $\lambda$ be the standard Gaussian probability measure on $\mathbb R$, i.e., $\diff \lambda(x) = \frac{1}{\sqrt{2 \pi}} e^{-x^2 / 2} \diff x$, and consider the probability measure $\nu_0 = \frac{1}{2} \mu + \frac{1}{2} \lambda$ on $(\mathbb R, \mathfrak B)$. Clearly, $A$ is not $\nu_0$-measurable either. Let $\nu$ be the probability measure obtained from $\nu_0$ by removing its atoms and renormalizing the resulting measure, and notice that $A$ is not $\nu$-measurable.

Let $h: \mathbb R \to \mathbb R$ be the cumulative distribution function of $\nu$, defined for $x \in \mathbb R$ by $h(x) = \nu((-\infty, x])$. Then, by construction of $\nu$, one deduces that $h$ is continuous, increasing, $h(x) \in (0, 1)$ for every $x \in \mathbb R$, $\lim_{x \to -\infty} h(x) = 0$, and $\lim_{x \to +\infty} h(x) = 1$. In particular, $h$ admits an inverse $h^{-1}: (0, 1) \to \mathbb R$ which is continuous and increasing. Recall also that $\nu(E) = m(h(E))$ for every $E \in \mathfrak B$, where $m$ denotes the Lebesgue measure on $(0, 1)$.

We claim that $h(A)$ is not Lebesgue measurable. Indeed, if it were not the case, there would exist two Borel subsets $F_1, F_2$ of $(0, 1)$ such that $F_1 \subset h(A) \subset F_2$ and $m(F_2 \setminus F_1) = 0$. Then, letting $E_i = h^{-1}(F_i)$ for $i \in \{1, 2\}$, we would have that $E_1$ and $E_2$ are Borel sets (since $h$ is continuous) with $E_1 \subset A \subset E_2$ and $\nu(E_2 \setminus E_1) = m(h(E_2 \setminus E_1)) = m(F_2 \setminus F_1) = 0$, implying that $A$ is $\nu$-measurable, a contradiction.

Let $T: (-1, 1) \to (0, 1)$ be the linear map defined by $T(x) = \frac{x+1}{2}$. Then $T^{-1}(h(A))$ is not Lebesgue measurable. Let $g: (-1, 1) \to \mathbb R$ be the continuous function defined by $g = h^{-1} \circ T$. Then $(f \circ g)^{-1}(B) = g^{-1}(A) = T^{-1}(h(A))$ and, since $B \in \mathfrak B$ and $T^{-1}(h(A)) \notin \mathfrak L$, one deduces that $f \circ g$ is not Lebesgue measurable, as required.
\end{proof}

\section{An optimal decay rate}\label{app:decay}

This appendix proves the following result, which identifies the optimal decay rate of $Q^{[n]}(x_0)$ when $q(x) = \frac{x}{(-\ln x)^p}$ for $x > 0$ small enough and $Q$ is defined from $q$ as in \eqref{Relation-Q-q} (cf.\ Remark~\ref{remk:optimal-rate}).

\begin{theorem}
Let $p > 0$, $M \in (0, 1)$, $q \in \mathcal C^1(\mathbb R_+, \mathbb R_+)$ be given by $q(x) = \frac{x}{(-\ln x)^p}$ for $x \in (0, \frac{M}{\sqrt{2}})$, and assume further that $0 < q(x) < x$ and $\abs{q^\prime(x)} < 1$ for every $x > 0$. Let $x_0 \in \mathbb R_+^\ast$, $Q$ be defined from $q$ by \eqref{Relation-Q-q}, and the sequence $(x_n)_{n\in \mathbb N}$ be given by $x_n=Q^{[n]}(x_0)$ for $n\geq 0$. 

Set $N=\floor*{\frac1{2p}}$. Then there exist $N+1$ real numbers $\alpha_k$, $k \in \{0, \dotsc, M\}$, with $\alpha_0=\left(2(p+1)\right)^{\frac1{p+1}}$, such that, as $n \to +\infty$, one has
\begin{equation}\label{eq:est-final-xn}
x_n \sim \frac1{\sqrt{2}}e^{-\sum_{k=0}^N\alpha_kn^{\frac{1-2pk}{p+1}}}.
\end{equation}
\end{theorem}

\begin{proof} Notice first that $q^\prime(0) = 0$. Thanks to the assumptions on $q$, $(x_n)_{n \in \mathbb N}$ is a decreasing sequence of positive real numbers with $x_n \to 0$ as $n \to +\infty$ and we assume that $x_n\in (0,\frac{M}{\sqrt{2}})$ for every $n\geq 0$ with no loss of generality (cf.\ Proposition~\ref{prop:fnplusn0}\ref{ItemTimeTranslation}). 
Let $F$ be the diffeomorphism defined in Proposition~\ref{prop:decay-rate}\ref{ItemQPrime0}. One computes that, for $z \in (0, x_0]$, it holds
\begin{equation}\label{eq:Fz}
F(z) = \frac{(-\ln(\sqrt{2} z))^{p+1}}{2(p+1)} + C,
\end{equation}
where $C$ is a positive constant. In particular, $q$ does not satisfy \eqref{EqConditionEquiv}. 

We start the argument for the theorem by setting some notations for the subsequent computations
\[
y_n=\sqrt{2}x_n,\ \ z_n=2(q+\id)^{-1}(y_n),\ \ \xi_n=\frac1{-\ln(y_n)},
\quad \mu_n=\frac1{-\ln(z_n/2)}\ \hbox{ for }n\in\mathbb N.
\]
It follows at once that all the sequences defined above are positive, decreasing, and tend to zero as $n$ tends to infinity. By manipulating their definitions and using also \eqref{Relation-Q-q}, the explicit expression of $q$, and the fact that $x_{n+1} = Q(x_n)$ for $n \in \mathbb N$, we deduce that, for $n\in\mathbb N$ and $a\in \mathbb R$,
\begin{equation}\label{eq:xin}
\frac1{\xi_{n+1}}=\frac1{\xi_n}\left[1+\xi_n\ln\Bigl(\frac{1+\mu_n^p}{1-\mu_n^p}\Bigr)\right], 
\end{equation}
and 
\begin{equation}\label{eq:xin-mun}
\xi_n^a=\mu_n^a\left(1-\mu_n^{p+1}\frac{\ln(1+\mu_n^p)}{\mu_n^p}\right)^{-a}.
\end{equation}
For the rest of the argument, we also use the standard symbols $\sim$, $O(\cdot)$ and $o(\cdot)$ as $n$ tends to infinity without writing the latter fact.

By \eqref{eq:xin-mun}, one has that $\xi_n\sim \mu_n$ and it follows from \eqref{eq:Fz} and Proposition~\ref{prop:decay-rate}\ref{ItemQPrime0} that
\begin{equation}
\label{eq:expl-lim}
\Big(\frac1{\xi_n}\Big)^{p+1}\sim \alpha_0^{p+1} n.
\end{equation}
Moreover, since $0 < \mu_n < 1$ for every $n \in \mathbb N$, one has
\begin{equation}
\label{eq:log}
\ln\Big(\frac{1+\mu_n^p}{1-\mu_n^p}\Big)=2\mu_n^p\sum_{k\geq 0}\frac{\mu_n^{2pk}}{2k+1}.
\end{equation}

Using \eqref{eq:log} in \eqref{eq:xin}, one gets for $n\in\mathbb N$ that
\begin{equation}\label{eq:xin1}
\frac1{\xi_{n+1}} =\frac1{\xi_{n}}
\left[1+2\xi_n^{p+1}\left(\frac{\mu_n}{\xi_n}\right)^p\sum_{k\geq 0}\frac{\mu_n^{2pk}}{2k+1}\right].
\end{equation}
On the other hand, one deduces from \eqref{eq:xin-mun}, \eqref{eq:expl-lim}, and the fact that $\xi_n \sim \mu_n$, that $\mu_n \sim \frac{1}{\alpha_0 n^{1/(p+1)}}$,
\begin{equation}\label{eq:xin-mun2}
\mu_n^a=\xi_n^a\Big(1+O\left(n^{-1}\right)\Big),
\end{equation}
and 
\begin{equation}\label{eq:mun2}
\sum_{k\geq 0}\frac{\mu_n^{2pk}}{2k+1}=1+O\left(n^{-\frac{2p}{p+1}}\right).
\end{equation}
The above equation yields, together with \eqref{eq:xin-mun2}, that \eqref{eq:xin1} can be written, after taking its $(p+1)$-th power, as
\begin{align}
\frac1{\xi_{n+1}^{p+1}}& =\frac1{\xi_{n}^{p+1}}
\Big[1+2\xi_n^{p+1}\Big(1+O\left(n^{-1}\right)\Big)
\Big(1+O\left(n^{-\frac{2p}{p+1}}\right)\Big)
\Big]^{p+1}, \notag\\
& =\frac1{\xi_{n}^{p+1}}\Big[1+\alpha_0^{p+1}\xi_n^{p+1}\Big(1+O\left(n^{-1}\right)+O\left(n^{-\frac{2p}{p+1}}\right)\Big)\Big],\notag\\
& =\frac1{\xi_{n}^{p+1}}+\alpha_0^{p+1}+O\left(n^{-1}\right)+O\left(n^{-\frac{2p}{p+1}}\right).\label{eq:xin-pp1}
\end{align}
By summing up the above equations between $1$ and $n$, one deduces that
\begin{equation}\label{eq:first}
\frac1{\xi_{n}^{p+1}}=\alpha_0^{p+1}n+O\bigl(\ln(n)\bigr)+O\left(n^{1-\frac{2p}{p+1}}\right)=
\alpha_0^{p+1}n\left(1+O\left(\frac{\ln(n)}n\right)+O\left(n^{-\frac{2p}{p+1}}\right)\right),
\end{equation}
which implies that
\[
\frac1{\xi_{n}}=\alpha_0n^{\frac1{p+1}}\left(1+O\left(\frac{\ln(n)}n\right)
+O\left(n^{-\frac{2p}{p+1}}\right)\right)=\alpha_0n^{\frac1{p+1}}
+O\left(n^{\frac{1-2p}{p+1}}\right)+O\left(\frac{\ln(n)}{n^{1-\frac1{p+1}}}\right).
\]
By taking the exponential of the above relation, the theorem is proved in the case $p>\frac12$ .

We next suppose that $p\in (0,\frac12]$ and in that case  $N=\floor*{\frac1{2p}}\geq 1$. By using \eqref{eq:xin-mun2}, one rewrites \eqref{eq:mun2} as
\begin{equation}\label{eq:mun3}
\sum_{k\geq 0}\frac{\mu_n^{2pk}}{2k+1}=1+\sum_{k= 1}^N\frac{\xi_n^{2pk}}{2k+1}+O\left(n^{-\left(1+\frac{2p}{p+1}\right)}\right)
+O\left(n^{-\frac{2p(N+1)}{p+1}}\right).
\end{equation}
One then rewrites \eqref{eq:xin-pp1} as 
\begin{equation}\label{eq:xin3}
\frac1{\xi_{n+1}^{p+1}}=\frac1{\xi_{n}^{p+1}}+\alpha_0^{p+1}+\alpha_0^{p+1}\sum_{k= 1}^N\frac{\xi_n^{2pk}}{2k+1}+O\left(n^{-1}\right)+O\left(n^{-\frac{2p(N+1)}{p+1}}\right).
\end{equation}
We next prove that there exists $N$ real numbers $\gamma_k$ for $k \in \{1, \dotsc, N\}$ such that
\begin{equation}\label{eq:xin4}
\frac1{\xi_{n}^{p+1}}=\alpha_0^{p+1}n+\sum_{k= 1}^N \gamma_kn^{1-\frac{2pk}{p+1}}+O\left(\ln(n)+n^{1-\frac{2p(N+1)}{p+1}}\right).
\end{equation}
To see that, we set $\gamma_0=\alpha_0^{p+1}$ and we will prove by induction on $j \in \{0, \dotsc, N\}$ the following property: there exist $N+1$ real numbers $\gamma_k$ for $k \in \{0, \dotsc, N\}$ so that, for every $j\in\{0, \dotsc, N\}$, setting
\[
f_j(n)=\sum_{\ell = 0}^j \gamma_kn^{1-\frac{2p\ell}{p+1}} \qquad \text{ and } \qquad F_j(n)=\frac1{\xi_{n}^{p+1}}-f_j(n),
\]
one has
\begin{equation}
\label{eq:Fjn-induction}
F_j(n) = O\Big(\ln(n)+n^{1-\frac{2p(j+1)}{p+1}}\Big).
\end{equation}
Note that $f_j(n)\sim f_0(n)=\alpha_0^{p+1} n$ and the property is clearly true for $j=0$ by \eqref{eq:first}. For the inductive step, assume the property holds for some $j \in \{0, \dotsc, N-1\}$ and let us establish it for $j+1$. It amounts to prove that there exists a real number $\gamma_{j+1}$ such that
\begin{equation}\label{eq:inductionF}
F_j(n)=\gamma_{j+1}n^{1-\frac{2p(j+1)}{p+1}}+O\Big(\ln(n)+n^{1-\frac{2p(j+2)}{p+1}}\Big).
\end{equation}
Using the induction assumption \eqref{eq:Fjn-induction} and the definition of $F_j(n)$, we have
\[
\frac{1}{\xi_n^{p+1}} = f_j(n) \left(1 + O\left(\frac{\ln(n)}{n}+n^{-\frac{2p(j+1)}{p+1}}\right)\right)
\]
and, since $j \leq N-1$, we have $\frac{2p(j+1)}{p+1} \leq \frac{1}{p+1} < 1$, showing that, for every $k \in \{1, \dotsc, N\}$,
\begin{equation}
\label{eq:obvious-1}
\xi_n^{2pk}=f_j(n)^{-\frac{2pk}{p+1}}\Big(1+O\left(n^{-\frac{2p(j+1)}{p+1}}\right)\Big).
\end{equation}
Moreover, we also have the estimate
\begin{equation}
\label{eq:obvious-2}
(n+1)^{1-\frac{2pk}{p+1}}-n^{1-\frac{2pk}{p+1}}=
\left(1-\frac{2pk}{p+1}\right)n^{-\frac{2pk}{p+1}}\Big(1+O(n^{-1})\Big).
\end{equation}
We inject the expression $\frac1{\xi_{n}^{p+1}}=f_j(n)+F_j(n)$ into \eqref{eq:xin3} and, after computations using \eqref{eq:obvious-1} and \eqref{eq:obvious-2}, one obtains that 
\begin{align}
F_j(n+1)-F_j(n) = {} & -\sum_{k=1}^j\gamma_k\left(1-\frac{2pk}{p+1}\right)n^{-\frac{2pk}{p+1}} \notag \\
& {} + \alpha_0^{p+1}\sum_{k=1}^N\frac{f_j(n)^{-\frac{2pk}{p+1}}}{2k+1}+O\Big(n^{-\frac{2p(j+2)}{p+1}}+n^{-1}\Big). \label{eq:inductionF1}
\end{align}
Denoting the second sum by $T_j(n)$, one has that
\begin{align*}
T_j(n) &= \sum_{k=1}^N\frac{\alpha_0^{-2pk} n^{-\frac{2pk}{p+1}}}{2k+1}
\Big(1+\sum_{\ell=1}^j\frac{\gamma_\ell}{\alpha_0^{p+1}}n^{-\frac{2p\ell}{p+1}}\Big)^{-\frac{2pk}{p+1}}.
\end{align*}
Let $\varphi_{j, k}(Z) = \left(1 + \sum_{\ell = 1}^j \frac{\gamma_\ell}{\alpha_0^{p+1}} Z^\ell\right)^{-\frac{2pk}{p+1}}$ and write its Taylor expansion around $Z = 0$ as $\varphi_{j, k}(Z) = 1 + \sum_{\ell \geq 1} r_{j, k, \ell} Z^\ell$. Then, letting $Z = n^{-\frac{2 p}{p+1}}$ in the previous expression, one gets
\begin{align}
T_j(n) &= \sum_{k=1}^N\frac{\alpha_0^{-2pk} n^{-\frac{2pk}{p+1}}}{2k+1}\Big(1+\sum_{\ell\geq 1}r_{j,k, \ell}n^{-\frac{2p\ell}{p+1}}\Big)\notag\\
 &= \sum_{k=1}^N\frac{\alpha_0^{-2pk} n^{-\frac{2pk}{p+1}}}{2k+1}+\sum_{1\leq k\leq N,\ell\geq 1}\frac{\alpha_0^{-2pk}}{2k+1}r_{j,k, \ell}n^{-\frac{2p(\ell+k)}{p+1}}\notag\\
 &=\sum_{1\leq \ell \leq j}\tilde{r}_{j,\ell}n^{-\frac{2p\ell}{p+1}}+
 \tilde{r}_{j,j+1}n^{-\frac{2p(j+1)}{p+1}}+O\Big(n^{-\frac{2p(j+2)}{p+1}}\Big),\label{eq:inductionF2}
\end{align}
for suitable coefficients $\tilde r_{j, \ell}$, $\ell \in \{1, \dotsc, j+1\}$. The key point is to notice that, for $\ell \in \{1, \dotsc, j\}$, the coefficients $r_{j,k,\ell}$ only depend on $\gamma_0, \dotsc, \gamma_\ell$ and not on $\gamma_s$ for $s > \ell$ nor on $j$. As a consequence, the coefficients $\tilde{r}_{j,\ell}$ only depend on $\gamma_0, \dotsc, \gamma_{\ell-1}$ and not on $\gamma_j$ nor $j$. Hence, $\gamma_\ell$, for $\ell \in \{0, \dotsc, N\}$, is chosen according to the relation 
\[
\gamma_\ell=\frac{\alpha_0^{p+1} \tilde{r}_{j,\ell}}{1-\frac{2p\ell}{p+1}},
\]
which is possible since the right-hand side only involves $\gamma_0, \dotsc, \gamma_{\ell-1}$ and does not depend on $j>l$. In other words, $\gamma_\ell$ is determined exactly at step $\ell$ of the induction.

Gathering \eqref{eq:inductionF1}
and \eqref{eq:inductionF2}, one obtains
\[
F_j(n+1)-F_j(n)=\alpha_0^{p+1} \tilde{r}_{j,j+1}n^{-\frac{2p(j+1)}{p+1}}+O\Big(n^{-1}+n^{-\frac{2p(j+2)}{p+1}}\Big).
\]
Setting $\gamma_{j+1}=\frac{\alpha_0^{p+1} \tilde{r}_{j,j+1}}{1-\frac{2p(j+1)}{p+1}}$, one gets \eqref{eq:inductionF} after summation of the previous equation between one and $n$ large. The induction step has been established, which concludes the proof of \eqref{eq:xin4}.

One deduces from \eqref{eq:xin4} that
\begin{equation}\label{eq:xin5}
\frac1{\xi_{n}}=\alpha_0n^{\frac1{p+1}}\left[1+
\sum_{k=1}^N\frac{\gamma_k}{\alpha_0^{p+1}}n^{-\frac{2pk}{p+1}}+O\left(\frac{\ln(n)}{n}+n^{-\frac{2p(N+1)}{p+1}}\right)\right]^{\frac1{p+1}}.
\end{equation}
Since the term $\sum_{k=1}^N\frac{\gamma_k}{\alpha_0^{p+1}}n^{-\frac{2pk}{p+1}}$ can be seen as a polynomial in the indeterminate $Z=n^{-\frac{2p}{p+1}}$ with zero constant term, it is clear that there exist $N$ real numbers $\alpha_k$ with $k\in \lbrace 1,\dotsc, N\rbrace$ such that 
\begin{align*}
\left[1+
\sum_{k=1}^N\frac{\gamma_k}{\alpha_0^{p+1}}n^{-\frac{2pk}{p+1}}+O\left(\frac{\ln(n)}{n}+n^{-\frac{2p(N+1)}{p+1}}\right)\right]^{\frac1{p+1}}\notag\\
=1+
\sum_{k=1}^N\frac{\alpha_k}{\alpha_0}n^{-\frac{2pk}{p+1}}+O\left(\frac{\ln(n)}{n}+n^{-\frac{2p(N+1)}{p+1}}\right).
\end{align*}
Plugging the above equation in \eqref{eq:xin5}, one gets that 
\[
\frac1{\xi_{n}}=\sum_{k=0}^N\alpha_kn^{\frac{1-2pk}{p+1}}+O\left(\frac{\ln(n)}{n^{1-\frac1{p+1}}}+n^{\frac{1-2p(N+1)}{p+1}}\right),
\]
for $n$ large enough. Taking the exponential yields \eqref{eq:est-final-xn}.
\end{proof}

\section{Proof of Proposition~\ref{prop:faster-than-exponential-but-not-so-fast}}
\label{AppProofProp}

\begin{proof}
If \eqref{eq:q'=0-x} holds true for some $\varphi$ as in the statement, then it still holds true for any function satisfying the same assumptions and which is larger than $\varphi$ on any interval $[x_0,+\infty)$, $x_0\geq 0$. By using Lemma~\ref{lem:LB}, it is therefore enough to prove the proposition for $\varphi$ which, in addition to the above mentioned hypotheses, is also $\mathcal{C}^2$, with $\varphi^{\prime}>0$, $\varphi^{\prime\prime}\leq 0$, $\varphi(0)>0$, and such that $0 \leq \frac{x \varphi^\prime(x)}{\varphi(x)} \leq 1$ and $0 \leq -\frac{x \varphi^{\prime\prime}(x)}{\varphi(x)} \leq 2$ for every $x \geq 0$, which we assume in the sequel.

We first note that, for every $C > 0$, one has
\begin{equation}\label{eq:varphiC}
\lim_{x\to+\infty}\frac{\varphi(x+C)}{\varphi(x)}=1,
\end{equation}
since, for $x \geq 0$, one has
\[
\abs*{\ln\left(\frac{\varphi(x+C)}{\varphi(x)}\right)}=\abs*{\int_x^{x+C}\frac{s\varphi^{\prime}(s)}{\varphi(s)}\frac{\diff s}s}\leq \int_x^{x+C}\frac{\diff s}s= \ln\left(1+\frac{C}x\right),
\]
which tends to zero as $x$ tends to infinity, yielding \eqref{eq:varphiC}.

Notice that it suffices to construct the functions $q$ and $Q$ in a neighborhood of zero (in $\mathbb R_+$) and to prove \eqref{eq:q'=0-x} for $x_0 > 0$ in a neighborhood of zero (in $\mathbb R_+$). Indeed, if that is done, one can immediately extend $q$ and $Q$ to $\mathbb R_+$ in such a way that the assumptions from the statement are satisfied and, in this case, for any $x_0 > 0$, the sequence $(Q^{[n]}(x_0))_{n \in \mathbb N}$ is decreasing and converging to zero, showing that $Q^{[n]}(x_0)$ is in a neighborhood of zero for every $n$ large enough. Hence, in the sequel, we only construct $q$ and $Q$ in a neighborhood of zero and we only show \eqref{eq:q'=0-x} for $x_0 \in (0, 1)$ belonging to that neighborhood.

If $q$ and $Q$ are defined in a neighborhood of zero as in the statement, $x_0 \in (0, 1)$ belongs to that neighborhood, and we let $y_n=-\frac1{\ln(Q^{[n]}(x_0))}$ for $n\geq 0$, then one verifies from straightforward computations that the sequence $(y_n)_{n \in \mathbb N}$ satisfies the recurrence relation
\begin{equation}
\label{eq:rec-y}
y_{n+1}=y_n-U(y_n),\quad n\in\mathbb N,
\end{equation}
where $U$ is defined in a neighborhood of zero by $U(0) = 0$ and
\begin{equation}
\label{eq:U}
U(y)=-\frac{y^2\ln\bigl(\psi(e^{-1/y})\bigr)}{1-y\ln\bigl(\psi(e^{-1/y})\bigr)}, \quad y>0,
\end{equation}
and $\psi$ is defined in a neighborhood of zero by $\psi(0) = 0$ and $\psi(x) = \frac{Q(x)}{x}$ for $x > 0$. Conversely, given a function $U$, defining $\psi$ in a neighborhood of zero in such a way that \eqref{eq:U} holds, setting $Q(x) = x \psi(x)$ and defining $q$ from $Q$ using \eqref{Relation-Q-q}, any sequence $(y_n)_{n \in \mathbb N}$ starting in a neighborhood of zero and satisfying \eqref{eq:rec-y} is of the form $y_n=-\frac1{\ln(Q^{[n]}(x_0))}$ for some suitable $x_0 \in (0, 1)$ in a neighborhood of zero. Moreover, in terms of the sequence $(y_n)_{n \in \mathbb N}$, \eqref{eq:q'=0-x} reads
\begin{equation}
\label{eq:q'=0-y}
\liminf_{n \to +\infty} n \varphi(n) - \frac{1}{y_n} > -\infty.
\end{equation}
Hence, constructing $q$ and $Q$ as in the statement is equivalent to constructing a function $U$ such that the functions $q$ and $Q$ defined from it as above satisfy the properties of the statement and such that any sequence $(y_n)_{n \in \mathbb N}$ satisfying \eqref{eq:rec-y} and starting in a neighborhood of zero verifies \eqref{eq:q'=0-y}.

Define $\Psi(x)=\frac2{x\varphi(x)}$ for $x > 0$. Then $\Psi$ realizes a $\mathcal{C}^2$ diffeomorphism from $\mathbb R_+^\ast$ to $\mathbb R_+^\ast$, mapping a neighborhood of $+\infty$ to a neighborhood of $0$. Moreover, for $x\geq 0$, one has
\begin{equation}\label{eq:psi-psi'}
\Psi^{\prime}(x)=\frac{-2}{x^2\varphi(x)}\Big(1+\frac{x\varphi^{\prime}(x)}{\varphi(x)}\Big)=-\frac{\Psi(x)}{x}\Big(1+\frac{x\varphi^{\prime}(x)}{\varphi(x)}\Big)
\end{equation}
and 
\[
\Psi^{\prime\prime}(x) = \frac{2 \Psi(x)}{x^2}\left(1 + \frac{x \varphi^\prime(x)}{\varphi(x)} + \frac{x^2 \varphi^\prime(x)^2}{\varphi(x)^2} - \frac{x^2 \varphi^{\prime\prime}(x)}{2 \varphi(x)}\right).
\]
In particular, one has $\Psi^{\prime}<0$ and $\Psi^{\prime\prime}>0$ on $\mathbb R_+^\ast$. Straightforward computations also show that $0 < -\frac{x \Psi^{\prime\prime}(x)}{\Psi^\prime(x)} \leq 4$ for every $x > 0$ and $y \Psi^{-1}(y) = \frac{2 \Psi^{-1}(y)}{\Psi^{-1}(y) \varphi(\Psi^{-1}(y))} \to 0$ as $y \to 0^+$. Using the above bound on $\frac{x \Psi^{\prime\prime}(x)}{\Psi^\prime(x)}$ and reasoning as in the argument to obtain \eqref{eq:varphiC}, one deduces that, for every $C > 0$,
\begin{equation}
\label{eq:PsiC}
\lim_{x \to +\infty} \frac{\Psi^\prime(x + C)}{\Psi^\prime(x)} = 1.
\end{equation}

We claim that the function $U$ defined by $U(0) = 0$ and $U=-\frac12 \Psi^{\prime}\circ\Psi^{-1}$ in $\mathbb R_+^\ast$ meets all the requirements. Indeed, $U$ is of class $\mathcal C^1$ in $\mathbb R_+^\ast$ and, since $U(y) = \frac{y}{2 Z}\Big(1+\frac{Z\varphi^{\prime}(Z)}{\varphi(Z)}\Big)$ with $Z = \Psi^{-1}(y)$, one deduces that $U$ is continuous at $0$, $0 < U(y) < y$ for every $y > 0$, $U^\prime(0) = \lim_{y \to 0^+} \frac{U(y)}{y} = 0$, and, using that $y \Psi^{-1}(y) \to 0$ as $y \to 0^+$, we also deduce that $\lim_{y \to 0^+} \frac{U(y)}{y^2} = +\infty$. Moreover, one has 
\[
U^\prime(y) = -\frac{1}{2} \frac{\Psi^{\prime\prime} \circ \Psi^{-1}(y)}{\Psi^\prime \circ \Psi^{-1}(y)} = \frac{1}{Z} \frac{1 + A + A^2 + B}{1 + A},
\]
where $Z = \Psi^{-1}(y)$, $A = \frac{Z \varphi^\prime(Z)}{\varphi(Z)} \in [0, 1]$, and $B = - \frac{Z^2 \varphi^{\prime\prime}(Z)}{2 \varphi(Z)} \in [0, 1]$,
yielding that $U^\prime(y) \to 0$ as $y \to 0^+$.

Let $\psi$ be defined by $\psi(0) = 0$ and
\[
\psi(x) = \exp\left(-\frac{U\left(-\frac{1}{\ln x}\right) (\ln x)^2}{1 + U\left(-\frac{1}{\ln x}\right)\ln x }\right)
\]
for $x$ in a neighborhood $(0, x_\ast)$ of $0$ with $x_\ast \in (0, 1)$, in such a way that \eqref{eq:U} holds for $y$ in a neighborhood of $0$. Using the above properties on $U$, one deduces that $\psi$ is of class $\mathcal C^1$ in $(0, x_\ast)$, continuous at $0$, and $\psi(x) \in (0, 1)$ for $x \in (0, x_\ast)$.

We claim that, up to reducing $x_\ast$, one has $x \psi^\prime(x) > 0$ for $x \in (0, x_\ast)$ and $x \psi^\prime(x) \to 0$ as $x \to 0^+$. Indeed, notice that $\psi$ can be written for $x \in (0, x_\ast)$ as 
\[
\psi(x) = \exp[V(U \circ L(x), L(x))], \quad L(x):= -\frac{1}{\ln x},\quad 
V(u, y):= -\frac{u/y^2}{1 - u/y} = \frac{1}{y} - \frac{1}{y - u}.
\]
Hence, for $x \in (0, x_\ast)$, we have $x \psi^\prime(x) = x \psi(x) W(x) L^\prime(x)$, where 
\[
W(x) = \partial_u V(U \circ L(x), L(x)) U^\prime \circ L(x) + \partial_y V(U \circ L(x), L(x)).
\]
A straightforward computation yields that 
\[x \psi^\prime(x) = \psi(x) \frac{1 - U^\prime \circ L(x)}{1 - \frac{U\circ L(x)}{L(x)}},
\] 
and the above properties of $U$ show that $x \psi^\prime(x) > 0$ for $x$ small enough and $x \psi^\prime(x) \to 0$ as $x \to 0^+$.

We finally define $Q$ in the neighborhood $[0, x_\ast)$ by $Q(x) = x \psi(x)$. The above properties of $\psi$ immediately yield that $Q(0) = 0$ and $0 < Q(x) < x$ for $x \in (0, x_\ast)$. Moreover, $Q$ is clearly continuous in $[0, x_\ast)$ and of class $\mathcal C^1$ in $(0, x_\ast)$. One has $Q^\prime(0) = \lim_{x \to 0^+} \frac{Q(x)}{x} = 0$ and, using that $Q^\prime(x) = \psi(x) + x \psi^\prime(x)$ for $x > 0$, one also deduces from the above properties of $\psi$ that $Q$ is of class $\mathcal C^1$ in $[0, x_\ast)$ and that $Q^\prime(x) > 0$ for $x \in (0, x_\ast)$. Finally, defining $q$ from $Q$ using \eqref{Relation-Q-q}, one immediately verifies that $q$ satisfies the assumptions from the statement.

We are now left to prove \eqref{eq:q'=0-y} for every sequence $(y_n)_{n \in \mathbb N}$ satisfying \eqref{eq:rec-y} and with $y_0 > 0$. Fix such a sequence $(y_n)_{n \in \mathbb N}$ and notice that, since $U$ is continuous and $0 < U(y) < y$ for every $y \in \mathbb R_+^\ast$, $(y_n)_{n \in \mathbb N}$ is a decreasing sequence of positive numbers converging to $0$. We claim that there exists $n_0 \in \mathbb N$ such that, for every $n \in \mathbb N$, one has
\begin{equation}
\label{eq:bound-yn}
y_n \geq \Psi(n + n_0).
\end{equation}
Indeed, since $U^\prime(0) = 0$, the function $\id - U$ is increasing in $(0, y_\ast)$ for some $y_\ast > 0$, and we take $n_0 \in \mathbb N$ such that $y_n \in (0, y_\ast)$ for every $n \geq n_0$. Using that $\Psi(x) \to 0$ as $x \to +\infty$ and \eqref{eq:PsiC}, increasing $n_0$ if necessary, we also have that $\Psi(n_0) \leq y_0$ and that
\begin{equation}
\label{eq:RatioPsiPrime}
\frac{\Psi^\prime(x + 1)}{\Psi^\prime(x)} \geq \frac{1}{2} \qquad \text{ for every } x \geq n_0.
\end{equation}

We prove \eqref{eq:bound-yn} by induction on $n$. By construction of $n_0$, \eqref{eq:bound-yn} is satisfied for $n = 0$. Assume now that $n \in \mathbb N$ is such that \eqref{eq:bound-yn} holds. Using the fact that $\id - U$ is increasing in $(0, y_\ast)$, the induction assumption, and the definition of $U$, we deduce that 
\[
y_{n+1} = y_n - U(y_n) \geq \Psi(n + n_0) - U(\Psi(n + n_0)) = \Psi(n + n_0) + \frac{1}{2} \Psi^\prime(n + n_0).
\]
Applying the mean value theorem and using \eqref{eq:RatioPsiPrime} and the fact that $\Psi^\prime$ is negative and increasing, we get that $\Psi(n + n_0) - \Psi(n + n_0 + 1) \geq -\Psi^\prime(n + n_0 + 1) \geq -\frac{1}{2}\Psi^\prime(n + n_0)$, yielding that $y_{n+1} \geq \Psi(n + n_0 + 1)$, as required. Hence \eqref{eq:bound-yn} is established for every $n \in \mathbb N$ by induction.

By \eqref{eq:varphiC}, we have that $\frac{(n + n_0) \varphi(n + n_0)}{n \varphi(n)} \to 1$ as $n \to +\infty$, and thus one deduces from \eqref{eq:bound-yn} that, for $n$ large enough, one has $y_n \geq \frac{1}{n \varphi(n)}$, which finally implies \eqref{eq:q'=0-y}, yielding the result.
\end{proof}

\section{Technical lemmas}
\label{AppLemmas}

This appendix provides a series of technical results used in the paper. The first one is useful for establishing existence and uniqueness results for solutions of \eqref{eq:wave} in Section~\ref{SecExistUnique}.

\begin{lemma}
\label{LemmaYpMeasurable}
Let $p \in [1, +\infty]$, $S: \mathbb R \rightrightarrows \mathbb R$, and assume that, for every $g \in \mathsf Y_p$, there exists $h \in \mathsf Y_p$ such that $h(s) \in S(g(s))$ for a.e.\ $s \in [-1, 1]$. Then, for every measurable function $g: [-1, 1] \to \mathbb R$, there exists a measurable function $h: [-1, 1] \to \mathbb R$ such that $h(s) \in S(g(s))$ for a.e.\ $s \in [-1, 1]$.
\end{lemma}

\begin{proof}
Let $g: [-1, 1] \to \mathbb R$ be measurable. Let $A_0 = g^{-1}([-1, 1])$ and, for $n \in \mathbb N^\ast$, let $A_n = g^{-1}([-n-1, -n) \cup (n, n+1])$. Then clearly $A_n$ is measurable for every $n \in \mathbb N$ and the sequence $(A_n)_{n \in \mathbb N}$ is a partition of $[-1, 1]$. For each $n \in \mathbb N$, define $g_n: [-1, 1] \to \mathbb N$ by $g_n = g \chi_{A_n}$, where $\chi_{A_n}$ denotes the characteristic function of $A_n$. Then $g_n$ is measurable and bounded, and hence $g_n \in \mathsf Y_p$. Hence, there exists a sequence $(h_n)_{n \in \mathbb N}$ in $\mathsf Y_p$ such that $h_n(s) \in S(g_n(s))$ for every $n \in \mathbb N$ and a.e.\ $s \in [-1, 1]$. Let $h = \sum_{n=0}^{\infty} h_n \chi_{A_n}$, which is measurable as the countable sum of measurable functions. For every $n \in \mathbb N$ and a.e.\ $s \in A_n$, one has $h(s) = h_n(s) \in S(g_n(s)) = S(g(s))$, and thus $h(s) \in S(g(s))$ for a.e.\ $s \in [-1, 1]$, as required.
\end{proof}

The definition of uniform global asymptotic stability of \eqref{eq:wave} requires \eqref{eq:UGAS} to be satisfied for some $\mathcal{KL}$ function $\beta$. Our next lemma provides sufficient conditions under which a function can be upper bounded by a $\mathcal{KL}$ function, and it is thus useful in several proofs of UGAS results.

\begin{lemma}\label{lem:KL}
Let $f: \mathbb{R}_+\times \mathbb{R}_+\to \mathbb{R}_+$ be a function so that 
\begin{enumerate}
\item\label{ItemKL-A} $f(0,\cdot) \equiv 0$;
\item\label{ItemKL-B} for every $t\geq 0$, $x\mapsto f(x,t)$ is nondecreasing and tends to zero as $x$ tends to $0$;
\item\label{ItemKL-C} for every $x\geq 0$, $t\mapsto f(x,t)$ is nonincreasing and tends to zero as $t$ tends to infinity. 
\end{enumerate}
Then there exists a ${\mathcal{K}}{\mathcal{L}}$ function $\beta$ 
such that 
\begin{equation}\label{eq:constructionKL}
f(x,t)\leq \beta(x,t),\quad \forall t\geq 0.
\end{equation}
\end{lemma}

\begin{proof}
The issue here arises from the fact that $f$ is not necessarily continuous. Define $\beta_0: \mathbb{R}_+ \times \mathbb{R}_+\to \mathbb{R}_+$ as follows: for $x \in \mathbb R_+$ and $t\in [0, 1]$, $\beta_0(x, t) = f(x, 0)$ and, for $x \in \mathbb R_+$, $n \in \mathbb N^\ast$, and $t \in [n, n+1]$,
\[
\beta_0(x, t) = \alpha_t f(x, n-1) + (1-\alpha_t) f(x, n),
\]
where $\alpha_t \in [0,1]$ is uniquely defined by the relation $t=\alpha_t n + (1-\alpha_t) (n+1)$. Then $\beta_0$ verifies the three items \ref{ItemKL-A}, \ref{ItemKL-B}, and \ref{ItemKL-C}, $f(x,t)\leq \beta_0(x,t)$ for every $(x,t)\in \mathbb{R}_+\times \mathbb{R}_+$, and, by construction, $t\mapsto\beta_0(x,t)$ is continuous for every $x\geq 0$. 

We next define $\beta_1: \mathbb{R}_+\times \mathbb{R}_+\to \mathbb{R}_+$ as follows: for every $t \in \mathbb R_+$, $n \in \mathbb Z$, and $x \in [2^{n-1}, 2^n]$, we set $\beta_1(0, t) = 0$ and
\[
\beta_1(x,t) = \alpha_x \beta_0(2^n, t) + (1-\alpha_x) \beta_0(2^{n+1}, t),
\]
where $\alpha_x \in [0,1]$ is uniquely defined by the relation $x=\alpha_x 2^{n-1} + (1-\alpha_x) 2^n$. It is immediate to check that $\beta_1$ is continuous, satisfies \ref{ItemKL-A}, \ref{ItemKL-B}, and \ref{ItemKL-C}, and $f(x, t) \leq \beta_1(x, t)$ for every $(x, t) \in \mathbb R_+ \times \mathbb R_+$. We conclude by taking $\beta(x, t) = \beta_1(x, t) + x e^{-t}$ for $(x, t) \in \mathbb R_+ \times \mathbb R_+$.
\end{proof}

The next result, used in Remark~\ref{Remk-ep} and in the proof of Theorem~\ref{thm:ISS}, is a generalization of Jensen's inequality to $L^p$ norms. Its proof follows closely the classical proof of Jensen's inequality and is provided here for sake of completeness.

\begin{lemma}
\label{LemmaJensen}
Let $(\Omega, \mathfrak A, \mu)$ be a measure space with $0 < \mu(\Omega) < +\infty$. Let $f: \mathbb R_+ \to \mathbb R_+$ be a nondecreasing concave function. Then, for every $p \in [1, +\infty)$ and every real-valued function $g \in L^p(\Omega, \mu)$, one has
\[
\norm{f \circ \abs{g}}_{L^p(\Omega, \mu)} \leq \mu(\Omega)^{1/p} f(\mu(\Omega)^{-1/p} \norm{g}_{L^p(\Omega, \mu)}).
\]
\end{lemma}

\begin{proof}
We assume, with no loss of generality, that $\mu(\Omega) = 1$, since, once the result is proved in this case, one retrieves the general case by applying it to the measure $\nu(\cdot) = \frac{\mu(\cdot)}{\mu(\Omega)}$. 

Fix $p \in [1, +\infty)$ and $g \in L^p(\Omega, \mu)$. Since $f$ is concave, there exist $A, B \in \mathbb R$ such that $f(t) \leq A t + B$ for every $t \in \mathbb R_+$ and $f(\norm{g}_{L^p(\Omega, \mu)}) = A \norm{g}_{L^p(\Omega, \mu)} + B$. In particular, $f \circ \abs{g} \leq A \abs{g} + B$. Moreover, since $f$ is nondecreasing, one has $A \geq 0$ and, since $f(0) \geq 0$, one has $B \geq 0$. Then, using Minkowski inequality, one deduces that
\[
\norm{f \circ \abs{g}}_{L^p(\Omega, \mu)} \leq \norm{A \abs{g} + B}_{L^p(\Omega, \mu)} \leq A \norm{g}_{L^p(\Omega, \mu)} + B = f(\norm{g}_{L^p(\Omega, \mu)}),
\]
as required.
\end{proof}

The next two lemmas provide suitable constructions of functions and are used, respectively, in the proofs of Proposition~\ref{prop:faster-than-exponential-but-not-so-fast} and Theorem~\ref{thm:ISS}.

\begin{lemma}\label{lem:LB}
Let $\varphi: \mathbb R_+ \to \mathbb R_+$ be an increasing function such that $\lim_{x\rightarrow +\infty} \varphi(x) = +\infty$. Then there exists a $\mathcal{C}^2$ function $\psi:\mathbb R_+ \to \mathbb R_+$ satisfying $\lim_{x\rightarrow +\infty} \psi(x) = +\infty$ such that $\psi^\prime > 0$, $\psi^{\prime\prime}\leq 0$, $\psi(0) > 0$, $\psi(x) \leq \varphi(x)$ for $x$ large enough, and such that $0 \leq \frac{x\psi^{\prime}(x)}{\psi(x)} \leq 1$ and $0 \leq -\frac{x\psi^{\prime\prime}(x)}{\psi(x)} \leq 2$ for every $x \in \mathbb R_+$.
\end{lemma}

\begin{proof}
It suffices to prove the result with the additional assumptions that $\varphi$ is also continuous and piecewise affine with positive constant derivative on every interval of the form $(n, n+1)$. Indeed, when this is not the case, one can easily construct a function $\phi$ using a procedure similar to that of Lemma~\ref{lem:KL} in such a way that $\phi(x) \leq \varphi(x)$ for $x$ large enough and $\phi$ satisfies the assumptions of the theorem as well as the previous additional assumptions. We then assume these additional assumptions in the sequel.

For $n\in \mathbb N$, let $\varphi_n = \varphi(n)$ and denote by $\tilde\varphi_n > 0$ the constant value of the derivative of $\varphi$ in the interval $(n, n+1)$. One has that $\varphi_n=\varphi_0+\sum_{k = 0}^{n-1}\tilde{\varphi}_k$, which is an increasing sequence tending to infinity.
Also note that we can assume with no loss of generality that $\varphi_0$ and $\tilde{\varphi}_0$ are both positive.

We first construct a continuous increasing piecewise affine function $F: \mathbb R_+ \to \mathbb R_+$ with $F \leq \psi$, $\displaystyle\lim_{x \to +\infty} F(x)\allowbreak = +\infty$, and such that its derivative $f = F^\prime$ is nonincreasing and constant at every interval of the form $(n, n+1)$. For that purpose, it is sufficient to construct the sequence $(F_n)_{n \in \mathbb N}$ of the values of $F(x)$ at the points $x = n$ and the sequence $(f_n)_{n \in \mathbb N}$ of the constant values of $f$ on the intervals $(n, n+1)$.

We define $(F_n)_{n\in \mathbb N}$ recursively as follows: $F_0=\varphi_0>0$, $f_0=\tilde{\varphi}_0>0$, and, for $n \in \mathbb N^\ast$, $F_n=F_0+\sum_{k = 0}^{n-1} f_k$ and $f_{n}=f_{n-1}$ if $F_n+f_{n-1}\leq \varphi_{n+1}$ or $f_{n}=\tilde{\varphi}_{n}$ otherwise. We easily show by induction that $F_n\leq \varphi_n$ for $n\geq 0$ and $(f_n)_{n \in \mathbb N}$ is a nonincreasing positive sequence. Indeed, $F_0 = \varphi_0$, $f_0 = \tilde\varphi_0$, $F_1 = \varphi_1$, and one has either $F_1 + f_0 \leq \varphi_2$, in which case $f_1 = f_0$, or $F_1 + f_0 > \varphi_2$, in which case $f_0 > \varphi_2 - F_1 = \varphi_2 - \varphi_1 = \tilde\varphi_1$, yielding that $f_1 = \tilde\varphi_1 < f_0$, so that $f_0 \geq f_1 > 0$ in both cases. Assume now that $n \in \mathbb N^\ast$ is such that $F_n\leq \varphi_n$ and $f_{n-1}\geq f_n>0$. If $F_{n}+f_{n-1}\leq \varphi_{n+1}$, then $f_{n}=f_{n-1}$, implying that $F_{n+1}=F_n+f_n\leq  \varphi_{n+1}$. If now $F_n+f_{n-1}> \varphi_{n+1}$, one deduces from that equation that $f_{n-1}>\tilde{\varphi}_n=f_n$, yielding that $F_{n+1}=F_n+f_n\leq \varphi_n + f_n = \varphi_n + \tilde\varphi_n = \varphi_{n+1}$ by the induction hypothesis. If $F_{n+1}+f_{n}>\varphi_{n+2}$ one deduces that $f_n>\tilde{\varphi}_{n+1}=f_{n+1}>0$ and otherwise $f_n=f_{n+1}>0$. This concludes the induction argument.

It remains to prove that $F_n$ tends to infinity as $n$ tends to infinity. Arguing by contradiction yields that both sequences $(F_n)_{n\in \mathbb N}$ and $(f_n)_{n\in \mathbb N}$ are bounded. Hence there exists an integer $n_0$ so that $F_n+f_{n-1}\leq \varphi_{n+1}$ for every $n\geq n_0$, since $(\varphi_n)_{n \in \mathbb N}$ tends to infinity. Therefore, by definition, $(f_n)_{n \geq n_0}$ is constant and equal to some $f > 0$, yielding that $F_n=F_{n_0}+f(n-n_0)$ for $n\geq n_0$, which contradicts the fact that $(F_n)_{n \in \mathbb N}$ is bounded, establishing this the result.

To obtain the required function $\psi$, we define $\psi(x) = F(x)$ for $x \in [0, \frac{1}{2}]$ and we regularize the function $F$ in a neighborhood of each positive integer as follows: for $n \in \mathbb N^\ast$ and $s \in [0, 1]$, we set
\[
\psi\left(n - \tfrac{1}{2} + s\right) = \frac{f_n - f_{n-1}}{2} s^2 + f_{n-1} s + F_n - \frac{f_{n-1}}{2}.
\]
Note that $\psi$ and $\psi^\prime$ coincide with $F$ and $f$, respectively, at all points of the form $n + \frac{1}{2}$ for $n \in \mathbb N$, and that $\psi^\prime$ is positive, continuous, and nonincreasing. In particular, from the latter fact, we get
\[
\psi(x)=\psi(0)+\int_0^x\psi^{\prime}(s)\diff s\geq \psi(0)+x\psi^{\prime}(x), \qquad x\geq 0,
\]
which implies that $0 \leq \frac{x \psi^\prime(x)}{\psi(x)} \leq 1$. Moreover, for every $x\geq \frac{1}{2}$, there exist $n \in \mathbb N^\ast$ and $s \in [0, 1]$ such that $x = n - \frac{1}{2} + s$, and one has
\[
0 \leq - \frac{x \psi^{\prime\prime}(x)}{\psi(x)} = \frac{x (f_{n-1} - f_n)}{\psi(x)} \leq \frac{x f_{n-1}}{\psi(x)} \leq \frac{2(n - \frac{1}{2}) \psi^\prime(n - \frac{1}{2})}{\psi(n - \frac{1}{2})} \leq 2,
\]
which concludes the proof.
\end{proof}

\begin{lemma}
\label{LemmKInfty}
Let $\mu: \mathbb R_+ \to \mathbb R_+$ be an upper semi-continuous nondecreasing function with $\mu(0) = 0$, $0 < \mu(r) < r$ for $r > 0$, and such that $r - \mu(r)$ tends to $+\infty$ as $r \to +\infty$. Then there exists a $\mathcal K_\infty$ function $\varphi$ such that $\id - \varphi$ is nondecreasing and $\mu \leq \id - \varphi$. Moreover, if there exist $a \in (0, 1)$ and $M \geq 0$ such that $\mu(r) \leq a r$ for every $r \geq M$, then $\varphi$ can in addition be chosen convex.
\end{lemma}

\begin{proof}
Let us first consider the function $\varphi_0: \mathbb R_+ \to \mathbb R_+$ defined by
\[
\varphi_0(s) = \inf_{r \geq s} r - \mu(r).
\]
Then $\varphi_0$ is nondecreasing, $\varphi_0(0) = 0$, $0 < \varphi_0(r) < r$ for every $r > 0$, $\mu \leq \id - \varphi_0$, and $\varphi_0(r) \to +\infty$ as $r \to +\infty$.

We construct the required function $\varphi$ by an argument similar to that of Lemma~\ref{lem:LB}. Choose an increasing sequence $(x_n)_{n \in \mathbb Z}$ in $\mathbb R_+$ with $\lim_{n \to -\infty} x_n = 0$ and $\lim_{n \to +\infty} x_n = +\infty$ such that the sequence $(\varphi_0(x_n))_{n \in \mathbb Z}$ is increasing, and, for $n \in \mathbb Z$, let $\delta_n = x_{n+1} - x_n > 0$ and $f_n = \frac{\varphi_0(x_n) - \varphi_0(x_{n-1})}{\delta_n} > 0$. We now define a sequence $(F_n)_{n \in \mathbb Z}$ as follows: for $n \leq 0$, we set
\[
F_n = \sum_{k = -\infty}^{n-1} \min(1, f_k) \delta_k,
\]
and we remark that, for every $n \leq 0$, one has $0 < F_n \leq \sum_{k = -\infty}^{n-1} f_k \delta_k \leq \varphi_0(x_{n-1})$, $F_n > F_{n-1}$, and $\frac{F_n - F_{n-1}}{\delta_{n-1}} = \min(1, f_{n-1}) \leq 1$. For $n > 0$, we define $F_n$ inductively by $F_n = \min(F_{n-1} + \delta_{n-1}, \varphi_0(x_{n-1}))$. Clearly, for every $n > 0$, we have $F_n \leq \varphi_0(x_{n-1})$ by construction, and one easily shows by induction that $F_n > F_{n-1}$ and $\frac{F_n - F_{n-1}}{\delta_{n-1}} \leq 1$.

Define $\varphi: \mathbb R_+ \to \mathbb R_+$ by setting $\varphi(0) = 0$ and $\varphi(x) = \alpha_x F_n + (1 - \alpha_x) F_{n+1}$ for $n \in \mathbb Z$ and $x \in [x_{n}, x_{n+1}]$, where $\alpha_x \in [0, 1]$ is the unique value such that $x = \alpha_x x_{n} + (1 - \alpha_x) x_{n+1}$. Clearly, $\varphi$ is continuous and increasing. For $n \in \mathbb Z$, $\varphi$ is affine in $[x_n, x_{n+1}]$ with (constant) derivative $\frac{F_{n+1} - F_n}{\delta_n} \in (0, 1]$ in $(x_n, x_{n+1})$. In particular, $\varphi^\prime(x) \in (0, 1]$ for almost every $x > 0$, and thus $\id - \varphi$ is nondecreasing. For $n \in \mathbb Z$ and $x \in [x_n, x_{n+1}]$, one has $\varphi(x) \leq F_{n+1} \leq \varphi_0(x_n) \leq \varphi_0(x)$, which implies that $\mu \leq \id - \varphi$.

We are only left to prove that $\varphi(x) \to +\infty$ as $x \to +\infty$ or, equivalently, that $F_n \to +\infty$ as $n \to +\infty$. Reasoning by contradiction yields that, since $(F_n)_{n \in \mathbb Z}$ is increasing, there exists $F_\ast > 0$ such that $F_n \to F_\ast$ as $n \to +\infty$ and $F_n < F_\ast$ for every $n \in \mathbb Z$. Let $n_0 \in \mathbb N$ be such that $\varphi_0(x_n) \geq F_\ast$ for every $n \geq n_0$, which exists since $\varphi_0(r) \to +\infty$ as $r \to +\infty$. Hence, for every $n > n_0$, we have $F_n < \varphi_0(x_{n-1})$, and the inductive definition of $F_n$ implies that $F_n = F_{n-1} + \delta_{n-1}$. Thus $F_n = F_{n_0} + \sum_{k = n_0}^{n-1} \delta_k = F_{n_0} + x_n - x_{n_0}$ and, as $n \to +\infty$, one has $x_n \to +\infty$, implying that $F_n \to +\infty$ and yielding the required contradiction.

Finally, we turn to the second part of the statement, namely that, under the extra assumption of the existence of $a \in (0, 1)$ and $M \geq 0$ such that $\mu(r) \leq a r$ for every $r \geq M$, one may construct $\varphi$ to be convex. With no loss of generality, we assume that $M > 0$. Let $\varphi$ be constructed from $\mu$ as above, $\lambda = \min\left(1, \frac{(1-a) M}{\varphi(M)}\right)$, and define $\varphi_1: \mathbb R_+ \to \mathbb R_+$ by $\varphi_1(x) = \lambda \varphi(x)$ for $x \in [0, M]$ and $\varphi_1(x) = \varphi_1(M) + (1-a)(x - M)$ for $x > M$. Then, by construction, $\varphi_1$ is a $\mathcal K_\infty$ function such that $\id - \varphi_1$ is nondecreasing and $\varphi_1 \leq \id - \mu$. Let $h: \mathbb R_+^\ast \to \mathbb R_+$ be the nondecreasing function given by $h(x) = \essinf_{r \geq x} \varphi_1^\prime(r)$ for every $x > 0$. Since $\varphi_1$ is piecewise affine with finitely many affine pieces on every interval of the form $[x, +\infty)$ with $x > 0$, $\varphi_1^\prime$ takes a finite number of values in each such interval, and thus $h(x) > 0$ for every $x > 0$. Moreover, $h(x) \leq \varphi_1^\prime(x)$ for a.e.\ $x \in \mathbb R_+$ and $h(x) = 1-a$ for $x > M$. We define
\[
\varphi_\ast(x) = \int_0^x h(r) \diff r,
\]
which is clearly a convex $\mathcal K_\infty$ function such that $\id - \varphi_\ast$ is nondecreasing and $\varphi_\ast \leq \varphi_1 \leq \id - \mu$, as required.
\end{proof}

Our final technical result in this appendix is the following lemma, used in the proof of Theorem~\ref{TheoremSloooooooow}.

\begin{lemma}
\label{LemmaSequences}
Let $\phi: \mathbb R_+ \to \mathbb R_+^\ast$ be a decreasing function such that $\lim_{t\rightarrow +\infty} \phi(t) = 0$. Define the sequence $(b_n)_{n \in \mathbb N}$ inductively by $b_0 = \phi(0)$, $b_1 = \max(b_0 - 1, \phi(1))$, and, for $n \geq 2$,
\[
b_n = \max(2 b_{n-1} - b_{n-2}, \phi(n)),
\]
and let $(a_n)_{n \in \mathbb N}$ be given by $a_0 = 1$ and $a_n = b_{n-1} - b_n$ for $n \geq 1$. Then the sequence $(a_n)_{n \in \mathbb N}$ is nonincreasing, the sequence $(b_n)_{n \in \mathbb N}$ is decreasing, and both are sequences of positive numbers converging to $0$.
\end{lemma}

\begin{proof}
Notice first that $b_n > 0$ for every $n \in \mathbb N$ since $b_n \geq \phi(n) > 0$.

To prove that $(b_n)_{n \in \mathbb N}$ is decreasing, we prove that $b_{n} < b_{n-1}$ for every $n \geq 1$ by induction. One has either $b_1 = b_0 - 1 < b_0$ or $b_1 = \phi(1) < \phi(0) = b_0$, and hence $b_1 < b_0$. Now, let $n \geq 2$ be such that $b_{n-1} < b_{n-2}$. If $b_{n} = 2 b_{n-1} - b_{n-2}$, then $b_{n} - b_{n-1} = b_{n-1} - b_{n-2} < 0$, and thus $b_{n} < b_{n-1}$. Otherwise, one has $b_{n} = \phi(n) < \phi(n-1) \leq b_{n-1}$. Hence, in all cases, $b_{n} < b_{n-1}$. Thus, by induction, $(b_n)_{n \in \mathbb N}$ is decreasing.

It now follows that $a_n > 0$ for every $n \in \mathbb N$, since $a_0 = 1$ and, for $n \geq 1$, $a_n = b_{n-1} - b_n$ and $(b_n)_{n \in \mathbb N}$ is decreasing.

Let us now show that $(a_n)_{n \in \mathbb N}$ is nonincreasing. One has $b_1 \geq b_0 - 1$, and thus $a_1 = b_0 - b_1 \leq 1 = a_0$. For $n \geq 2$, one has $b_n \geq 2 b_{n-1} - b_{n-2}$, which implies that $b_n - b_{n-1} \geq b_{n-1} - b_{n-2}$, and hence $a_n \leq a_{n-1}$, as required.

Since $(b_n)_{n \in \mathbb N}$ is decreasing, this sequence admits a limit $b_\ast \in [0, b_0)$. Then 
\[\lim_{n \to +\infty} a_n = \lim_{n \to +\infty} (b_{n-1} - b_n) = b_\ast - b_\ast = 0.
\]

Assume, to obtain a contradiction, that $b_\ast > 0$. Using also the fact that $\phi(t) \to 0$ as $t \to +\infty$, one deduces that there exists $N \geq 2$ such that, for every $n \geq N$, one has $b_n > \frac{b_\ast}{2}$ and $\phi(n) < \frac{b_\ast}{2}$. Hence, one has necessarily $b_n = 2 b_{n-1} - b_{n-2}$ for every $n \geq N$, which implies that $b_{n-1} - b_n = b_{n-2} - b_{n-1}$, and thus $a_n = a_{n-1}$ for every $n \geq N$. Since $a_n \to 0$ as $n \to +\infty$, this implies that $a_n = 0$ for every $n \geq N$, which contradicts the fact that $a_n > 0$ for every $n \in \mathbb N$. This contradiction establishes that $b_\ast = 0$, as required.
\end{proof}

\section*{Acknowledgments}

The authors would like to thank Mateusz Kwaśnicki and Ludovic Rifford for their insightful comments on universally measurable functions.

\bibliographystyle{abbrv}
\bibliography{bibsm}
\end{document}